\documentclass[leqno,11pt]{amsart}
\usepackage{amssymb, amsmath}

\usepackage{color,ulem,textcomp}
\usepackage{nohyperref}
\numberwithin{equation}{section}
\usepackage{enumerate}
\usepackage{graphicx}
\usepackage{cancel}

\theoremstyle{definition}
\theoremstyle{plain}
\newtheorem{thm}{Theorem}[section]
\newtheorem{Prop}[thm]{Proposition}

 \newtheorem{Thm}{Theorem}[section]
 \newtheorem{Rmk}[thm]{Remark}
 \newtheorem{Lem}[thm]{Lemma}
  
\def\D {\mathbb{D}}

\def\bbE {\mathbb {E}}
\def\bbP {\mathbb {P}}

\def\N {\mathbb{N}}
\def\R {\mathbb{R}}
\def\bbR {\mathbb{R}}
\def\T {\mathbb{T}}
\def\bbV {\mathbb {V}}
\def\Z {\mathbb{Z}}

\def\n{\bold{n}}
\def\bk{\bold{K}}
\def\bS{\bold{S}}

\def\cA {\mathcal{A}}
\def\cB {\mathcal{B}}
\def\cC {\mathcal{C}}

\def\cD {\mathcal{D}}

\def\cG {\mathcal{G}}
\def\cL {\mathcal{L}}

\def\cN {\mathcal{N}}

\def\cP {\mathcal{P}}

\def\cZ {\mathcal{Z}}

\def\cR {\mathcal{R}}
\def\cT {\mathcal{T}}

\def\eps {{\varepsilon}}
\def\e {{\varepsilon}}

\def\indc {{\bf 1}}

\def\la {\langle}
\def\ra {\rangle}

\def\d {{\partial}}

\newcommand{\Cov}{\operatorname{Cov}}

\newcommand{\ba}{\begin{aligned}}
\newcommand{\ea}{\end{aligned}}

\newcommand{\be}{\begin{equation}}
\newcommand{\ee}{\end{equation}}

\numberwithin{equation}{section}

\begin{document}


\title[ Long-time correlations for a hard-sphere gas at equilibrium
] {Long-time correlations for a hard-sphere gas  \\at equilibrium
}
\author{Thierry Bodineau, Isabelle Gallagher, Laure Saint-Raymond, Sergio Simonella}


\begin{abstract} 
{\color{black}
It has been known since Lanford~\cite{La75} that the dynamics of a hard sphere gas  is described in the low density limit by the Boltzmann equation, at least for short times. The classical strategy of proof fails for longer times, even close to equilibrium.

In this paper, we introduce a duality method coupled with a pruning argument to prove that the covariance of the fluctuations around equilibrium is governed by the linearized Boltzmann equation globally in time (including in diffusive regimes). This method is much more robust and simple than the one {\color{black}devised} in \cite{BGSR2} which was specific to the 2D case.
}
 
\end{abstract}

\maketitle
\tableofcontents
\section{Introduction}
\setcounter{equation}{0}
The goal of this paper is to study the  dynamical fluctuations  of a hard sphere gas at equilibrium in the low density limit. The equilibrium is described by a Gibbs measure,  which is  a product measure up to the spatial exclusion of the particles, and stationary under the microscopic dynamics. 

A major challenge in statistical physics is  to understand 
the long time behavior of the correlations even in such an equilibrium regime.
Our ultimate goal would be to prove that the fluctuations are described  in the low density limit by the fluctuating Boltzmann equation   on long  kinetic times.
The present paper provides a  first step of this program, by characterizing the evolution of the covariance of the fluctuations on such time scales. 
We are hopeful  that the method introduced in this paper could be extended to study the convergence of the higher moments and therefore to complete the program.



Time correlations are expected to  evolve deterministically as dictated by the linearized Boltzmann equation.
At the mathematical level, such a result can be regarded as a variant of the rigorous validity of the nonlinear Boltzmann equation, which was first obtained for short times in \cite{La75} (see also~\cite{IP89,S2,CIP94,GSRT, PS17, denlinger, GG18,GG20}). In fact the same method as in~\cite{La75}, combined with a low density expansion of the invariant measure, was applied in \cite{BLLS80} to prove the validity of the linearized Boltzmann equation.  The result in \cite{BLLS80} suffered however from the same time restriction of the nonlinear case, in spite of the fact that the solution to the linearized Boltzmann equation is globally well defined.

This limitation was finally overcome in~\cite{BGSR2}, in the case of a two-dimensional gas of hard disks. The method of~\cite{BGSR2} used, in particular, that the canonical partition function is uniformly bounded in two space dimensions. For $d \geq 3$ the limit is however more singular, as the accessible volume in phase space is exponentially small.
The goal of the present  paper is to present a much more robust method, based on a duality argument,  which does not depend on dimension. 
Our analysis is quantitative and the validity holds for arbitrarily large kinetic times, even slowly diverging. Hence a hydrodynamical limit can be also obtained in the same way as  in~\cite{BGSR2}, but we shall not repeat this discussion here.

  \subsection{The hard-sphere model}

The microscopic model consists of  identical hard spheres of unit mass and of diameter~$\eps$.
 \begin{figure}[h] 
\centering
\includegraphics[width=3in]{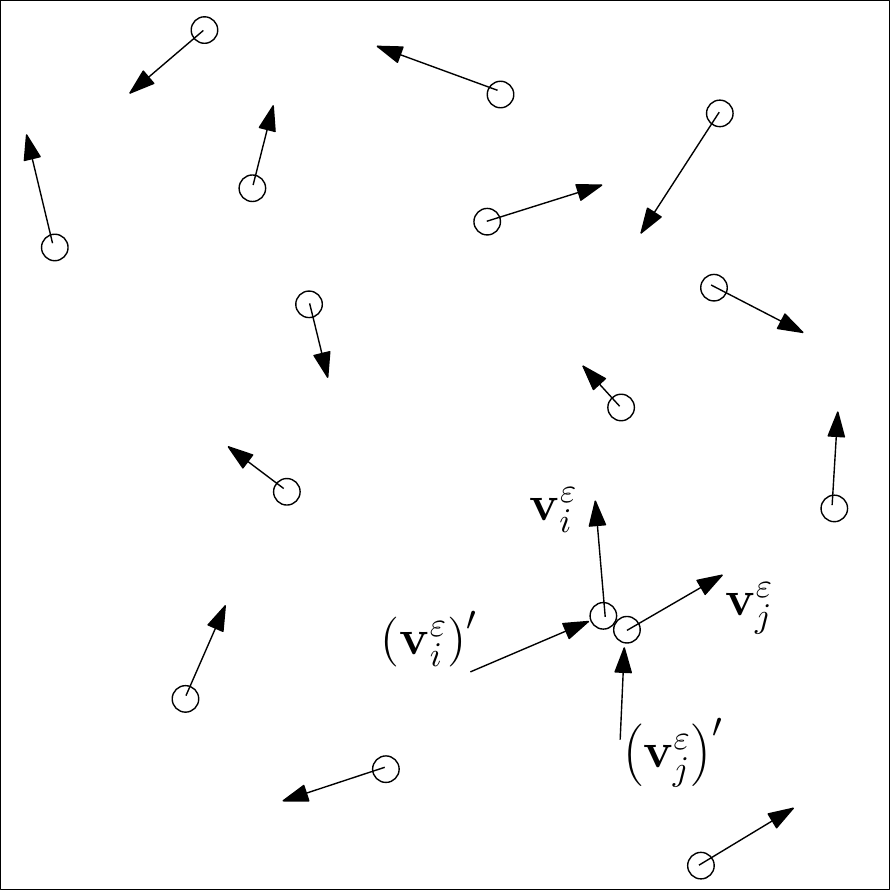} 
\caption{\small 
Transport and collisions in a hard-sphere gas. The square box represents the $d$-dimensional torus.
}
\end{figure}
The  motion of $N$ such hard spheres is governed by a system of ordinary differential equations, which are set in~$\D^N:=( \T ^d\times\R^d)^{N }$ where~$\mathbb T^d$ is  the unit~$d$-dimensional periodic box: writing~${\bf x}^{\e}_i \in  \T ^d$ for the position of the center of the particle labeled~$i$ and~${\bf v}^{\e}_i \in  \R ^d$ for its velocity, one has
\begin{equation}
\label{hardspheres}
{d{\bf x}^{\e}_i\over dt} =  {\bf v}^{\e}_i\,,\quad {d{\bf v}^{\e}_i\over dt} =0 \quad \hbox{ as long as \ } |{\bf x}^{\e}_i(t)-{\bf x}^{\e}_j(t)|>\eps  
\quad \hbox{for \ } 1 \leq i \neq j \leq N
\, ,
\end{equation}
with specular reflection at collisions: 
\begin{equation}
\label{defZ'nij}
\begin{aligned}
\left. \begin{aligned}
 \left({\bf v}^{\e}_i\right)'& := {\bf v}^{\e}_i - \frac1{\eps^2} ({\bf v}^{\e}_i-{\bf v}^{\e}_j)\cdot ({\bf x}^{\e}_i-{\bf x}^{\e}_j) \, ({\bf x}^{\e}_i-{\bf x}^{\e}_j)   \\
\left({\bf v}^{\e}_j\right)'& := {\bf v}^{\e}_j + \frac1{\eps^2} ({\bf v}^{\e}_i-{\bf v}^{\e}_j)\cdot ({\bf x}^{\e}_i-{\bf x}^{\e}_j) \, ({\bf x}^{\e}_i-{\bf x}^{\e}_j)  
\end{aligned}\right\} 
\quad  \hbox{ if } |{\bf x}^{\e}_i(t)-{\bf x}^{\e}_j(t)|=\eps\,.
\end{aligned}
\end{equation}
The sign of the scalar product $({\bf v}^{\e}_i-{\bf v}^{\e}_j)\cdot ({\bf x}^{\e}_i-{\bf x}^{\e}_j)$ identifies post-collisional (+) and pre-collisional ($-$) configurations.
This flow does not cover all possible situations, as multiple collisions are excluded.
But one can show (see \cite{Ale75}) that for almost every admissible initial configuration~$({\bf x}^{\e 0}_i, {\bf v}^{\e 0}_i)_{1\leq i \leq N}$, there are neither multiple collisions,
nor accumulations of collision times, so that the dynamics is globally well defined.

\medskip
We are not interested here in one specific realization of the dynamics,  but rather in a statistical description. This is achieved by introducing a measure at time 0, on the phase space we now specify.
The collections of~$N$ positions and velocities are denoted respectively by~$X_N := (x_1,\dots,x_N) $ in~$ \T^{dN}$ and~$V_N := (v_1,\dots,v_N) $ in~$ \R^{d N}$,   we set~$Z_N:= (X_N,V_N) $ in~$ ( \T ^d\times\R^d)^{N }$, with~$Z_N =(z_1,\dots,z_N)$. Thus a set of~$N$ particles is characterized by~${\mathbf Z}^\eps_N =  ({\mathbf z}^\eps_1,\dots , {\mathbf z}^\eps_N)$ which evolves in the phase space $$\label{D-def}
{\mathcal D}^{\eps}_{N} := \big\{Z_N \in \D^N \, / \,  \forall i \neq j \, ,\quad |x_i - x_j| > \eps \big\} \, .
$$
%
To avoid spurious correlations due to a given total number of particles, we shall consider a grand canonical state. 
At equilibrium the probability density of finding $N$ particles in $Z_N$ is given by
\begin{equation}
\label{eq: initial measure}
\frac{1}{N!}M^{\eps }_{N}(Z_N) 
:= \frac{1}{\cZ^ \eps} \,\frac{\mu_\eps^N}{N!} \, \indc_{
{\mathcal D}^{\eps}_{N}}(Z_N) 
 \, M^{\otimes N}  (V_N)\, , \qquad \hbox{ for } N= 0, 1,2,\dots \end{equation} 
 with~
 \begin{equation}
\label{eq: max}
  M(v) := \frac1{(2\pi)^\frac d2} \exp\big( {-\frac{ |v|^2}2}\big)\,,
\qquad M^{\otimes N}(V_N):= \prod_{i=1}^N M(v_i)\, , 
\end{equation}
and  the partition function is given by
\begin{equation}
\label{eq: partition function}
\cZ^\eps :=  1 + \sum_{N\geq 1}\frac{\mu_\eps^N}{N!}  
\int_{\T^{dN}\times \R^{dN}}  \left( \prod_{i\neq j}\indc_{ |x_i - x_j| > \eps} \right)\left( \prod_{i=1}^N M(v_i)\right)dX_N\,dV_N\,.
\end{equation}
In the following the probability of an event~$X$ with respect to the equilibrium measure~(\ref{eq: initial measure}) will be denoted~${\mathbb P}_\eps(X)$, and~${\mathbb E}_\eps$ will be the expected value. 


In the low density regime, referred to as  the Boltzmann-Grad scaling, the density (average number) of particles is tuned by the parameter $\mu_\e := \e^{-(d-1)}$, ensuring that the mean free path between collisions is of order one \cite{Gr49}. Then, if the particles are distributed according to the grand canonical Gibbs measure~(\ref{eq: initial measure}),
the limit $\e\to0$ provides an ideal gas with velocity distribution $M$.

 \subsection{The linearized Boltzmann equation}
 Out of equilibrium, if the particles are initially identically distributed according to a smooth, sufficiently decaying function~$f^0$,  then in the Boltzmann-Grad limit $\mu_\eps \to \infty$, the average behavior is governed for short times by the Boltzmann equation~\cite{La75}  
$$
\left\{ \begin{aligned}
& \d_t f +v \cdot \nabla _x f = \! \displaystyle\int_{\R^d}\int_{{\mathbb S}^{d-1}}   \! \Big(  f(t,x,w') f(t,x,v') - f(t,x,w) f(t,x,v)\Big) 
\big ((v-w)\cdot \omega\big)_+ \, d\omega \,dw \,  ,\\
&  f(0,x,v) = f^0(x,v)
 \end{aligned}
  \right.
  $$
  where the precollisional velocities $(v',w')$ are defined by the scattering law
\begin{equation}\label{scattlaw}
 v' := v- \big( (v-w) \cdot \omega\big)\,  \omega \,  ,\qquad
 w' :=w+\big((v-w) \cdot \omega\big)\,  \omega  \, .
\end{equation}

At   equilibrium,   $M$ is a stationary solution to the Boltzmann equation, and 
in particular 
the  empirical density defined by
\begin{equation}
\label{eq: empirical}
\pi^{\eps}_t:=\frac{1}{\mu_\e}\sum_{i=1}^\cN \delta_{{\bf z}^\e_i(t)} 
\end{equation} concentrates on~$M$:
for any test function $h: \D\to\R$ and any $\delta>0$, $t \in \R$,
\begin{equation}\label{LLN}
\bbP_\eps \left( \Big|\pi^\eps_t(h) -  
  {\mathbb E}_\eps\big(\pi^\eps_t(h) \big)\Big| > \delta \right) \xrightarrow[\mu_\eps \to \infty]{} 0\;.
\end{equation}
 It is well-known that  the Boltzmann equation dissipates entropy, contrary to  the original particle system~(\ref{hardspheres})-(\ref{defZ'nij}) which is time reversible.  Thus some information is lost  in the Boltzmann-Grad limit, and describing the fluctuations   is a first way to 
capture part of this lost  information.
As in the standard central limit theorem, we expect these fluctuations to be of order~$1/\sqrt{\mu_\eps}$. 
We therefore define  the fluctuation field $\zeta^\eps$ by 
\begin{equation}
\label{eq: fluctuation field}
\zeta^\eps_t \big(  h  \big) :=  { \sqrt{\mu_\eps }} \, 
\Big( \pi^\eps_t(h) -  {\mathbb E}_\eps\big(\pi^\eps_t(h) \big)     \Big) 
 \end{equation}
 for any test function $h$. 
 This process $\zeta^\eps$ has been studied for short times in \cite{BGSS1, BGSS2} and was proved to solve a fluctuating equation.
  Here we focus on the time correlation
 \begin{equation}
 \label{cov-def}
  \Cov_\eps (t,g_0,h) := \bbE_\eps \Big( \zeta^\eps_0 (g_0) \zeta_t^\eps (h) \Big)\,.
  \end{equation}
  Before stating our main result, let us define the linearized Boltzmann operator
  $$
   \cL g := - v \cdot \nabla_x g +\int_{\R^d \times {\mathbb S}^{d-1}} \!\!\! \!\!\!  M(w) \left( (v - w) \cdot \omega \right)_+ \left[g(v') +  g(w') -  g(v) -  g(w)\right]d \omega \, d w
  $$
  which is well-defined in the space~$L^2_M$, denoting for~$1 \leq p < \infty$
$$
L^p_M:=  \Big \{
g : \T^d\times \R^d \rightarrow \R \, , \, \|g\|_{L^p_M}:= \Big( \int_{ \T^d\times \R^d} |g|^p \, Mdxdv\Big)^\frac1p < \infty
\Big\}  \, .
$$
   \begin{Thm} [{\bf Linearized Boltzmann equation}]
\label{thmTCL}
Consider a system of hard spheres at equilibrium  in a~$d$-dimensional periodic box with~$d\geq 3$.  Let~$g_0$ and~$h$ be two functions in~$ L^2_M$.
Then, in the Boltzmann-Grad limit~$\mu_\eps \to \infty$, the covariance of the fluctuation field~$\left(\zeta^\eps_t\right)_{t \geq 0}$ defined by~{\rm(\ref{cov-def})} converges on~$\R^+$ to $\displaystyle \int M g (t) h dxdv$ where $g$ is the solution of the linearized Boltzmann equation~$\partial_t g  =\cL g $, with~$  g_{| t = 0} = g_0$.
\end{Thm}

 \begin{Rmk}
 
It is classical that there is a unique solution to the linearized Boltzmann equation, which is bounded globally in time  in~$L^2_M$.
 
 The same result as Theorem~{\rm\ref{thmTCL}} was proved in dimension~$2$ in ~\cite{BGSR2} with a different, more technical and less robust strategy. The proof presented here could be adapted to the two-dimensional case, at the price of slightly more intricate geometric estimates (see Appendix~{\rm\ref{geometric estimates}}), but we choose not to deal with this case here. 

The limit is stated for any fixed time~$t$, however as will be clear from the proof, one can choose~$t$ diverging slowly with~$\eps$, as~$ o\big((\log| \log \eps|)^{1/4} \big) $ and thus the hydrodynamical limit holds true, see~\cite{BGSR2}.
\end{Rmk}

\begin{Rmk}
Previous work on the (more general) nonequilibrium setting has led to construct the Gaussian limiting fluctuation field for short times by using cumulant expansions \cite{S81,PS17,BGSS1,BGSS2}. For further discussions on the fluctuation theory of the hard sphere gas we refer to these references, as well as to \cite{EC81,S83,S2}.
\end{Rmk}

\medskip
Our strategy starts, as in the classical approach of \cite{La75}, from an expansion over collisions of the BBGKY hierarchy, moving backwards in time from time~$t$ to time~0. These collisions are represented by binary tree graphs. Moreover following \cite{BGSR1,BGSR2}, we sample collisions over small time intervals, and introduce stopping rules in order to avoid super-exponential collision trees. Here we introduce a second stopping rule, to avoid also trajectories with recollisions (a practice known to be efficient in the quantum setting \cite{Er12}). The principal part of the expansion is shown to converge to the expected limit by classical arguments, while the remainder is conveniently controlled by duality in $L^2$-norm, using a global a priori bound on the fluctuations of the invariant measure. In order to implement this strategy, we actually need to control the number of recollisions on the  last small time step before stopping time. This can be done by restricting the initial data to configurations that do not lead to clusters of particles of cardinality $\gamma$, mutually close on a microscopic scale: for $\gamma$ finite but large enough, the cost of this restriction vanishes in the Boltzmann-Grad limit.

The paper is organized as follows.  In Section \ref{sec:strategy} we setup our strategy, introduce several error terms and list the corresponding estimates (see in particular Section  \ref{sec:pseudo} for a simplified description of the method). 
Section \ref{section - L2} contains a general dual bound in $L^2$-norm (based on cluster expansion), which is then used in Sections \ref{section - clustering estimates}, \ref{section - recollisions}, \ref{section - principal} to control the principal part and the error terms. The required geometric estimates on recollision sets are  discussed in the appendix, restricting this part for brevity to $d \geq 3$.

\section{Strategy of the proof} 	\label{sec:strategy}

\setcounter{equation}{0}

\subsection{Reduction to smooth mean free  data}

Let us first prove that, without loss of generality, we can restrict our attention to functions $g_0,h$ satisfying 
\begin{equation} 
\label{eq:mn}
\displaystyle \int M  g_0 dz = \int M  h dz = 0\, .
\end{equation}
We start by noticing that  there is a constant~$ c_\eps$ such that for all  $h\in L^2_M $,
\begin{equation}
\label{eq: mean function}
{\mathbb E}_\eps\big(\pi^\eps_t(h) \big)  = c_\eps \int_{\D}  M (v)h(z) dz\, .
\end{equation}
Indeed
$$
\begin{aligned}
{\mathbb E}_\eps\big(\pi^\eps_t(h) \big) &=    \frac{1}{\cZ^\eps} \sum_{n\geq 1} \frac{\mu_\eps^{n-1}}{(n-1)!} \int _{\cD_n^\eps}M^{\otimes n}(V_n)   h  (z_1) \, dZ_n \\
&=  \int dz_1 M(v_1) h(z_1) \Big(\frac{1}{\cZ^\eps} \sum_{p\geq 0} \frac{\mu_\eps^{p}}{p!}  \int_{\cD_\eps^p}d\bar Z_p
  \prod_{1 \leq i \leq p} \indc_{|x_1-\bar x_i|>\eps} M^{\otimes p}(\bar V_p)  \Big)\\
&= c_\eps \int_{\D}  M (v)h(z) dz \end{aligned}
$$
using the translation  invariance. Expanding the exclusion condition~$ \prod_{1 \leq i \leq p} \indc_{|x_1-\bar x_i|>\eps}-1$ actually leads to~$ c_\eps = 1+O(\eps)$ but this fact will not be used in the following. 

Denoting by $\la \cdot  \ra $ the average with respect to the probability measure $Mdvdx$ and by~$\widehat g := g - \la g \ra $, we get according to \eqref{eq: mean function},
$${\mathbb E}_\eps\big(\pi^\eps_t(\widehat g_0 )\big)
= {\mathbb E}_\eps\big(\pi^\eps_t(\widehat h )\big) = 0\,.$$
Now, shifting $g_0$ and $h$ by their averages boils down to recording the fluctuation of the total number of particles (in the grand canonical ensemble) 
$$\Cov_\e (t, g_0 , h ) =  \Cov_\e (t, \widehat g_0  , \widehat h )  
+\la g_0 \ra   \bbE_\eps \Big( \zeta^\eps (1) \zeta_t^\eps (\widehat h ) \Big) + \la h \ra  \bbE_\eps \Big( \zeta^\eps (1) \zeta_0^\eps (\widehat g_0) \Big)
+ \la h \ra \la g_0 \ra \bbE_\eps \Big( \zeta^\eps (1)^2 \Big)\, ,\\
$$
where we used the time independent field $\zeta^\eps (1) = \frac{1}{\sqrt{\mu_\eps}} \big( \cN - \bbE_\eps (\cN) \big)$. Using the time invariance of the Gibbs measure, the time evolution of  $\Cov_\e$ is unchanged and the result follows from the fact that for all   functions~$h_1$  and~$h_2$ in~$L^2_M$
$$
\int M (\cL \widehat h_1)\widehat  h_2 \, dxdv = \int M( \cL   h_1  ) h_2 \, dxdv\, .
$$
 It will be also useful in the following to work with    functions~$g_0$  and~$h$ with additional smoothness (namely assuming~$g_0$ Lipschitz in space, and both functions to be in~$L^\infty$ and not only~$L^2_M$). For this we notice that we can introduce sequences of smooth, mean free functions~$\left(g^\alpha_0\right)_{\alpha>0}$  and~$\left(h^\alpha\right)_{\alpha>0}$ approximating~$g_0$  and~$h$ in~$L^2_M$ as~$\alpha \to 0$.
By the Cauchy-Schwarz inequality there holds for all mean free functions~$h_1$  and~$h_2$ in~$L^2_M$
$$
 \begin{aligned}
 \Cov_\eps (t,h_1,h_2) & = \bbE_\eps\big (\zeta^\eps_0 (h_1)\zeta^\eps_t (h_2)\big)\\
 & \leq \bbE_\eps \big (\zeta^\eps_0 (h_1)^2 \big)^\frac12\, \bbE_\eps \big (\zeta^\eps_t (h_2)^2 \big)^\frac12\, ,
 \end{aligned}
$$
which is bounded uniformly (for small $\eps$) by virtue of the a priori estimate (see \cite{S2} or Remark \ref{rem:CF} below)
\begin{equation}
\label{eq: moment ordre 2}
 \forall h \in L^2_M\, , \quad \bbE_\eps \big (\zeta^\eps_t (h)^2 \big)^\frac12 \leq C \|h\|_{L^2_M}\,,\quad C >0
\, .
\end{equation}
  In particular $$
\Big | \Cov_\eps (t,g_0,h)- \Cov_\eps (t,g^\alpha_0,h^\alpha) \Big | \longrightarrow 0 \, , \quad \alpha \to 0 \, ,
$$
uniformly in~$\eps$.
In the following, we therefore  assume that~$g_0$ and~$h$ are mean free and smooth.

\subsection{The Duhamel iteration} 
For any  test  function $h : \D\rightarrow \bbR$, let us compute
 $$
  \bbE_\eps\big ( \zeta^\eps_0 (g_0) \zeta_t^\eps (h)\big)=   
 \frac1{\mu_\eps }\bbE_\eps\left ( \Big(  \sum_{i = 1}^\cN   g_0 \big({\bf z}^\eps_i(0)\big)\Big  )
 \Big(  \sum_{i = 1}^\cN   h\big ({\bf z}^\eps_i(t)\big) \Big) \right ) \, .
 $$
Thanks to the exchangeability of the particles, this can be written
\begin{equation}
\label{eq:covCF}
 \bbE_\eps\big ( \zeta^\eps_0 (g_0) \zeta_t^\eps (h)\big) =  \int G^{\eps}_1(t,z)\, h(z)\, dz
\end{equation}
 where  $G^\eps_1$ is     the one-particle ``correlation function" 
$$G^{\eps}_1 (t,z_1)
:= \frac1{\mu_\eps } \,\sum_{p=0}^{\infty} \,\frac{1}{p!} \,\int_{ \D^p} dz_{2}\dots dz_{1+p} \,
W_{1+p}^\eps (t, Z_{1+p}) \, ,
$$
and~$W^\eps_N(t)$ is defined as follows. At time zero we set
\begin{equation}\label{biased initial data}
\frac1{N!}W^{\eps0}_{N }  (Z_N):= \frac{1}{\cZ^ \eps} \,\frac{\mu_\eps^N}{N!} \, \indc_{{\mathcal D}^{\eps}_{N}}(Z_N)
 \,M^{\otimes N}(V_N)  \sum_{i = 1} ^N g_0 (z_i) \, ,
\end{equation}
and~$W^\eps_N(t)$ 
solves   the Liouville equation
\begin{equation}
\label{Liouville}
	\d_t W^{\eps}_N +V_N \cdot \nabla_{X_N} W^{\eps}_N =0  \,\,\,\,\,\,\,\,\, \hbox{on } \,\,\,{\mathcal D}^{\eps}_{N}\, ,
	\end{equation}
	with specular reflection (\ref{defZ'nij}) on the boundary $|x_i - x_j|= \eps$.
		We actually extend $W^\eps_N$ by zero outside $\cD_N^\eps$.

As a consequence, to prove Theorem~\ref{thmTCL} we need to prove that~$G^{\eps}_1(t)$ converges for all times to~$Mg(t)$, where~$g$ solves the linearized Boltzmann equation.

Similarly for any test function $h_n : \D^n\rightarrow \bbR$, one defines
the $n$-particle ``correlation function" 
\begin{align}
\label{eq: densities at t}
G^{\eps}_n (t,Z_n)
:= \frac1{\mu_\eps^n } \,\sum_{p=0}^{\infty} \,\frac{1}{p!} \,\int_{ \D^p} dz_{n+1}\dots dz_{n+p} \,
W_{n+p}^\eps (t, Z_{n+p})
\end{align}
so that
$$
\bbE_\eps \left( \frac{1}{\mu_\eps^n} \Big(  \sum_{i = 1}^\cN  g_0 \big({\bf z}^\eps_i(0)\big)  \Big)
\Big( \sum_{(i_1, \dots i_n )  }  
h_n \big( {\bf z}^{\eps}_{i_1} (t),\dots,{\bf z}^{\eps}_{i_n} (t) \big) 
\Big) \right)
=\int G^{\eps}_n(t,Z_n)\, h_n(Z_n)\, dZ_n\, .
$$
Here and below we use the shortened notation
$$
\sum_{\left(i_1,\dots,i_n\right)} = \sum_{\substack{i_1,\dots,i_n \in \{1,\dots,\cN\} \\ i_j \neq i_k,\, j \neq k}}
\;.$$
 By integration of the Liouville equation for fixed $\eps$, we obtain that 
the one-particle correlation function $G^\eps_1(t,x_1,v_1)$  satisfies 
\begin{equation}
\label{first-correlation}
 \partial_t G^\eps_1 + v_1 \cdot \nabla_{x_1} G^\eps_1 = C_{1,2}^{\eps} G^\eps_{2}  
 \end{equation}
where the collision operator comes from the boundary terms in Green's formula (using the reflection condition to rewrite the gain part in terms of pre-collisional velocities):
\begin{equation} \label{choice hemisphere}
 \begin{aligned}
(C_{1,2} ^\eps G^\eps_2 )(x_1,v_1)
&:=  \int G^\eps_{2} (x_1,v_1', x_1+\e \omega,v_2') \big( (v_2- v_1)\cdot \omega \big)_+ \, d \omega dv_2\\
&\quad -   \int  G^\eps_{2} (x_1,v_1, x_1+\e \omega,v_2) \big( (
v_2 - v_1
)\cdot \omega \big)_- \, d \omega dv_2\, ,
 \end{aligned}
\end{equation}
with as in~(\ref{scattlaw})
$$v_1' = v_1 - (v_1-v_2) \cdot \omega\,  \omega\, , \quad v_2' = v_2 +(v_1-v_2) \cdot \omega\,  \omega \,.$$

Similarly, we have the following evolution equation for the $n$-particle correlation function~:
\begin{equation}
\label{BBGKYGC}
 \partial_t G^\eps_n + V_n \cdot \nabla_{X_n} G^\eps_n = C_{n,n+1}^\eps G^\eps_{n+1} \quad \mbox{on} \quad {\mathcal D}^\eps_{n  }\;,
\end{equation}
 with specular boundary reflection as in \eqref{Liouville}. This is the well-known  BBGKY hierarchy (see~\cite{Ce72}), which is the  elementary brick in the proof of   Lanford's theorem for short times. As~$C^\eps_{1, 2} $ above, $C^\eps_{n, n+1}$ describes
 collisions between one ``fresh'' particle (labelled $n+1$) 
 and one given particle~$i\in \{1,\dots, n\}$. As in~(\ref{choice hemisphere}), this term is decomposed into two parts according  to the hemisphere~$\pm(
v_{n+1} - v_i
)\cdot \omega >0$:
$$ C_{n,n+1} ^\eps G^\eps_{n+1}   := \sum_{i=1}^n C_{n,n+1} ^{\eps,i} G^\eps_{n+1}  $$
with
 $$
 \begin{aligned}
 (C_{n,n+1} ^{\eps,i} G^\eps_{n+1} )(Z_n)&:=  \int G^\eps_{n+1} (Z_n^{\langle i \rangle}, x_i,v_i',x_i+\e \omega ,v_{n+1} ') \big( (v_{n+1} - v_i)\cdot \omega  \big)_+ \, d \omega\, dv_{n+1} \\
&\quad -   \int  G^\eps_{{n+1} } (Z_n, x_i+\e \omega,v_{n+1} ) \big( (
v_{n+1}  - v_i
)\cdot \omega  \big)_- \, d \omega \,dv_{n+1} \, ,
 \end{aligned}
 $$
 where~$(v'_i,v_{n+1} ')$ is recovered from~$(v_i,v_{n+1} )$ through the scattering laws~(\ref{scattlaw}), and with the notation
$$Z_n^{\langle i \rangle} := (z_1,\dots,z_{i-1},z_{i+1},\dots,z_n )\,.
$$ 
Note that performing the change of variables~$\omega \longmapsto -\omega$ in the pre-collisional term gives rise to
$$
 \begin{aligned}
 (C_{n,n+1} ^{\eps,i} G^\eps_{n+1} )(Z_n)&:=  \int 
 \Big(
 G^\eps_{n+1} (Z_n^{\langle i \rangle}, x_i,v_i',x_i+\e \omega ,v_{n+1} ') 
 -
 G^\eps_{{n+1} } (Z_n, x_i-\e \omega,v_{n+1} ) \Big) \\
&\qquad\qquad  \times\big( (
v_{n+1}  - v_i
)\cdot \omega  \big)_+ \, d \omega \,dv_{n+1} \, .
\end{aligned}
$$
Since the equation on~$G^\eps _n$ involves~$G^\eps _{n+1}$, obtaining 
the convergence of~$G^\eps _1$ requires understanding the behaviour of the whole family~$(G^\eps _n)_{n\geq 1}$. A natural first step consists in obtaining uniform bounds.
Denote by~$S^\eps_n$\label{Sn-def}  the group associated with free transport in $\cD^\eps_n$ (with specular reflection on the boundary). 
Iterating Duhamel's formula
$$
G^\eps _n  (t) = S^\eps_n(t) G_{n }^{\eps 0}+ \int_0^t  S^\eps_n(t-t_1)  C_{n,n+1} ^\eps G^\eps_{n+1}(t_1)\, dt_1
$$
 we can express formally the solution~$G^\eps _n(t)$ of the hierarchy~(\ref{BBGKYGC}) as a sum of operators acting on the initial data~:
\begin{equation} \label{eq:seriesexp}
G^\eps _n  (t) =\sum_{m\geq0}    Q^\eps_{n,n+m}(t) G_{n+m}^{\eps 0} \, ,
\end{equation}
where we have defined for $t>0$
$$
\begin{aligned}
Q^\eps_{n,n+m}(t) G_{n+m}^{\eps 0 }  := \int_0^t \int_0^{t_{1}}\dots  \int_0^{t_{m-1}}  S^\eps_n(t-t_{ 1}) C^\eps_{n,n+1}  S^\eps_{n+1}(t_{1}-t_{2}) C^\eps_{n+1,n+2}   \\
\dots  S^\eps_{n+m}(t_{m})    G_{n+m}^{\eps 0} \: dt_{m} \dots d t_{1} 
\end{aligned}$$
and~$Q^\eps_{n,n}(t)G^{\eps0}_{n} := S^\eps _n(t)G_{n}^{\eps0}$,  $Q^\eps_{n,n+m}(0)G^{\eps 0} _{n+m} := \delta_{m,0}G^{\eps0}_{n+m}$.

\medskip
Let us sketch how an a priori bound can be derived from the series expansion~(\ref{eq:seriesexp}).
  We say that~$a$ belongs to the set of (ordered, signed)  {\rm collision trees}  $ \cA^\pm_{n,   m} $ if~$a=(a_{i}, s_{i}) _{  1\leq i \leq   m}$ with labels~$a_{i}\in \{1,\dots, n+ i-1\}$   describing which particle collides with particle~$n+i$,  and with signs $s_{i} \in \{-,+\}$ specifying the collision hemispheres.
  Each elementary integral appearing in the operator~$Q^\eps_{n,n+m}$ thus corresponds to a collision tree in~$\cA^{\pm}_{n,m}$ with $m$ branching points,  involving a simplex in time~($t_{1}> t_{2}> \dots > t_{m}$). If we replace, for simplicity, the cross-section factors  by a bounded function (cutting off high energies), we immediately get that the integrals are bounded, for each fixed collision tree~$a \in \cA^\pm_{n,m}$, by~$\|g_0\|_{L^\infty}C_0^n(C_0 t)^m/ m! $. 
Since~$|\cA^\pm_{n,m}| = 2^m (m+n-1)! / (n-1)! $, summing over all trees gives rise to a bound~$C^{n+m} t^m\|g_0\|_{L^\infty} $. The series expansion is therefore uniformly absolutely convergent only for short times.   In  the  presence of the true cross-section factor, the result remains valid (with a slightly different value of the convergence radius), though the proof requires some extra care \cite{Ki75,La75}.

\subsection{Pseudo-trajectories and duality in~$L^2$} \label{sec:pseudo}
The   strategy described above does not account for possible cancellations between positive and negative terms in the collision integrals: the number of collisions is not under control a priori and this is responsible for the short time of validity of the result. To implement that strategy for long times, it is therefore crucial to take into account those cancellations, which are particularly visible on the invariant measure. The idea is therefore to take advantage of the proximity of the invariant measure to control pathological behaviours. Moreover this has to be done in  an adequate functional setting: the usual Lanford proof~\cite{La75} consists in using~$L^\infty$ norms, but this is problematic as the~$L^\infty$ norm of~$G^\eps_n(t)$ scales as~$\mu_\eps$. Actually (as apparent in~(\ref{eq: moment ordre 2}) for instance),  a weighted~$L^2$ setting is  more appropriate. In this paragraph we explain, in the case of a simplified dynamics without recollisions, how a duality argument enables us to exploit the a priori~$L^2$ bound~(\ref{eq: moment ordre 2}). 

\subsubsection{Pseudo-trajectories}\label{pseudotraj}
(see e.g.\,\cite{BGSS2})
For all   parameters~$(t_{i}, \omega_{i} , v_{n+i} )_{i= 1,\dots, m}$ with~$t_i>t_{i+1}$ and all collision trees $a \in \cA^\pm_{n,   m}$, one   constructs  {\it pseudo-trajectories} on $[0,t]$ 
$$\Psi^{\eps}_{n,m} = \Psi^{\eps}_{n,m} \Big(Z_n, (a_{i} , s_{i}, t_{i}, \omega_{i}, v_{n+i})_{i= 1,\dots, m}\Big)$$
iteratively on $i= 1,2,\dots, m$ as follows
(denoting by $Z^\e_{n+i}(\tau )$ the coordinates of  the  pseudo-particles at time~$\tau \leq t_{i}$, and setting $t_0 = t$):
\begin{itemize}
\item starting from $Z_n$ at time $t $,
\item transporting all existing particles backward on $(t_{i}, t_{i-1})$  (on ${\mathcal D}^\eps_{n +i-1}$ with specular reflection at collisions),
\item adding a new particle labeled $n +i$  at time $t_{i}$, at position $x^\e_{a_{i}} (t_{i}) +\eps s_{i}\omega_{i}$ and with velocity~$v_{n +i}$,
\item   applying the scattering rule (\ref{scattlaw}) if $s_{i}>0$.
\end{itemize}
We discard non admissible parameters for which this procedure is ill-defined; in particular we exclude values of $\omega_{i}$ corresponding to an overlap of particles (two spheres at distance strictly smaller than $\e$) as well as those such that~$\omega_i \cdot \big( v_{n+i} -v^\e_{  a_{i}} (t_{i}^+)\big) \leq0$.
In the following we denote by~$\cG^\eps_m(a, Z_n )$ the set of admissible parameters.

\medskip

With these notations, one gets the following geometric representation of 
the correlation function $G^\eps_n$ : 
$$
\begin{aligned}
 G^\eps _n (t, Z_n) &=
 \sum_{m \geq 0} \sum_{a \in \cA^\pm_{n,   m} }\int_{\cG_{  m}^{\e}(a, Z_n )}    dT_{m}  d\Omega_{m}  dV_{n+1,n+m}\\
&\qquad \times \left(\prod_{i=  1}^{  m}  s_{i}\Big(\big( v_{n+i} -v^\e_{  a_{i}} (t_{i}^+)\big) \cdot \omega_{i} \Big)_+\right)  
G_{n+m}^{\eps 0} \big (Z^\e_{1+m}(0)\big)\, , 
\end{aligned}
$$
where $(T_{m}, \Omega_{m}, V_{n+1,n+m}) := (t_{i}, \omega_{i}, v_{n+i})_{  1\leq i\leq   m}$.

\medskip

 In the following we   concentrate on the case~$n=1$ since as explained above, it is the key to studying the covariance of the fluctuation field: our goal is indeed to study~$\displaystyle  \int dz_1     G^\eps _1 (t, z_1)  h(z_1)$ introduced in \eqref{eq:covCF}. 

\subsubsection{The duality argument in the absence of recollisions}
\label{duality-argument}

In the language of pseudo-trajecto\-ries, a {\it recollision} is a collision between pre-existing particles, namely a collision which does not correspond to the addition of a fresh particle in the backward pseudo-trajectory.

Let us assume momentarily that there is no recollision  in the pseudo-dynamics. Denoting by~$Q^{\eps0}_{1,1+m}$   the restriction of~$Q^\eps_{1,1+m}$ to pseudo-trajectories without recollision, and recalling the series expansion~(\ref{eq:seriesexp}), we therefore  focus in this paragraph on
 $$
 I^0 : =    \sum_{m \geq 0} I^0_m : =   \sum_{m \geq 0} \int dz_1 h(z_1) Q^{\eps0} _{1,1+m} (t)   G^{\eps 0}_{1+m} \,.
 $$
 Let us fix the integer~$m\geq 0$.
Expanding the collision operators leads to
$$
I^0_m
=  \sum_{a \in \cA^\pm_{1,   m} }
\int_{\mathcal P_a}  dz_1 h(z_1)
   dT_{m}  d\Omega_{m}  dV_{ 2,  m+1}  \left(\prod_{i= 1}^{m}  s_{i}\Big(\big( v_{1+ i} -v^\e_{  a_{ i}} (t_{ i}^+)\big) \cdot \omega_{ i} \Big)_+\right)  
G_{1+m}^{\eps 0} \big (Z^\e_{1+m}(0)\big)\;,
$$
where~$\mathcal P_a $ is the subset of~$\D \times ([0,t] \times {\mathbb S}^{d-1} \times \mathbb \R^d)^{m}$ such that for any~$z_1, ( t_i,\omega_i,v_{1+i})_{1 \leq i \leq m} $ in~$\cP_a$, the associate backward pseudo-trajectory satisfies the requirements that as time goes from~$t$ to~$0$,  there are exactly~$m$ collisions according to the collision tree~$a$, and no recollision.
 Recall that a tree $a$ encodes both
 the labels of the colliding particles (namely~$1+i$ and~$a_{i}$)
  and the signs~$s_{i}$   prescribing  at each collision if there is    scattering or not. 
  
  \medskip
  
Given a tree~$a \in \cA^\pm_{1,m }$, consider the change of variables, of range~$\cR_a$:
\begin{equation}
\label{change of variables}
\big(z_1, ( t_i,\omega_i,v_{1+i})_{1 \leq i \leq m}\big) \in \cP_a\longmapsto Z^\e_{1+m}(0) \in \, \cR_a \, .
\end{equation}
It is injective  since the particles evolve by free-transport with no recollision,  and its jacobian  is
$$
 \frac1{\mu_\eps ^{m}} \prod_{i=1}^{m} \Big(\big(
v_{1+i}-v^\e_{a_i}(t_i^+)\big)\cdot \omega_i
\Big)_+ \, .
$$
Denoting by~$z_1^\e(t,Z_{1+m})$ the configuration of   particle~1 at time~$t$ starting from~$Z_{1+m} \in \cR_a$ at time~$0 $, one  can therefore write  
$$
I^0_m 
= \sum_{a \in \cA^\pm_{1,m}}  { \mu_\eps^{m}   }  \int_{\cR_a} dZ_{1+m}
  G^{\eps0}_{1+m} ( Z_{1+m} ) h\big(z_1^\e(t,Z_{1+m})\big )   \prod_{i=1}^{m} s_i\, .
$$
Note that   the restriction to~$\cR_a$ implies  that~$Z_{1+m}$ is configured in such a way that collisions will take place in a prescribed order (first~$1+m$ with~$a_{m}$, then~$m$ with $a_{m-1}$,  etc.)
  and with prescribed successions of scatterings or not.  
Using the exchangeability of the initial distribution, we can symmetrize over the labels of particles   and  set
\begin{equation}
\label{defPhiNk}
\Phi^0_{m+1}(Z_{m+1}):= \frac  { \mu_\eps^{m}   }{(m+1)!}   \sum_{\sigma \in {\mathfrak S}_{m+1}}  \sum_{a \in \cA^\pm_{1,m}}
h\big(z^\eps_{\sigma (1)}(t,{\color{blue}} Z_{\sigma})\big) \indc_{\{ Z_{\sigma} \in \cR_a \}    }\prod_{i=1}^{m} s_i 
\end{equation}
where~${\mathfrak S}_{m+1}$ denotes the permutations of~$\{1,\dots,m+1\}$, $\sigma = (\sigma(1),\cdots,\sigma(m+1))$ and
$$ Z_\sigma = (z_{\sigma(1)},\dots,z_{\sigma(m+1)})\;.$$
By definition, $\Phi^0_{m+1}$ encodes $m$ independent  constraints of size~$1/\mu_\eps$ corresponding to the collisions in the pseudo-dynamics on $[0,t]$, so we expect
$$\int |\Phi^0_{m+1} (Z_{m+1} ) | M^{\otimes (m+1)}(V_{m+1}) dZ_{m+1} \leq  C(Ct)^m$$
for some $C >0$.
 In order to estimate 
\begin{equation}
\label{eq: integrale PhiNk}
I^0_m =  \int_{ } dZ_{m+1}
  G^{\eps0}_{m+1} ( Z_{m+1} )   \Phi^0_{m+1}(Z_{m+1})\, ,
\end{equation}
the key idea is now to  use the  Cauchy-Schwarz inequality
to
  decouple the initial fluctuation from the dynamics on $[0,t]$: indeed, setting 
  \begin{equation} \label{eq:notaver}
\bbE_\eps( \Phi^0_{m+1}) = \bbE_\eps \left( \frac{1}{\mu_\eps^{m+1}}
\Big( \sum_{(i_1, \dots i_{m+1} )  }  
\Phi^0_{m+1} \big( {\bf z}^{\eps}_{i_1},\dots,{\bf z}^{\eps}_{i_{m+1}} \big) 
\Big) \right)
 \end{equation}
and introducing the centered variable 
\begin{equation} \label{eq:hatPhidef}
\hat \Phi^0_{m+1}\left({\mathbf Z}^\eps_{\cN}\right):= \frac{1}{\mu_\eps^{m+1}}\sum_{(i_1, \dots i_{m+1} )  } \Phi_{ m+1}^0 \big( {\bf z}^{\eps}_{i_1},\dots,{\bf z}^{\eps}_{i_{m+1}} \big)  - \bbE_\eps( \Phi^0_{m+1})\;,
\end{equation}
we have  
\begin{equation}
\label{CauchySchwarz}
    \begin{aligned}
 \sum_{m\geq 0}\bbE_\eps\Big( \hat\Phi^0_{m+1}\left({\mathbf Z}^\eps_{\cN}\right)
  \sum _{i=1}^\cN  g_0\big({\bf z}^{\eps}_{i}\big)
   \Big)
 &  :=
 \sum_{m\geq 0}\bbE_\eps\Big( \mu_\eps^{\frac12} \,
 \hat\Phi^0_{m+1}\,
  \zeta^\eps_0(  g_0)
   \Big)
  \\
 & \leq 
\bbE_\eps \Big( (\zeta^\eps_0 (g_0)) ^2 \Big)^{1/2}
   \;   \sum_{m\geq 0} \bbE_\eps \Big(   \, 
  \mu_\eps \Big(
  \hat\Phi^0_{m+1}
  \Big)^2  \; 
  \Big)^{1/2} \,,
  \end{aligned}
 \end{equation}
 and $I^0$ differs from the above quantity by a small error coming from the subtraction of the average
 (which will be shown to be negligible).
 
One important step in this paper will be  the estimate of the last expectation in \eqref{CauchySchwarz}. It requires to expand  the square and to control the cross products using the clustering structure of $ \hat \Phi^0_{m+1}(Z_{m+1} ) \hat \Phi^0_{m+1}(Z'_{m+1} )$.
This will be achieved in Proposition \ref{prop: Quasi-orthogonality estimates}.

\bigskip
At this stage, the duality method does not seem to be much better than the usual method since we expect an estimate of the form
$$|I^0_m|  \leq C (Ct)^m\,,$$
which diverges as $m\to \infty$  despite the fact that it does not even take into account pseudo-dynamics involving recollisions, 
for which the change of variables~{\rm(\ref{change of variables})} is not injective.

 However, since the duality method  somehow  ``decouples" the dynamics and the initial distribution, it will be easier to introduce additional constraints on the dynamics.
 Typically we will require that 
 \begin{itemize}
 \item the total number $m$ of collisions remains under control (much smaller than $|\log \eps|$);
 \item the number of recollisions per particle is bounded, in order to control the defect of injectivity in~{\rm(\ref{change of variables})}.\end{itemize}

\subsection{Sampling}
\label{sec: time sampling} 

As in \cite{BGSR2}, we   introduce a pruning procedure to control the number of terms in the expansion (\ref{eq:seriesexp}) as well as the occurrence of recollisions. We shall rely on  the  geometric interpretation of this expansion: to have a convergent series expansion on a long time~$(0,\theta)$, $\theta \gg 1$, we shall stop the (backward)  iteration whenever one of the two following conditions is fulfilled:
\begin{itemize}
\item super-exponential branching : on the time interval $(\theta- k\tau, \theta - (k-1) \tau)$, with~$\tau\ll 1$ to be tuned, the number~$n_k $ of created particles is larger than~$ 2^k$;
\item recollision : on the time interval $(\theta - (k-1) \tau -r\delta, \theta - (k-1) \tau - (r-1) \delta)$ with~$\delta \ll \tau$ to be tuned, there is at least one recollision.
\end{itemize}
Note that this sampling is more involved than in \cite{BGSR2} since we essentially stop the iteration as soon as there is one recollision in the pseudo-dynamics~: this will be used to apply the duality method. Note also that both conditions (controlled growth and absence of recollision) have to be dealt with simultaneously~: it is indeed hopeless to control the number of recollisions if the number of collisions can be of the order of $|\log \eps|$.

The principal part of the expansion will   correspond to all pseudo-trajectories for
  which the number of created particles on each time step~$(\theta- k\tau, \theta - (k-1) \tau)$, for~$1 \leq k \leq  \theta/\tau$, is smaller than~$ 2^k$, and for which there is no recollision. 
  Recalling that~$Q^{\eps0}_{n,n+m}$ denotes the restriction of~$Q^\eps_{n,n+m}$ to pseudo-trajectories without recollision, and setting~ $K := \theta/\tau$ and~$N_k = 1+\dots + n_k$, we thus define the main part of the expansion as
   \begin{equation} \label{eq:G1epsmain}
    G^{\eps,{\rm main}} _1 (\theta ):=  \sum _{( n_k \leq 2^k)_{k\leq K} } Q^{\eps0} _{1, n_1} (\tau) \dots Q^{\eps0}_{N_{K-1}, N_K} (\tau)   G^{\eps0}_{N_K}  \, .
 \end{equation}
 In order to prove that~$   G^{\eps} _1 -   G^{\eps,{\rm main}} _1 $ is small, we will use the duality argument discussed in Section \ref{duality-argument}, together with an a priori control on the number of recollisions allowed in the dynamics.  We will therefore need to restrict the support of  the initial data, in a way which is harmless in the limit~$\mu_\eps \to \infty$.   Given  an integer~$\gamma\geq 2$, we define a {\it microscopic cluster of size~$\gamma$} as a set~$\cG $ of~$\gamma$ particles in~$\D$ such that
\begin{equation}
\label{eq: Upislon}
  (z,z') \in \cG\times\cG \Longleftrightarrow \exists \, z_{1} = z,z_2,\dots,z_{\ell}= z' \,  \,  \, \mbox{in} \,   \,  \, \cG\, \,     \, \mbox{s.t.} \,  \,|x_{i}-x_{i+1} | \leq 2 \bbV \delta \, , \quad \forall1 \leq i \leq \ell-1\, ,
\end{equation}
for some parameter $\bbV \in \R^+$ which will be tuned later, in Proposition \ref{proposition - high energies},
as a cut-off on the  energies.

We define $\Upsilon_N^\eps$ as  the set
 of configurations~$ Z_N \in {\mathcal D}^\eps_N $ such that for all integers~$1 \leq k \leq \theta/\tau$ and~$1 \leq r \leq  \tau /\delta$,  any cluster present in the configuration~${\mathbf Z}^\eps_N(\theta- (k-1)  \tau - r\delta)$ is of size at most~$ \gamma$.  
The parameters will be chosen so that the set $\Upsilon^\eps_\cN$ is typical under the initial measure. Thus the main contribution to the Duhamel expansion will be given by the restriction to configurations in $\Upsilon^\eps_\cN$. For this reason, we introduce the tilted measures
\begin{equation}
\label{eq: densities at t tilde}
\widetilde W_{N}^{\eps} = W_{N}^{\eps} \,  \indc_{   \Upsilon_{N}^\eps} 
\quad \text{and} \quad 
^c \widetilde W_{N}^{\eps} = W_{N}^{\eps} \, \indc_{^c \Upsilon_{N}^\eps} 
\end{equation}
and the corresponding correlation functions $\left(\widetilde G^{\eps }_{n}\right)_{n \geq 1}$, $\left( ^c \widetilde G^{\eps }_{n}\right)_{n \geq 1}$
defined as in \eqref{eq: densities at t}.
For the measure supported on $\Upsilon_{N}^\eps$,  it is easy to see that if the velocities of the particles at play  at time~$\theta - (k-1) \tau - r \delta$  are under control (the total energy is less than~$\frac12 |\bbV|^2 $, with~$|\bbV \delta | \gg \e$), then on the time interval~$(\theta - (k-1) \tau -r\delta, \theta - (k-1) \tau - (r-1) \delta)$, two particles from different clusters will not be able to recollide. 
 
Now recall that~$K = \theta/\tau$ and~$N_k = 1+\dots + n_k$ (where~$n_k$ is the number  of created particles on the interval~$(\theta- k\tau, \theta - (k-1) \tau)$ in the backward dynamics), and let us set~$R: = \tau /\delta$.  Defining
$$Q^{\rm rec }_{n,n+m} :=Q^\eps_{n,n+m}-Q^{\eps0}_{n,n+m}$$
the restriction of~$Q^\eps_{n,n+m}$ to pseudo-trajectories which have at least one recollision, we can write the following decomposition of~$ \widetilde G^\eps _1$:
 \begin{equation}
\label{main-decomposition}
 \widetilde G^\eps _1 (\theta ) =  
   G^{\eps,{\rm main}} _1 (\theta ) -G^{\eps,{\rm clust}}_{1} (\theta)  +    G^{\eps,{\rm exp}} _1 (\theta )  +   G^{\eps,{\rm vel}} _1 (\theta ) +   G^{\eps,{\rm rec }} _1 (\theta )
\end{equation}
 with  
  $$
G^{\eps,{\rm clust}}_{1} (\theta)  :=  \sum _{( n_k \leq 2^k)_{k\leq K} } Q^{\eps0} _{1, n_1} (\tau) \dots Q^{\eps0}_{N_{K-1}, N_K} (\tau) \,  ^c \tilde G^{\eps0}_{N_K}
 $$ 
 and where
 $$
   G^{\eps,{\rm exp}}_{1} (\theta)  :=  \sum_{k = 1} ^{K } \sum _{( n_j \leq 2^j)_{j\leq k-1 } }\sum_{n_k >2^k} Q^{\eps0} _{1, n_1} (\tau) \dots Q^{\eps0}_{N_{k-1}, N_k} (\tau) \widetilde G^\eps_{N_k} (\theta- k \tau )
 $$
 is the error term coming from super-exponential trees. The term~$   G^{\eps,{\rm vel}} _1 (\theta ) +   G^{\eps,{\rm rec }} _1 (\theta )$ encodes the occurrence of a recollision, depending on the size of the energy at stopping time compared to some value~$ \frac12\mathbb V^2$ to be tuned later. Let us define those two remainder terms: we denote by~$n_k^{\rm{rec}}\geq 0$ the number of particles added on the   time step~$ (\theta - (k-1) \tau -r\delta, \theta - (k-1) \tau - (r-1) \delta) $ (on which by definition there is a recollision), and by~$n_k^0:=n_k - n_k^{\rm{rec}}$ the number of  particles added on the time step~$ (\theta - (k-1) \tau - (r-1) \delta,\theta - (k-1) \tau) $  (on which by definition there is no recollision). We  then define 
 $$\begin{aligned}
 &   G^{\eps,{\rm rec }} _1 (\theta ):=  \sum_{k = 1} ^{K } \sum _{( n_j \leq 2^j)_{j\leq k-1 } }  \sum_{ r = 1}^R  \sum_{n_k\geq 0    }   \sum_{   n_k^0 + n_k^{\rm{rec}}=n_k}Q^{\eps0} _{1, n_1} (\tau) \dots Q^{\eps0}_{N_{k-2}, N_{k-1} } (\tau)      \\
 &\qquad \circ Q^{\eps0} _{N_{k-1},N_{k-1} +n_k^{0}  }( (r-1) \delta)  Q^{\rm rec } _{N_{k-1} +n_k^0 , N_{k-1} +n_k^0 + n_k^{\rm{rec}}   }  ( \delta) \widetilde G^{\eps}_{N_k}  (\theta- (k-1)  \tau - r\delta  ) \indc_{|V_{N_k} | \leq\mathbb V} 
 \end{aligned}
$$
 the error term due to  the occurrence of a recollision, with controled velocities at stopping time, and finally
  $$
 \begin{aligned}
&  G^{\eps,{\rm vel}} _1 (\theta ) 
:=  \sum_{k = 1} ^{K } \sum _{( n_j \leq 2^j)_{j\leq k-1 } }  \sum_{ r = 1}^R  \sum_{n_k\geq 0    }   \sum_{   n_k^0 + n_k^{\rm{rec}}=n_k}Q^{\eps0} _{1, n_1} (\tau) \dots Q^{\eps0}_{N_{k-2}, N_{k-1} } (\tau)  \\
 & \qquad \circ Q^{\eps0} _{N_{k-1},N_{k-1} +n_k^{0}  }( (r-1) \delta)  Q^{\rm rec } _{N_{k-1} +n_k^0 , N_{k-1} +n_k^0 + n_k^{\rm{rec}}   } ( \delta)  \widetilde G^{\eps}_{N_k} (\theta- (k-1)  \tau - r\delta   )  \indc_{|V_{N_k} | > \mathbb V} 
  \end{aligned}$$
  the error coming from large velocities.

\subsection{Analysis   of the remainder terms}
\label{section - analysis}
 
 Recall that our aim is to compute the integral in~\eqref{eq:covCF}.
According to the previous paragraph, recalling the definition
 $$
^c \widetilde G^{\eps }_{1}(\theta)  :=
\sum_{m\geq0}    Q^\eps_{1,m+1}(\theta) \,  ^c \tilde G^{\eps0}_{m+1}
= G^\eps_1(\theta)  - \widetilde G^\eps_1(\theta) \, , 
$$ there holds
\begin{equation}\label{main-decomposition-final}
 G^\eps_1(\theta) =   G^{\eps,{\rm main}} _1 (\theta ) -G^{\eps,{\rm clust}}_{1} (\theta) +    G^{\eps,{\rm exp}} _1 (\theta )  +   G^{\eps,{\rm vel}} _1 (\theta ) +   G^{\eps,{\rm rec }} _1 (\theta )+ {}^c \widetilde G^{\eps }_{1} (\theta)\, .
\end{equation}
The    remainder terms~$ G^{\eps,{\rm clust}}_{1} (\theta) $ and~$ ^c \widetilde G^{\eps }_{1}(\theta) $ consist essentially in measuring the cost of the constraint on~$ \Upsilon_{m+1}^\eps$.
 They are easily shown to be small thanks to the invariant measure: the following proposition is proved in Section~\ref{section - conditioning}.
\begin{Prop}[Cost of restricting the initial data]
\label{proposition - conditioning}
If the parameters $\theta, \delta, \bbV$ satisfy, for some $\gamma \in \N$, 
\begin{equation}
\label{H1}
  \lim_{\mu_\eps \to \infty} 
\theta  \, \mu_\eps^{\gamma+3} \;  \delta^{d \gamma -1} \;\bbV^{d \gamma }    = 0
 \end{equation}
 and
 if  $\theta, \tau$ are chosen such that 
 \begin{equation}
\label{H2}
  \lim_{\mu_\eps \to \infty} {\theta\over \tau  \log|\log \eps|}= 0 \,,
 \end{equation}
 then
\begin{equation}
\label{eq: borne G clust}
\lim_{\mu_\eps \to \infty}  \int dz_1 G^{\eps,{\rm clust}}_1(\theta,z_1) h(z_1)  = \lim_{\mu_\eps \to \infty}  \int dz_1 {}^c \widetilde G^{\eps }_{1}(\theta,z_1) h(z_1) 
=0\,.
\end{equation}
Furthermore, the probability of the complement of  $\Upsilon^\eps_\cN$ is bounded by
\begin{equation}
\label{eq: complementaire Upsilon}
\bbP_\eps \big(  ^c \Upsilon^\eps_\cN \big) \leq  \theta \big( \gamma \; \bbV \big)^{d \gamma} \;\mu_\eps^{\gamma +1} \, \delta^{d \gamma-1}\, 
\end{equation}
and
  there holds
\begin{equation}
 \label{eq: borne reste avec puissances en cN}
 \Big|  \bbE_\eps\left(  \zeta^\eps_0(  g_0)  \indc_{ \Upsilon^\eps_\cN}  \right) \Big| 
\leq    C_\gamma \|  g_0\|_{L^2_M}  \; 
\theta^{\frac12} \;  \bbV ^{\frac{d \gamma}{2}} \;\mu_\eps^{\frac{\gamma +1}{2}} \, \delta^{\frac{d \gamma-1}{2}} \, .
 \end{equation}
\end{Prop}
The control of high energies is also an easy matter thanks to the Gaussian bound on the initial data. The following result is proved in Section~\ref{section - high energies}.  
\begin{Prop}[Cost of high energies]
\label{proposition - high energies}
If  there exists~$a>0$ such that
\begin{equation}
\label{existsa}
\lim_{\mu_\eps \to \infty} {\eps^a  \over \delta} = 0   \end{equation}
and if the parameters $\theta, \tau$ 
satisfy \eqref{H2} 
then choosing $\bbV = |\log \eps|$,
$$
\lim_{\mu_\eps \to \infty}  \int dz_1  \,  G^{\eps,{\rm vel}}_1(\theta,z_1) h(z_1) 
=0
 \, .
$$
\end{Prop}
It remains  to study~$  G^{\eps,{\rm exp}} _1 (\theta )$ and~$   G^{\eps,{\rm rec}}  _1 (\theta )$. For these two terms we   use the a priori $L^2$ control on fluctuations, and thus resort to the duality argument sketched in Paragraph~\ref{duality-argument}.  The following proposition is proved in Section~\ref{section - clustering estimates} thanks to the
quasi-orthogonality estimates of Section~\ref{section - L2} and the
 clustering estimates of Section~\ref{section - clustering estimates}, the extra smallness coming from the assumption that the tree becomes superexponential on a short  time interval of size~$\tau$.
\begin{Prop}[Superexponential trees]
\label{proposition - superexponential trees}
If the parameters $\delta, \bbV, \theta  $ satisfy~{\rm(\ref{H1})} and if
\begin{equation}
\label{H3}
  \lim_{\mu_\eps \to \infty} \;  \theta^3 \tau =0\,,
\end{equation}
then 
$$
\lim_{\mu_\eps \to \infty} \int dz_1  G^{\eps,{\rm exp}} _1 (\theta,z_1) h(z_1)  =0\,.$$
\end{Prop}
The possibility of recollisions makes the analysis of~$G^{\eps,{\rm rec}}  _1 $ more intricate : it is however possible to revisit the arguments of Section~\ref{section - clustering estimates},
to gain smallness thanks to the presence of a recollision on a time interval of size~$\delta$. The following   proposition is proved in Section~\ref{section - recollisions}.

  \begin{Prop}[Recollisions]\label{proposition - recollisions}
If the parameters $\delta, \bbV, \theta, \tau$ satisfy~{\rm(\ref{H1})}, \eqref{H2} and if~{\rm(\ref{existsa})} holds with~$0<a<1$, then 
$$
\lim_{\mu_\eps \to \infty} \int dz_1   G^{\eps,{\rm rec}}  _1 (\theta,z_1) h(z_1) =0 \, .
$$
\end{Prop}

\subsection{End of the proof of Theorem \ref{thmTCL}}\label{tuning parameters}


To conclude the proof of the main theorem, it remains to study the convergence of the principal part, and to check that there exists a possible choice of parameters satisfying all assumptions (\ref{H1})(\ref{H2})(\ref{existsa})(\ref{H3}).

\begin{Prop}[Principal part]
\label{proposition - PP}
Under assumptions~{\rm(\ref{H1})(\ref{H2})(\ref{existsa})(\ref{H3})}, there holds 
$$\lim_{\eps \to 0} \int   G^{\eps,{\rm main}} _1 (\theta,z ) \,h(z)\, dz =  \int M(v)\,g(\theta,z ) \,h(z)\, dz\qquad \forall\theta \in\R^+\;,$$
where $g(\theta)$ is the solution of the linearized Boltzmann equation with initial datum $g_0$~:
\begin{equation}\label{eq:LBE}
\d_t g = \cL g \,.
\end{equation}
\end{Prop}
\noindent The proof of this proposition is the content of Section \ref{section - main part}.

\bigskip
Collecting this together with the decomposition~\eqref{main-decomposition-final} and the propositions of Section~\ref{section - analysis}, Theorem \ref{thmTCL} is proved, provided that the scaling assumptions are compatible. The convergence holds quasi-globally in time, i.e. for any finite $\theta$ and even for very slowly diverging~$\theta = o\big((\log| \log \eps|)^{1/4} \big) $.

We first choose 
$$\delta = \eps ^\eta \,  \hbox{ with }  \,  {d-1\over d} <\eta < 1\,,$$
which ensures that assumption (\ref{existsa}) is satisfied with~$0<a<1$.

Then, we choose $\gamma$ large enough so that
$$\theta  \, \mu_\eps^{\gamma+3} \;\delta  ^{ d \gamma -1} |\log \eps| ^{d \gamma }  = \theta \eps ^{ \gamma (\eta d - (d-1) ) -3(d-1) - \eta }  |\log \eps| ^{d \gamma } \ll 1\,.$$
This will imply that assumption (\ref{H1}) is satisfied, choosing $\bbV =  |\log \eps|$.

It remains to prescribe $\tau$ in order that (\ref{H3}) and  (\ref{H2}) are satisfied. We can take for instance
$$ \tau =  ( \theta^2 \log|\log \eps|) ^{-1/2}\,.$$

\section{Quasi-orthogonality estimates}
\label{section - L2}

To control the remainders associated with super exponential branching $   G^{\eps,{\rm exp}}_{1} (\theta) $ and recollisions $  G^{\eps,{\rm rec}}_{1} (\theta) $, we shall follow the strategy presented in Section~\ref{duality-argument} using a duality argument. More precisely,
in order to use the   $L^2$ estimate on the initial fluctuation field~$\zeta^\eps_0 (g_0)$, we need to establish $L^2$ estimates on the associate test functions~$ \Phi_{N_k}  $, see~(\ref{eq:notaver})-(\ref{CauchySchwarz}).
We prove here a general statement
which will be applied to~ the superexponential case  in   Section~\ref{section - clustering estimates}, and to the case of recollisions in Section~\ref{section - recollisions}.

In the following we denote for~$i<j$
$$
Z_{i, j}:=(z_{i},z_{i+1},\dots z_{ j}) \, .
$$
\begin{Prop} 
\label{prop: Quasi-orthogonality estimates}
Let $ \Phi_N$ be a symmetric function of $N$ variables satisfying 
\begin{eqnarray}
&& \sup_{x_N \in \T^d} \int  | \Phi_N(Z_N) |  M^{\otimes N}(V_{N})
\,dX_{N-1} dV_{N} 
 \leq C^{N }  \rho_0\label{eq: quasi orth 1}\\
&& \sup_{x_{2N-\ell} \in \T^d}\int  | \Phi _N (Z_N) \Phi_N ( Z_\ell, Z_{N+1, 2N - \ell}) |  M^{\otimes {2N - \ell}}(V_{{2N - \ell}})
\,dX_{2N-\ell-1} dV_{{2N - \ell}} \label{eq: quasi orth 2}
 \\
&& \qquad\qquad\qquad
\leq   C^{N }  \; {\mu_\eps^{\ell -1}\over N^\ell }  \;
\rho_\ell \, ,\nonumber \qquad \ell = 1,\dots,N\;,
\end {eqnarray}
for some $C,\rho_0,\rho_\ell >0$.
Define  the centered variable~$\hat \Phi_N$ as in~{\rm(\ref{eq:notaver})}-\eqref{eq:hatPhidef}.
Then there is a constant $\tilde C >0$ such that 
\begin{equation}\label{eq: lem1 moyenne}
 |\bbE_\eps( \Phi_N)|   \leq \tilde C^N \rho_0
\end{equation}
and
\begin{equation}
 \begin{aligned}
\label{L2-PhiN} \bbE_\eps \Big(   \, 
  \mu_\eps 
  \hat\Phi_{N}^2  \; 
  \Big) &   \leq \tilde C^N   \sum_{ \ell = 1}^N  \rho_\ell  + O \left(\tilde C^N \rho_0^2  \eps\right)\, .
   \end{aligned}
\end{equation}
\end{Prop}

Properties (\ref{eq: quasi orth 1}) and (\ref{eq: quasi orth 2})  will come from the fact that $\Phi_{N_k} $ is a sum of elementary functions supported on dynamical clusters, which can be represented by minimally connected graphs with~$N_k$ vertices, where each edge has a cost in $L^1$ of the order of $O(1/\mu_\eps)$. In order to compute the $L^1$ norm of tensor products, we will then extract minimally connected graphs from the union of two such trees, which provides  independent variables of integration. Additional smallness (encoded in the constants $\rho_0, \rho_\ell$) will come from the conditions that  collisions/recollisions are localized in a small time interval,  in Sections~\ref{section - clustering estimates} and~\ref{section - recollisions}.

\begin{proof}We start by computing the expectation
 \begin{equation}
 \begin{aligned}
 \label{eq:mediaprova1}
  \bbE_\eps( \Phi_N)&=  {1\over \mu_\eps^N} 
\bbE_\eps \Big( \sum _{(i_1,\dots, i_{N} )}  \Phi_N \big({\bf z}^\eps_{i_1}, \dots, {\bf z}^\eps_{i_{N} }\big) \Big)
   \\
   &= \frac{1}{\cZ^\eps}
  \sum_{p\geq 0 } \int dZ_{N+p}  \frac{\mu_\eps^{p}}{p!} \indc_{\cD^\eps_{N+p}} (Z_{N+p}) M^{\otimes (N+p) } (V_{N+p} )  \Phi_N(Z_{N})\;.
  \end{aligned}
\end{equation}
This expression will be  estimated by expanding the exclusion condition on $Z_{N+p}  = ( Z_N,  \bar Z_p) $  using classical cluster techniques. We will consider $Z_N$ as a block represented by one vertex, and $(\bar z_i)_{1\leq i \leq p}$ as $p$ separate vertices. We denote by $d(y,y^*)$ the minimum relative distance (in position) between elements $y,y^*\in \{Z_N, \bar z_1, \dots \bar z_{p} \}$. We then have
$$
\begin{aligned}
\indc_{\cD^\eps_{N+p} } (Z_{N+p} ) &=  \indc_{\cD^\eps_N} (Z_N)    \prod_{ y,y^* \in \{Z_N, \bar z_1, \dots \bar z_{p} \}\atop y\neq y^*} \indc_{ d(y,y^*) >\eps} \\
&= \indc_{\cD^\eps_N} (Z_N)   \sum_{\sigma_0 \subset \{1,\dots ,p\}} \indc_{\cD^\eps_{|\sigma_0|} }(\bar Z_{\sigma_0} ) \,  \varphi ( Z_N,\bar Z_{\sigma_0^c} ) 
\end{aligned}
$$
where $\sigma_0$ is a (possibly empty) part of $\{1,\dots, p\} $, $\sigma_0^c$ is its complement, and where  the cumulants $\varphi$ are defined as follows
\begin{equation}
\label{cumulant-def}
\begin{aligned}
 \varphi ( Z_N, \bar Z_{\sigma} ): = \sum_{G \in \cC_{1+|\sigma|}} \prod _{(y,y^*) \in E(G)} ( - \indc_{ d(y,y^*) \leq \eps})\,,
\end{aligned}
\end{equation}
denoting by $\cC_n$ the set of connected graphs with $n$ vertices, and by $E(G)$ the set of edges of such a graph $G$.
By exchangeability of the background particles, we therefore obtain 
 \begin{equation}
 \begin{aligned} \label{eq:calcolomedia}
 \bbE_\eps( \Phi_N)
  &= \frac{1}{\cZ^\eps} \left(  \sum_{p_0 \geq 0} {\mu_\eps ^{p_0} \over p_0!}  \int M^{\otimes p_0} \indc_{\cD^\eps_{p_0}} (\bar Z_{p_0} ) d\bar Z_{p_0}\right)  \\
&\quad \times  \sum_{ p_1 \geq 0} {\mu_\eps ^{p_1} \over p_1 !}  \int M^{\otimes (N + p_1)} \varphi ( Z_N,  \bar Z_{p_1})   \indc_{\cD^\eps_N} (Z_N)   \Phi_N(Z_{N}) dZ_N d\bar Z_{p_1}\\
& =  \sum_{ p_1 \geq 0} {\mu_\eps ^{p_1} \over p_1 !}  \int M^{\otimes (N + p_1)} \varphi ( Z_N,  \bar Z_{p_1})   \indc_{\cD^\eps_N} (Z_N)   \Phi_N(Z_{N}) dZ_N d\bar Z_{p_1}\;,
  \end{aligned}
\end{equation}
where in the last step we used the definition of the grand canonical partition function $\cZ^\eps$.

A powerful tool to sum cluster expansions of  exclusion processes is the  tree inequality due to Penrose (\cite{Pe67}, see also \cite{ gibbspp}) estimating sums over connected graphs in terms of sums over minimally connected graphs. It states that the cumulants defined by (\ref{cumulant-def}) satisfy
 \begin{equation}
 \label{eq:treeineq}
 \left| \varphi ( Z_N,  \bar Z_{p_1})  \right| \leq \sum _{T\in \cT _{1+ p_1} } \prod_{(y,y^*) \in E(T)}  \indc_{ d(y,y^*) \leq \eps}\,,
 \end{equation}
 where $\cT_{1+p_1}$ is the set of minimally connected graphs with $1+p_1$ vertices.
 
 The product of indicator functions in \eqref{eq:treeineq} is a sequence of~$p_1$ constraints, confining the space coordinates to  balls of size $\eps$ centered at the positions   $X_N, \bar x_1, \dots \bar x_{p_1}$.  We rewrite it  as a constraint on  the positions   $x_N,  \bar x_1, \dots \bar x_{p_1}$ (recalling that $X_N$ is considered as a  block, meaning that the relative positions inside it are fixed).
  Integrating the indicator function with respect to $\bar X_{p_1}$ provides a factor
$N^{d_1} \eps^{d p_1 } $ where $d_1$ is the degree of the vertex $X_N$  in $T$. Then, using \eqref{eq: quasi orth 1} to  integrate with respect to $X_{N-1}, V_N$ provides a factor $C^N \rho_0$.

The number of minimally connected graphs with 
specified vertex degrees~$d_1,\dots,d_{1+p_1}$ is given by 
\begin{equation}
\label{eq: tree number}
{(p_1-1)!} / {\displaystyle \prod_{i=1}^{1+p_1} (d_i-1)!} \,. 
\end{equation}
Therefore, combining \eqref{eq:calcolomedia} and (\ref{eq:treeineq}), we conclude that there exists $C'>0$ such that
\begin{equation}
\label{est-2N}
 \begin{aligned}
 |\bbE_\eps( \Phi_N)| \leq  C^N \rho_0  \sum_{p_1 \geq 0}\left[ (C' \eps^d \mu_\eps )^{p_1}  \sum_{d_1,\dots,d_{p_1+1} \geq 1} \frac{N^{d_1} }{\prod_{i=1}^{p_1+1}(d_i-1)!} \right]\;,
\end{aligned}
\end{equation}
from which \eqref{eq: lem1 moyenne} follows by taking $\e$ small enough and using the fact that
$$
  \sum_{d_1,\dots,d_{p_1+1} \geq 1} \frac{N^{d_1} }{\prod_{i=1}^{p_1+1}(d_i-1)!} =
  \sum_{d_1}\frac{N^{d_1} }{(d_1-1)!}  \sum_{d_2}\frac{1}{(d_2-1)!} \dots   \sum_{d_{p_1+1}}\frac{1 }{(d_{p_1+1}-1)!} \leq Ne^N e^{p_1}\, .
$$

\bigskip
 
In order to establish (\ref{L2-PhiN}), we note that
\begin{equation}
\begin{aligned} \label{estimatePhi-EPhi}
 \bbE_\eps \Big(   \, 
  \mu_\eps 
  \hat\Phi_{N}^2  \; 
  \Big) 
  =   \frac1{ \mu_\eps 
^{2N-1}}  \bbE_\eps  \Big(  \sum _{(i_1,\dots, i_{N} )}  \Phi_N \big({\bf z}^{\eps}_{i_1}, \dots, {\bf z}^{\eps}_{i_{N} }) \Big)^2- \mu_\eps 
\big( \bbE_\eps (\Phi_{N}) \big)^2
\end{aligned}
\end{equation}
  and
 first expand the square
$$\begin{aligned}
   &\bbE_\eps\left( \Big(  \sum _{(i_1,\dots, i_{N} )} \Phi_N \big({\bf z}^\eps_{i_1}, \dots, {\bf z}^\eps_{i_{N} }\big) \Big) ^2 \right) \\ & = \bbE_\eps\left( \sum _{(i_1,\dots, i_{N} )}\Phi_N \big({\bf z}^\eps_{i_1}, \dots, {\bf z}^\eps_{i_{N} }\big) \sum _{(i'_1,\dots, i'_{N} )}\Phi_N \big({\bf z}^\eps_{i'_1}, \dots, {\bf z}^\eps_{i'_{N} }\big)\right)\, .
     \end{aligned}
$$
Notice that we have two configurations of (different) particles labelled by $(i_1,\cdots,i_N)$ and $(i'_1,\dots,i'_N)$, with a certain number $\ell $ of particles in common, $\ell = 0,1,\dots,N$. Using the symmetry of the function $\Phi_N$, we can choose $i_1=i'_1,i_2=i'_2,\dots, i_\ell=i'_\ell$ as the common indices and we find that
\begin{equation}
\label{square}
\begin{aligned}
   &\bbE_\eps\left( \Big(  \sum _{(i_1,\dots, i_{N} )}  \Phi_N \big({\bf z}^\eps_{i_1}, \dots, {\bf z}^\eps_{i_{N} }\big) \Big) ^2 \right) \\
    & =  \sum_{\ell = 0} ^N {\binom{N}{\ell} }^2   \ell! \quad \bbE_\eps \Big( \sum _{(i_k )_{k\in \{1,\dots 2N-\ell\} } }
   \Phi_N\big({\bf z}^\eps_{i_1}, \dots, {\bf z}^\eps_{i_N}\big) \Phi_N \big({\bf z}^\eps_{i_1}, \dots, {\bf z}^\eps_{i_\ell}, {\bf z}^\eps_{i_{N+1} },\dots, {\bf z}^\eps_{i_{2N-\ell} }\big)\Big) \;,
  \end{aligned}
 \end{equation}
  where the combinatorial factor   $  {\binom{N}{\ell} }^2$ comes from 
all possible choices for sets~$A$ and~$A'$ in~$ \{1,\dots N\}$,   with $|A|=|A'|=\ell$,   corresponding to the positions of the common indices    in both~$N$-uplets. The factor~$\ell!$ is due to    all possible bijections between $A$ and $A'$, corresponding to the permutations of the repeated indices.

  \bigskip
Next we treat separately the cases $\ell=0$ and $\ell \neq 0$.  
  
\bigskip

\noindent
{\it Step 1. The case when  all indices are different $\ell = 0$. }  Let us compute
\begin{equation}
\label{terme2N}
 \begin{aligned}
& \frac1{\mu_\eps ^{ 2N-1}}
   \bbE_\eps \Big( \sum _{(i_1,\dots, i_{2N} )}  \Phi_N \big({\bf z}^\eps_{i_1}, \dots, {\bf z}^\eps_{i_{N} }\big) \Phi_N \big({\bf z}^\eps_{i_{N+1}}, \dots, {\bf z}^\eps_{i_{2N} } \big) \Big)\\
  & = \frac{\mu_\eps}{\cZ^\eps}
  \sum_{p\geq 0 } \int dZ_{2N+p}  \frac{\mu_\eps^{p}}{p!} \indc_{\cD^\eps_{2N+p}} (Z_{2N+p}) M^{\otimes (2N+p) } (V_{2N+p} )  \Phi_N(Z_{N})\Phi_N(Z_{N+1, 2N} )
  \, .
  \end{aligned}
\end{equation}
 \begin{figure}[h]\centering
\includegraphics[width=4in]{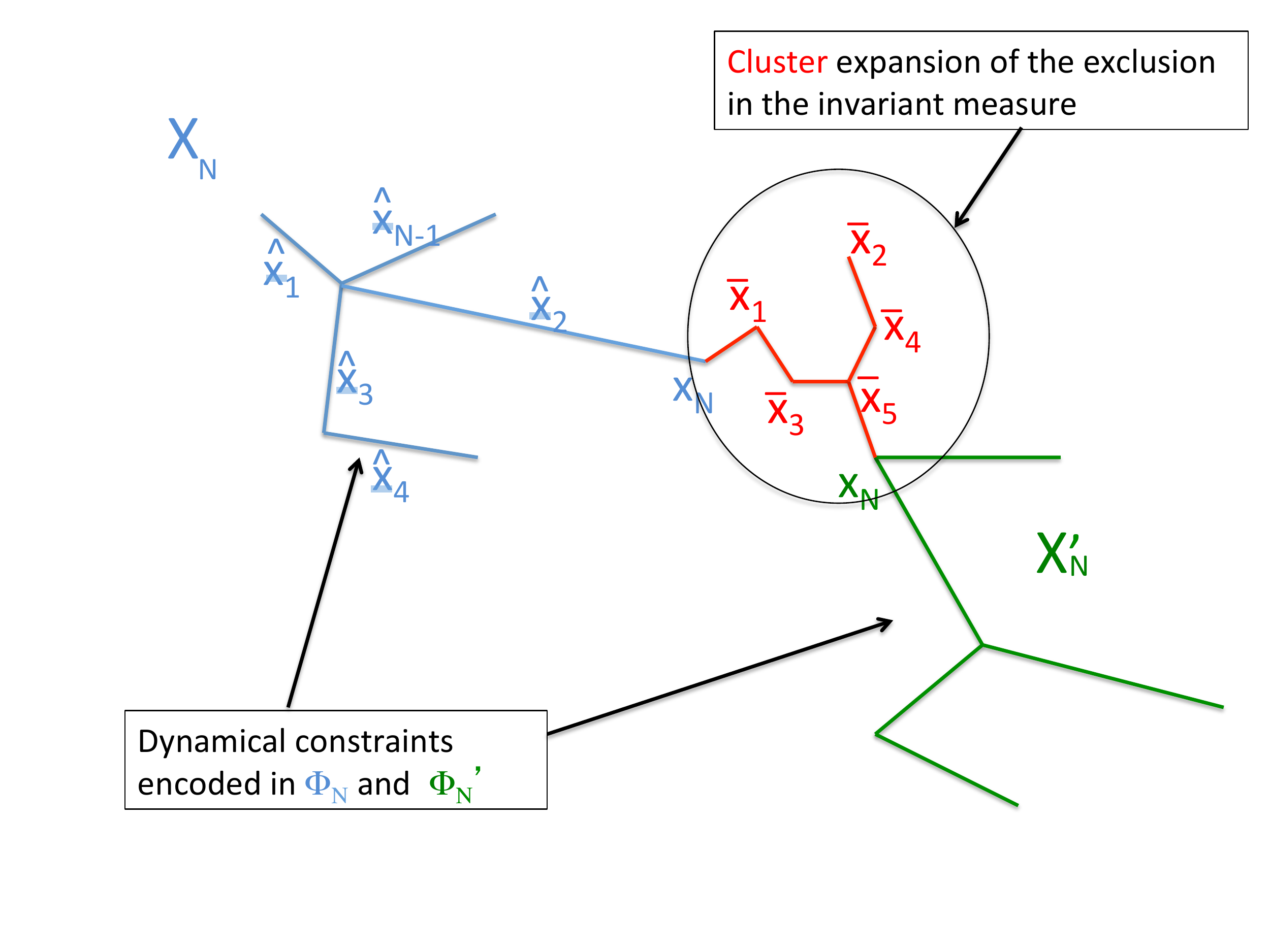} 
\caption{\small 
Cluster expansion of the exclusion and separation of integration variables when $Z_N$ and $Z'_N$ are disjoint.}
 \label{fig-terme2N}
\end{figure}
We can proceed as in the proof of \eqref{eq: lem1 moyenne} by expanding the exclusion condition on $Z_{2N+p}  = ( Z_N, Z'_N, \bar Z_p) $  (see the red part in Figure~\ref{fig-terme2N}) and considering $Z_N$ and $Z'_N$ as blocks represented each by one vertex. We then have
$$
\begin{aligned}
\indc_{\cD^\eps_{2N+p} } (Z_{2N+p} ) 
&= \indc_{\cD^\eps_N} (Z_N)  \indc_{\cD^\eps_N} (Z'_N)   \sum_{\sigma_0 \subset \{1,\dots ,p\}} \indc_{\cD^\eps_{|\sigma_0|} }(\bar Z_{\sigma_0} )  \Big[ \varphi ( Z_N, Z_N', \bar Z_{\sigma_0^c} ) \\
& \qquad \qquad \qquad \qquad \qquad\qquad \qquad \qquad  + \sum_{\sigma\cup \sigma' = \sigma_0^c \atop \sigma \cap \sigma ' = \emptyset} \varphi(Z_N, \bar Z_{\sigma}) \varphi ( Z'_N, \bar Z_{\sigma'}) \Big] 
\end{aligned}
$$
where $\sigma_0, \sigma, \sigma'$ are (possibly empty) parts of $\{1,\dots, p\} $, and where we use \eqref{cumulant-def} and 
$$
\varphi ( Z_N, Z'_N, \bar Z_{\sigma} ) := \sum_{G \in \cC_{2+|\sigma|}} \prod _{(y,y^*) \in E(G)} ( - \indc_{ d(y,y^*) \leq \eps})\;.
$$
By exchangeability of the background particles, we therefore obtain (as in \eqref{eq:calcolomedia})
 \begin{equation}
\label{terme2N-exclusion}
\begin{aligned}
&\frac{\mu_\eps}{\cZ^\eps}\sum_{p\geq 0} {\mu_\eps ^{p}  \over p!}  \int M^{\otimes (2N + p)} \indc_{\cD^\eps_{2N+p}} (Z_N, Z'_N, \bar Z_p ) \Phi_N(Z_{N})\Phi_N(Z'_N) dZ_N dZ'_N  d \bar Z_p\\
  &=\sum_{ p_1 \geq 0} {\mu_\eps ^{p_1+1} \over p_1 !}  \int M^{\otimes (2N + p_1)} \varphi ( Z_N, Z'_N , \bar Z_{p_1})   \indc_{\cD^\eps_N} (Z_N)   \indc_{\cD^\eps_N} (Z'_N) \Phi_N(Z_{N})\Phi_N(Z'_N) dZ_N dZ'_N d\bar Z_{p_1} \\
& \quad  + \mu_\eps \Big( \sum_{p_1 \geq 0} {\mu_\eps ^{p_1} \over p_1 !}  \int M^{\otimes (N + p_1)} \varphi ( Z_N,   \bar Z_{p_1})   \indc_{\cD^\eps_N} (Z_N)  \Phi_N(Z_{N}) dZ_N d\bar Z_{p_1} \Big) ^2 \,.
\end{aligned}
\end{equation}
The last term is equal to  $\mu_\eps \left(\bbE_\eps( \Phi_N)\right)^2$ by  \eqref{eq:calcolomedia}, therefore it cancels out in the computation of~\eqref{estimatePhi-EPhi}.

The second line in \eqref{terme2N-exclusion} is treated as before.
By the tree inequality
$$
 \left| \varphi ( Z_N, Z'_N , \bar Z_{p_1})  \right| \leq \sum _{T\in \cT _{2+ p_1} } \prod_{(y,y^*) \in E(T)}  \indc_{ d(y,y^*) \leq \eps}\;,
 $$
we reduce to~$p_1+1$ constraints confining the space coordinates to  balls of size $\eps$ centered at the positions   $X_N, X'_N, \bar x_1, \dots \bar x_{p_1}$, which we can rewrite as a constraint on  the positions~$x_N, x'_N, \bar x_1, \dots \bar x_{p_1}$ (recalling that $X_N$ and $X'_N$ are considered as blocks, meaning that the relative positions inside each one of these blocks are fixed).
  Integrating the indicator function with respect to $\bar X_{p_1}, x_N, x'_N$ provides a factor
$N^{d_1+d_2} \eps^{d( p_1+1) } $ where $d_1$ and $d_2$ are the degrees of the vertices $X_N$ and $X'_N$ in $T$. Then, using \eqref{eq: quasi orth 1} to  integrate with respect to~$X_{N-1},  X_{N-1}', V_N, V'_N$ provides a factor $\left(C^N \rho_0\right)^2 $.
We conclude that the second line in~\eqref{terme2N-exclusion} is bounded by
\begin{equation}
\label{est-2N'}
 \begin{aligned}
\left(C^N \rho_0\right)^2  \sum_{p_1 \geq 0}\left[ (C' \eps^d \mu_\eps )^{p_1+1}  \sum_{d_1,\dots,d_{p_1+2} \geq 1} \frac{N^{d_1 + d_2} }{\prod_{i=1}^{p_1+2}(d_i-1)!} \right] = O(\tilde C^N \rho_0^2 \eps) 
\end{aligned}
\end{equation}
and it follows that
\begin{equation}
\label{est-2N''}
 \begin{aligned}
 \frac1{\mu_\eps ^{ 2N-1}}
   \bbE_\eps \Big( \sum _{(i_1,\dots, i_{2N} )}  \Phi_N \big({\bf z}^\eps_{i_1}, \dots, {\bf z}^\eps_{i_{N} }\big) \Phi_N \big({\bf z}^\eps_{i_{N+1}}, \dots, {\bf z}^\eps_{i_{2N} } \big) \Big) \\= 
\mu_\eps \left(\bbE_\eps( \Phi_N)\right)^2   + 
O(\tilde C^N \rho_0^2 \eps)  \;.
\end{aligned}
\end{equation}

\bigskip

\noindent
{\it Step 2. The case  when some indices are repeated. } For~$\ell\in [ 1,N]$ given, we   consider
$$
 \begin{aligned}
&\frac1{\mu_\eps ^{ 2N-1}}
   \bbE_\eps \Big( \sum _{(i_k )_{k\in \{1,\dots 2N-\ell\} } }
   \Phi_N\big({\bf z}^\eps_{i_1}, \dots, {\bf z}^\eps_{i_{N} }\big) \Phi_N \big({\bf z}^\eps_{i_1}, \dots, {\bf z}^\eps_{i_\ell}, {\bf z}^\eps_{i_{N+1} },\dots, {\bf z}^\eps_{i_{2N-\ell} }\big)\Big)  \\
   &=  \frac{\mu_\eps^{1-\ell}}{\cZ^\eps} \sum_{p \geq 0}\frac{\mu_\eps^p}{ p!}  \int  dZ_{2N+p-\ell} 
    \indc_{\cD^\eps_{2N+p-\ell}} (Z_{2N+p-\ell})  M^{\otimes (2N+p- \ell )} (V_{2N+p-\ell })   \Phi_N ( Z_N)  \Phi_N (Z'_N) 
  \end{aligned}
$$
denoting  $Z_N = (Z_\ell, Z_{\ell+1, N})$, $Z'_N = (Z_{\ell}, Z_{N+1, 2N-\ell})$ and $\bar Z_p = Z_{2N-\ell+1, 2N-\ell+p}$.

 \begin{figure}[h] 
\centering
\includegraphics[width=4in]{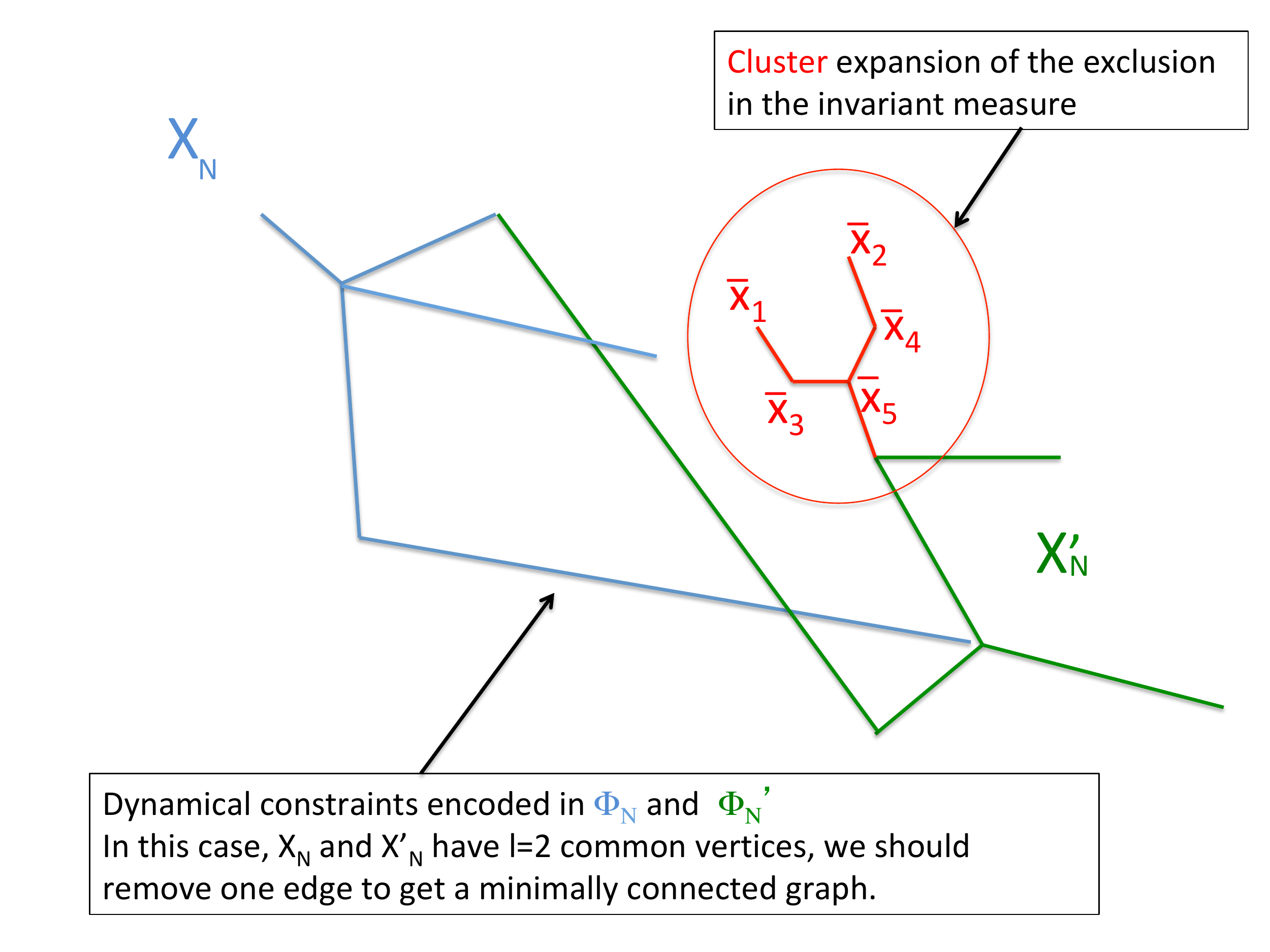} \label{Fig-croise}
\caption{\small 
Cluster expansion of the exclusion and separation of integration variables when $Z_N$ and $Z'_N$ have $\ell$ common elements.}
\end{figure}

 This expression is of the same form as~\eqref{eq:mediaprova1}, but now $\Phi_N(Z_N)$ is replaced by $\Phi_N ( Z_N)  \Phi_N (Z'_N) $ which is a function of $2N-\ell$ particle variables. It can be therefore estimated in exactly the same way (by considering $Z_{2N-\ell}$ as one block since the dynamical constraints will  provide a cluster structure on $Z_{2N-\ell}$: see the red part in Figure \ref{Fig-croise}). The role of the cluster estimate~\eqref{eq: quasi orth 1} is now played by \eqref{eq: quasi orth 2} and this leads to (see \eqref{est-2N})
 \begin{equation}
\label{est-croise}
 \begin{aligned}
&\frac1{\mu_\eps ^{ 2N-1}}
   \bbE_\eps \Big( \sum _{(i_k )_{k\in \{1,\dots 2N-\ell\} } }
  \Phi_N\big({\bf z}^\eps_{i_1}, \dots, {\bf z}^\eps_{i_{N} }\big) \Phi_N \big({\bf z}^\eps_{i_1}, \dots, {\bf z}^\eps_{i_\ell}, {\bf z}^\eps_{i_{N+1} },\dots, {\bf z}^\eps_{i_{2N-\ell} }\big)\Big) \\
&
\qquad \leq  C^N {\rho_\ell \over N^\ell} \sum_{p_1 \geq 0}\left[ (C' \eps^d \mu_\eps )^{p_1}  \sum_{d_1,\dots,d_{p_1+1} \geq 1} \frac{(2N-\ell) ^{d_1 } }{\prod_{i=1}^{p_1+1}(d_i-1)!} \right]\leq \tilde C^{N} \rho_\ell  N^{-\ell} \;.
\end{aligned}
\end{equation}
Combining \eqref{estimatePhi-EPhi}, (\ref{square}), (\ref{est-2N''}) and (\ref{est-croise})
we conclude that
\begin{equation}
 \begin{aligned}
\bbE_\eps \Big(   \, 
  \mu_\eps 
  \hat\Phi_{N}^2  \; 
  \Big) \leq \mu_\eps \left(\bbE_\eps( \Phi_N)\right)^2   + 
O(\tilde C^N \rho_0^2 \eps)
+  \sum_{\ell = 0} ^N {\binom{N}{\ell} }^2   \ell!\, \tilde C^{N} \rho_\ell  N^{-\ell} - \mu_\eps \left(\bbE_\eps( \Phi_N)\right)^2
\end{aligned}
\end{equation}
and, remarking that $\ell ! \leq N^\ell $, this leads to \eqref{L2-PhiN} by enlarging the constant $\tilde C$.
\end{proof}
 \begin{Rmk} \label{rem:CF}
For $N=1$ and $\Phi_1 = h \in L^2_M$, one has $\bbE_\eps \Big(   \, 
  \mu_\eps 
  \hat\Phi_{1}^2  
  \Big) =  \bbE_\eps \big (\zeta^\eps (h)^2 \big)$. A simple corollary of the above proof
  leads then to \eqref{eq: moment ordre 2}.
\end{Rmk}


\section{Clustering estimates}
\label{section - clustering estimates}

\setcounter{equation}{0}

In this section we will prove Proposition~\ref{proposition - superexponential trees}. 
We consider
\begin{equation}
\label{eq:Gexp}
 \int  G^{\eps,{\rm exp}} _1 (\theta )\, h(z)\, dz= \sum_{k=1}^K \sum _{( n_j \leq 2^j)_{j\leq k-1} }
 \sum _{  n_k \geq 2^k } \int dz\, h(z)\, Q^{\eps0} _{1, n_1} (\tau) \dots Q^{\eps0}_{N_{K-1}, N_K} (\tau) \,\widetilde G^{\eps}_{N_K}  (\theta - k\tau)  \, .
\end{equation}
Each term of the sum will be estimated by using  Proposition \ref{prop: Quasi-orthogonality estimates}. 
With the notation $t_{{\rm stop}}:=\theta -k\tau$, we set
\begin{equation}
\label{eq: def I nk}
I_{\n_k}  := \int  h(z_1)
Q^{\eps0} _{1, n_1} (\tau) \dots Q^{\eps0}_{N_{k-1}, N_k} (\tau) \widetilde G^\eps_{N_k} (t_{{\rm stop}})dz_1
\end{equation}
where~$1 \leq k \leq K$ is fixed, as well as the set~$\n_k = (n_j)_{1 \leq j \leq k}$  of integers.  
Given a collision tree $a \in \cA^\pm_{1,N_k-1}$, we will use, as explained in \eqref{change of variables},  the injectivity of the  change of variables
\begin{equation}
\label{change of variables 2}
\big(z_1, ( t_i,\omega_i,v_{1+i})_{1 \leq i \leq N_k - 1}\big)  \longmapsto Z^\e_{N_k}(0) \in \, \cR_{a, \n_k}  \, ,
\end{equation}
where the configurations in $\cR_{a,{\mathbf n}_k} $ have to be  compatible with 
 pseudo-trajectories satisfying the following constraints : 
 \begin{itemize}
 \item[(i)] there are $n_j$ particles added on the time intervals $(\theta- j \tau, \theta- (j-1) \tau)$
 for $j \leq k$;
 \item[(ii)]  the addition of new particles is prescribed by the collision tree $a$;
 \item[(iii)] the pseudo-trajectory involves no recollision on $(t_{{\rm stop}}, \theta)$.
 \end{itemize}
We can thus write
$$
I_{\n_k} = \int    \Phi_{N_k} (Z_{N_k}) \, \widetilde G^{\eps}_{N_k} ( t_{{\rm stop}} , Z_{N_k} )  \,dZ_{N_k} \, ,
$$
with 
\begin{equation}
\label{defPhiN exp}
\Phi_{N_k}  (Z_{N_k}  ):= \frac  { \mu_\eps^{N_k-1}   }{ N_k !}   \sum_{\sigma \in {\mathfrak S}_{N_k} }  \sum_{a \in \cA^\pm_{1,N_k-1}}
h\big(z^\e_{\sigma (1)}(\theta,Z_\sigma)\big) \indc_{\{ Z_\sigma \in \cR_{a,{\mathbf n}_k}  \}    }\prod_{i=1}^{N_k - 1}  s_i \, .
\end{equation}
 Using same the notation as \eqref{eq:hatPhidef}, we set
\begin{equation} 
\label{eq:hatPhidef 2}
\hat \Phi_{N}\left({\mathbf Z}^\eps_{N}\right):= \frac{1}{\mu_\eps^{N}}\sum_{(i_1, \dots i_{N} )  } \Phi_N \big( {\bf z}^{\eps}_{i_1},\dots,{\bf z}^{\eps}_{i_{N}} \big)  - \bbE_\eps( \Phi_N )\;,
\end{equation}
so that $I_{\n_k}$ becomes
\begin{align}
\label{eq:hatPhidef 2 bis}
I_{\n_k}   = \bbE_\eps\Big( \mu_\eps^{1/2} \; \hat\Phi_{N_k} \big({\mathbf Z}^\eps_{{N_k}} (t_{{\rm stop}}) \big) \;
  \zeta^\eps_0(  g_0) \; \indc_{  \Upsilon_\cN^\eps}  \Big)
 +  \mu_\eps^{1/2} \;  \bbE_\eps\left(\Phi_{N_k}\right) 
 \bbE_\eps\left( \zeta^\eps_0(  g_0)  \indc_{  \Upsilon_\cN^\eps}  \right)\, ,
\end{align}
where the indicator function on $\Upsilon_\cN^\eps$ stands for the restriction on the cluster sizes \eqref{eq: densities at t tilde}.
Applying the Cauchy-Schwarz inequality, as in \eqref{CauchySchwarz}, leads to the following upper bound
\begin{equation}\begin{aligned}
\label{eq: inegalite I exp}
|I_{\n_k} | & \leq 
\bbE_\eps \Big( (\zeta^\eps_0 (g_0)) ^2 \Big)^{1/2}
   \;    \bbE_\eps \Big(   \, 
  \mu_\eps \Big(  \hat\Phi_{N_k} \big( {\mathbf Z}^\eps_{{N_k}} (t_{{\rm stop}})  \big) \Big)^2  \; 
  \Big)^{1/2} \\
& \quad + \mu_\eps^{1/2} \;  |\bbE_\eps\left(\Phi_{N_k}\right) 
 \bbE_\eps\left( \zeta^\eps_0(  g_0)  \indc_{  \Upsilon_\cN^\eps}  \right) |
\end{aligned}
\end{equation}
which can be estimated by Proposition \ref{prop: Quasi-orthogonality estimates}.
To do this, we are going to check, in Lemmas~\ref{lem: hyp quasi orth 1} and \ref{lem: hyp quasi orth 2} stated below, that  $\Phi_{N_k} $ satisfies the assumptions \eqref{eq: quasi orth 1} and \eqref{eq: quasi orth 2} of Proposition~\ref{prop: Quasi-orthogonality estimates}.
The last term involving the expectation will be negligible thanks to estimate \eqref{eq: borne reste avec puissances en cN} of Proposition~\ref{proposition - conditioning} and the   tuning of the parameters performed in Section~\ref{tuning parameters}.
\begin{Lem}
\label{lem: hyp quasi orth 1}
There exists $C>0$ such that 
\begin{equation}\label{eq:lem1}
\sup_{x_{N_k}  \in \T^d}\int \big | \Phi _{N_k} (Z_{N_k}) \big|  M^{\otimes N_k}(V_{N_k})
\, dX_{N_k-1} dV_{N_k} 
 \leq C^{N_k }  \|h\|_{L^\infty(\D)}\theta^{N_{k-1}-1} \tau^{n_k}\;.
\end{equation}
\end{Lem}
\begin{Lem}
\label{lem: hyp quasi orth 2}
There exists $C>0$ such that, for any $\ell = 1,\dots, N_k $,
\begin{equation}\label{eq:lem2}
\begin{aligned}
\sup_{x_{2N_k - \ell}  \in \T^d}\int \big | \Phi _{N_k} (Z_{N_k}) \Phi _{N_k} (Z_{\ell} , Z_{N_k +1, 2N_k - \ell})   \big|  M^{\otimes (2N_{k} - \ell)}(V_{{2N_k - \ell}})
\, dX_{{2N_k - \ell-1}} dV_{{2N_k - \ell }} \\
 \leq C^{{N_k } }
  \mu_\eps^{\ell -1} N_k^{-\ell}  
\|h\|_{L^\infty(\D)}^2 \; 
\theta^{2N_{k}-\ell-1- n_k} \tau^{n_k} \;.
\end{aligned}
\end{equation}
\end{Lem}
Assuming those lemmas are true,  let us   complete the estimate of $I_{\n_k}$.
Starting from~\eqref{eq: inegalite I exp}, it is enough to apply Proposition \ref{prop: Quasi-orthogonality estimates}, and~(\ref{eq: borne reste avec puissances en cN}) of Proposition~\ref{proposition - conditioning}.
We finally get 
\begin{equation}
\begin{aligned}
\label{eq: borne finale I}
|I_{\n_k}   |
& \leq 
  C^{{N_k } } \|g_0\|_{L^2_M} \|h\|_{L^\infty(\D)}  
\Big( \big(  \sum_{\ell =1}^{N_k} \theta^{2N_{k}-\ell-1- n_k} \tau^{n_{k}}  +  \e   \theta^{2(N_{k-1}-1)} \tau^{2n_k}\big)^{1/2}
 \\
 & \qquad\qquad\qquad\qquad\qquad\qquad 
 + \theta^{N_{k-1} - \frac12}  \tau^{n_k} \;  \bbV ^{\frac{d \gamma}{2}} \;\mu_\eps^{\frac{\gamma}{2} + 1 } \, \delta^{\frac{d \gamma-1}{2}}\Big) \\
& \leq 
 \|g_0\|_{L^2_M}\|h\|_{L^\infty(\D)} \; ( C \theta) ^{{N_{k-1} } + n_k/2} \; \tau^{n_{k}/2} 
\end{aligned}
\end{equation}
thanks to~(\ref{H1}) and~$\delta \leq \tau \ll 1$.

To complete Proposition~\ref{proposition - superexponential trees}, we will show that the contribution of the superexponential trees is negligible. The superexponential trees are such that 
 $N_{k-1} \leq 2^{k} \leq  n_k$, this leads to
\begin{align}
\label{eq: borne finale I exp}
|I_{\n_k}   | \leq \|h\|_{L^\infty(\D)}  \|g_0\|_{L^2_M}\; ( C \theta) ^{N_{k-1}+n_k/2}  
\; \tau^{n_k/2} \leq \|h\|_{L^\infty(\D)} \;  ( C \theta^3 \tau) ^{n_k/2} .
\end{align}
The parameters $\theta, \tau$ satisfy \eqref{H3} so   we can sum over $(n_j)_{j\leq k}$ and the series   is controlled by
\begin{equation} \label{eq:superepxconcl}
\Big | \int dz_1    G^{\eps,{\rm exp}}_{1} (\theta,z_1)
h(z_1)\Big |
\leq \|h\|_{L^\infty(\D)} \|g_0\|_{L^2_M}  \sum_{k = 1} ^K 2^{k^2} ( C \theta^3 \tau) ^{2^{k-1}}\,. 
\end{equation}
The proof of Proposition~\ref{proposition - superexponential trees} is complete.

\bigskip

Before proving Lemmas \ref{lem: hyp quasi orth 1} and \ref{lem: hyp quasi orth 2},  let us   introduce some notation. For any positive integer~$N$, we shall denote  as previously by $\cT_N$ the set of trees
(minimally connected graphs)  with $N$ vertices.
We further denote by~$\cT^\prec_N$ the set of ordered trees. A tree $T_\prec \in \cT^\prec_N$  is represented by an ordered sequence of edges $\left(q_i,\bar q_i\right)_{1 \leq i \leq N-1}$.

\begin{proof}[Proof of Lemma {\rm\ref{lem: hyp quasi orth 1}}]

 For each configuration $Z_{N_k}$, there exist at most $4^{N_k -1}$  different $(\sigma, a)$ such that $Z_{ \sigma }\in \cR_{a,{\mathbf n}_k} $. Indeed at each collision in the forward pseudo-trajectory, 
the particle which  disappears has to be chosen, as well 
as a possible   scattering. To fix these discrepancies, we introduce two sets of signs~$\bar s_i$ and~$s_i$ which determine respectively which particle should be removed (say~$\bar s _i = +$ if the particle with largest index remains,~$\bar s _i = -$ if it disappears) and whether there is scattering ($  s _i = +$) or not ($  s _i = - $).
Note that
the signs~$(s_i)_{1 \leq i \leq N_k -1}$ are encoded  in
 the tree $a$ while~$(\bar s_i)_{1 \leq i \leq N_k -1}$ are known if~$\sigma$ is given. If we prescribe the set~$\bS _{N_k- 1} :=( s_i, \bar s_i) _{1\leq i \leq N_k- 1}$, then the mapping
$$
\big(a, \sigma, z_1, ( t_i,\omega_i,v_{1+i})_{1 \leq i \leq N_k - 1}\big) \longmapsto Z^\e_{\sigma }(t_{{\rm stop}})  
$$
restricted to pseudo-trajectories compatible with $\bS _{N_k- 1}$, is  injective.
This leads to
\begin{equation}
\label{estimatePhiexp}
\big| \Phi _{N_k} (Z_{N_k} ) \big| \leq  \| h \|_\infty 
{\mu_\eps^{N_k-1}\over N_k!}  \sum_{\bS _{N_k- 1}} 
 \indc_{\{ Z_{N_k} \in \cR_{\bS _{N_k- 1}}  \}    } \, ,
\end{equation}
where $\cR_{\bS _{N_k- 1}} $ is the set of configurations such that the forward flow compatible with $\bS _{N_k- 1}$ exists, and with the  constraints respecting the sampling (we drop the dependence of the sets on $\n_k$, not to overburden notation).

\medskip

We are now going to evaluate the cost of the constraint $Z_{N_k} \in \cR_{\bS _{N_k- 1}} $ for a given~$\bS _{N_k- 1}$.
For this it is convenient to record the collisions in the forward dynamics in an ordered tree~$T_\prec=  (q_i, \bar q_i)_{1\leq i \leq N_k-1}$: the first collision, in the forward flow, is between particles~$q_1$ and~$\bar q_1$ at time~$\tau_1 \in (t_{{\rm stop}},\theta)$, and the last collision is between~$q_{N_k-1}$ and~$\bar q_{N_k-1}$ at time~$\tau_{N_k-1} \in (\tau_{N_k-2},\theta)$. Notice that compared with the definition of (backward)  pseudo-trajectories,  since we follow the trajectories forward in time we choose an increasing order in the collision times (namely~$\tau_{i}= t_{N_k-i}$). 
This leads to
\begin{equation}
\label{estimatePhi}
\big| \Phi_{N_k} (Z_{N_k} ) \big| \leq  \| h \|_\infty 
{\mu_\eps^{N_k-1}\over N_k!}  \sum_{  \bS_{N_k- 1}} \sum_{T_\prec \in \cT^\prec_{N_k}} 
 \indc_{\{ Z_{N_k} \in \cR_{T_\prec,  \bS_{N_k- 1}}  \}    } \, ,
\end{equation}
where $\cR_{T_\prec,   \bS_{N_k- 1}} $ is the set of configurations such that the forward flow compatible with the couple~$(T_\prec,  \bS_{N_k- 1})$ exists,
 and with the  constraints respecting the sampling.  Actually  note that   the above sum over ordered trees corresponds to a partition, meaning that for any given~$Z_{N_k} $, at most one term is non zero.

\medskip

 Given such an admissible tree $T_\prec$ let us define the relative positions at time~$t_{{\rm stop}}$
 $$
 \hat x_i:=x_{q_i}-x_{\bar q_i}\, .
 $$
Given the relative positions $\left(\hat x_s \right)_{s < i}$ and the velocities $V_{N_k}$, we fix a forward flow with collisions at times $\tau_{ 1}< \dots < \tau_{i-1} < \theta$. By construction, $q_i$ and $\bar q_i$ belong to two forward pseudo-trajectories that have not interacted yet. In other words,~$q_i$ and $\bar q_i$ do not belong to the same connected component in the graph~$G_{i-1} := (q_j,\bar q_j)_{1 \leq j \leq i-1}$. Inside each connected component, relative positions are fixed by the previous constraints, and one degree of freedom remains. Therefore we are going to vary $\hat x_i$ so that a forward collision at time $\tau_i \in ( \tau_{i-1},\theta)$ occurs between $q_i$ and $\bar q_i$ (moving rigidly the corresponding connected components). This collision condition defines a set $\cB_{T_\prec, i}   (\hat x_{1}, \dots, \hat x_{i-1}, V_{N_k})$.
The particles $q_i$ and $\bar q_i$ move in straight lines, therefore the measure of this set can be estimated by
$$
|\cB_{T_\prec, i} | \leq \frac{C}{\mu_\eps} |v^\e_{q_i}(\tau^+_{i-1}) - v^\e_{\bar q_i}(\tau^+_{i-1})| 	\left(\theta - \tau_{i-1}\right)
$$
and there holds
\begin{equation}
\label{eq: borne somme qi}
\sum_{q_i,\bar q_i}|\cB_{T_\prec, i} |  \leq \frac{C}{\mu_\eps} \left( V_{N_k}^2 + N_k\right) N_k 	\left(\theta - \tau_{i-1}\right)\;.
\end{equation}
Hence by Fubini's theorem
\begin{align}
& \sum_{T_\prec \in \cT^{\prec }_{N_k}}\int d\hat X_{N_k-1} 
\prod_{i=1}^{N_k-1} \indc_{\cB_{T_\prec, i }} \leq 
\sum_{T_\prec \in \cT^{\prec }_{N_k}} \int d\hat x_1
 \indc_{\cB_{T_\prec, 1 }} \int d\hat x_{ 2}\,  \dots\int d\hat x_{N_k-1} \indc_{\cB_{T_\prec, N_k-1}}
 \nonumber \\
 & \qquad \qquad \qquad \qquad 
 \leq \left( \frac{C}{\mu_\eps}\right)^{N_k-1}  \left( V_{N_k}^2 + N_k\right)^{N_k-1} N_k^{N_k-1}
 \int_{t_{{\rm stop}}}^{\theta} d\tau_{ 1}\, \dots \int_{\tau_{N_k-2} }^{\theta} d \tau_{N_k-1} 
 \indc_{\n_{k}}
 \label{eq: integration Bi}
\end{align}
where $\indc_{\n_{k}}$ is the constraint on times respecting the sampling in \eqref{eq: def I nk}. 
Retaining only the information that $n_k$ times are in the interval $(t_{stop}, t_{stop} + \tau)$ and the other $N_{k-1}-1$ times are in $(t_{stop} + \tau, \theta)$, we get by integrating over these ordered times an upper bound of the form
\begin{equation}
\label{eq: time constraints 4.1}
\frac{\tau^{n_k}}{n_k!} \frac{\theta^{N_{k-1}-1}}{(N_{k-1}-1) !} 
\leq  \frac{2^{N_k-1}}{(N_k-1) !} \tau^{n_k} \, \theta^{N_{k-1}-1} \; .
\end{equation}
Up to a factor $C^{N_k}$,  the  factorial $N_k!$ compensates the factor  $N_k^{N_k}$ in \eqref{eq: integration Bi}.
Furthermore,  for any $K,N$ and dimension $D>0$ 
\begin{equation}
\label{eq: inegalite exponentielle}
\sup_{ V \in \bbR^D} \Big\{ \exp \big( - \frac18 |V|^2 \big)  \; ( |V|^2 + K)^N  \Big\} \leq C^N e^K \; N^N.
\end{equation}
After integrating the velocities with respect to the measure~$M^{\otimes N_k}$, we deduce from the previous inequality that the term $\left( V_{N_k}^2 + N_k\right)^{N_k}$ leads to another factor of order $N_k^{N_k}$ which is compensated, up to a factor $C^{N_k}$, by the $N_k !$ in
\eqref{estimatePhiexp}.
Combining all these estimates,~$\big| \Phi _{N_k} \big|$ can be bounded from above uniformly with respect to one remaining parameter which takes into account the translation invariance of the system. For clarity, we decide arbitrarily that the 
 remaining degree of freedom is indexed by the variable~$x_{N_k}$.
This completes the proof of Lemma \ref{lem: hyp quasi orth 1}.
\end{proof}

\begin{proof}[Proof of Lemma~{\rm\ref{lem: hyp quasi orth 2}}]
The proof is similar to the one of the previous lemma, however, we have to analyse now  the dynamical constraints associated with two configurations $Z_{N_k} = (Z_\ell, Z_{\ell +1,  N_k} )$ and 
$Z'_{N_k} = (Z_\ell, Z_{N_k +1, 2N_k - \ell})$ sharing $\ell$ particles.
For each configuration, we fix the parameters  coding the collisions $\bS _{N_k- 1} =( s_i, \bar s_i) _{1\leq i \leq N_k- 1}$ and $\bS _{N_k- 1}' =( s_i', \bar s_i') _{1\leq i \leq N_k- 1}$.
By analogy with formula \eqref{estimatePhiexp}, we get
\begin{equation}\begin{aligned} 
\label{eq: Phi au carre}
&\big| \Phi_{N_k} (Z_{N_k}) \Phi_{N_k} (Z_{\ell} , Z_{N_k +1, 2N_k - \ell}) \big| \\
&\qquad \leq \| h \|_\infty^2
 \left( {\mu_\eps^{N_k-1}\over N_k!} \right)^2  
  \sum_{\bS _{N_k- 1}\atop \bS _{N_k- 1}'}
 \indc_{\{ Z_{N_k} \in \cR_{\bS _{N_k- 1}} \}    }
  \indc_{\{ Z_{N_k}' \in \cR_{\bS _{N_k- 1}'}  \}    } \, .
\end{aligned}
\end{equation}
We consider the forward flows of each set of particles~$Z_{N_k}$ and~$Z'_{N_k} $ starting
 at time~$t_{{\rm stop}}$. Both dynamics evolve independently and each one of them should have exactly 
 $N_k-1$ collisions to be compatible with an ordered tree as the ones used in the proof of Lemma~\ref{lem: hyp quasi orth 1}. 
  As the configurations $Z_{N_k}$ and~$Z'_{N_k} $ share $\ell$ particles in common,     strong correlations are imposed in order to produce a total of $2( N_k-1)$ collisions. For our purpose, it is enough to relax these constraints and to record only $2 N_k - \ell -1$ (weakly dependent)   ``clustering collisions" which will be indexed by an ordered graph~$T''_\prec$ with  $2N_k-\ell-1$ edges, as well as relative positions~$(\hat x_i)_{1 \leq i \leq 2N_k-\ell-1}$ at time $t_{{\rm stop}}$.

\begin{figure}[h] 
\centering
\includegraphics[width=3.5in]{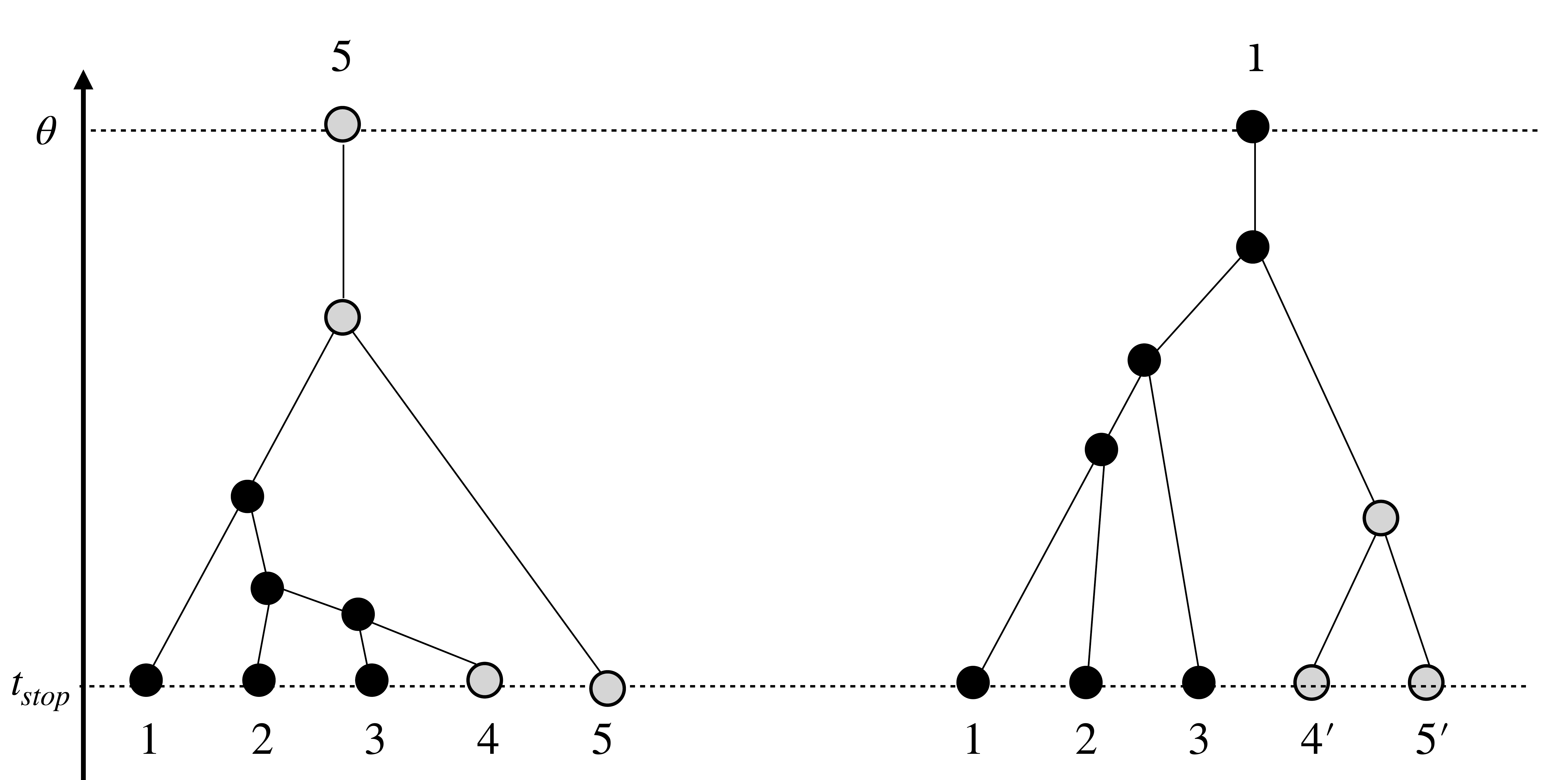} 
\hskip.6cm
\includegraphics[width=2.2in]{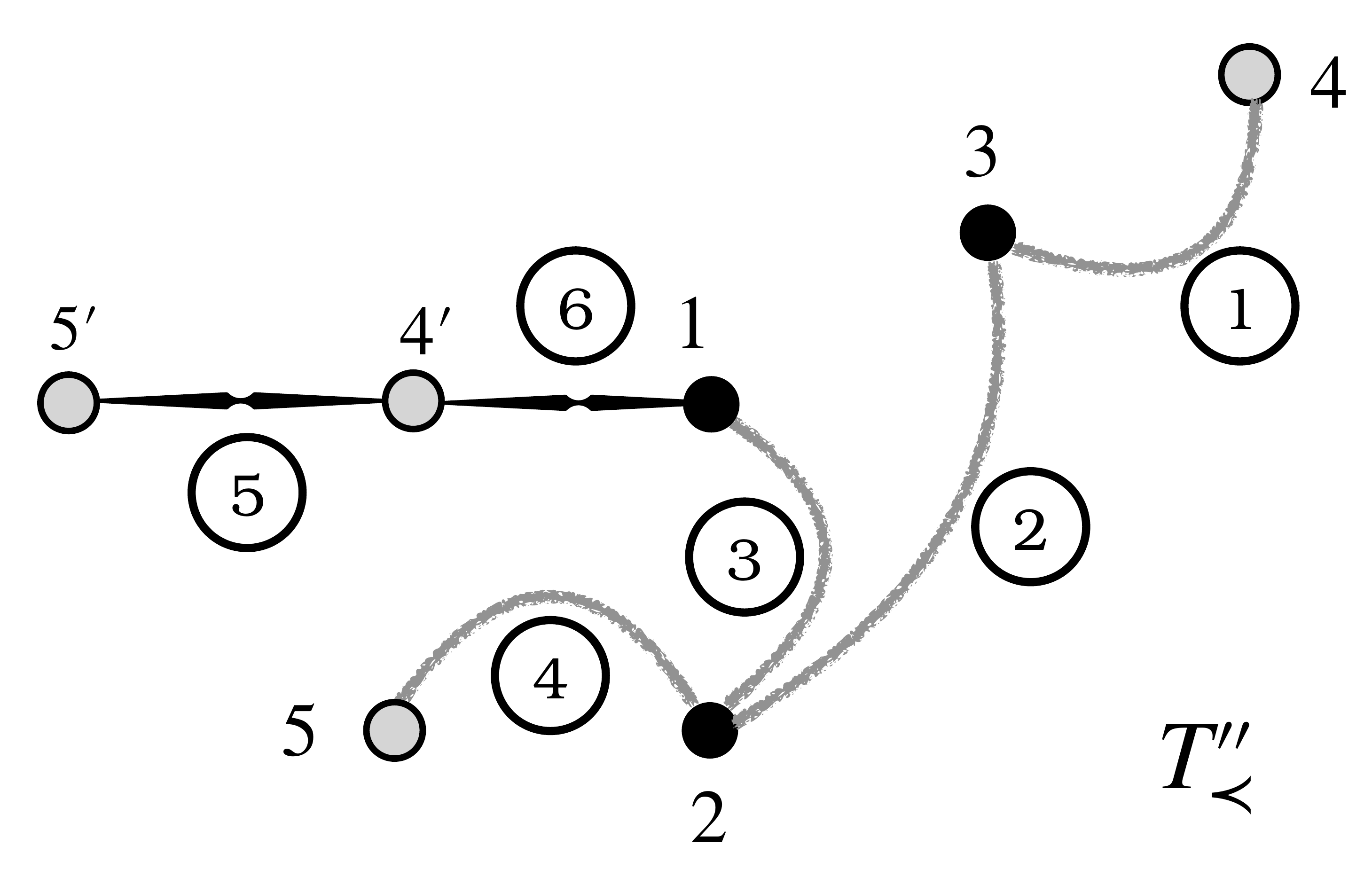} 
\caption{\small 
In the   figure on the left, an example of 2 pseudo-trajectories sharing $\ell = 3$ particles with $N_k = 5$.
The collision graph $T_\prec$ associated with the left pseudo-trajectory starting from $Z_5$ is depicted by 
the bended grey edges ordered according to the collision times.
The complete tree $T''_\prec$ is built starting from $T_\prec$ to which  two additional straight edges (numbered $5$ and $6$)
have been added to connect $4'$ and $5'$.} 
\label{figure: pseudo-trajectories}
\end{figure}

 The ordered graph~$T''_\prec$   is constructed as follows. As in the proof of Lemma~\ref{lem: hyp quasi orth 1}, we denote by~$T_\prec$ the ordered collision tree corresponding to the forward flow  of~$Z_{N_k}$, and by~$(\tau_i)_{1 \leq i \leq N_k-1}$ and~$(\hat x_i)_{1 \leq i \leq N_k-1}$ the  collision times and relative positions. 
 The first $N_k-1$ edges $(q_i, \bar q_i  )_{1 \leq i \leq N_k-1}$ of the graph ~$T''_\prec$ are the edges of  the  ordered tree~$T_\prec$, so that~$T_\prec$ is fully embedded in~$T''_\prec$
(this prescribes the constraints on the particles $Z_{N_k}$). The last $N_k - \ell$ edges in $T''_\prec$ will record the additional constraints on the remaining particles $Z_{N_k +1, 2N_k - \ell}$ which are involved in the dynamics of $Z_{N_k}'$ (see Figure \ref{figure: pseudo-trajectories}).

  The  edges~$(q_i, \bar q_i  )_{N_k \leq i \leq 2N_k-\ell}$     are added  as follows, keeping only the {\it clustering collisions} in the forward dynamics of $Z'_{N_k}$, i.e. the collisions associated with edges which are not creating  cycles in the graph :
\begin{itemize}
 \item the first clustering collision is the first collision in the forward flow of~$Z'_{N_k}$ involving at least one particle with label in~$[N_k+1,   2N_k-\ell]$.
  We denote  by~$(q_{N_k }, \bar q_{N_k })$   the  labels of the colliding particles and by~$\tau_{N_k }$  the corresponding   colliding time. We also define the ordered graph~$G_{N_k} = (q_j,  \bar q_j)_{1\leq j \leq N_k}$. Note that on Figure \ref{figure: pseudo-trajectories}, the graph $G_5$ is made of two components $\{1, 2, 3,4, 5\}$ and $\{ 4',5'\}$.
  \item
  for~$N_k +1 \leq i \leq 2N_k-\ell-1$, the $i$-th clustering collision is the first collision (after $\tau_{i-1}$) in the forward flow of~$Z'_{N_k}$ involving two particles which are not  in the same connected component of the graph $G_{i-1}$. 
By construction at least one of these particles belongs to   $Z_{N_k +1, 2N_k - \ell}$.
  We denote  by~$(q_{i }, \bar q_{i })$   the  labels of the colliding particles and by~$\tau_{i }$  the   corresponding collision time.   We also define the ordered graph~$ G_{i} = (q_j,  \bar q_j)_{1\leq j \leq i}$. 
   \end{itemize}
 By this procedure, we end up with a tree~$T''_\prec := (q_i, \bar q_i  )_{1\leq i \leq 2N_k-\ell - 1}$ with no cycles (nor multiple edges).
 We define as above the relative positions~$ \hat x_i:=x_{q_i}-x_{\bar q_i}$.

 \medskip
 Note that the sequence of times~$(\tau_i)_{1\leq i \leq 2N_k-\ell - 1}$ is only partially ordered. Indeed the times $\tau_1 < \dots < \tau_{N_k-1}$ associated with $Z_{N_k}$ are ordered, as well as the times   
 $\tau_{N_k} < \dots < \tau_{2 N_k-\ell -1}$ associated with the clustering collisions in $Z'_{N_k}$, but they are not mutually ordered. 
 Nevertheless, this is not  a problem since the only important point is that the  collision sets~$(\cB_{T''_\prec, i})_{1 \leq i \leq 2N_k-\ell-1}$, defined as  in the proof of Lemma~\ref{lem: hyp quasi orth 1}, only depend  on~$\hat x_{1}, \dots, \hat x_{i-1}, V_{2N_k-\ell}$. 
  When~$i \geq N_k$, this is less obvious than in the previous case since in the construction of~$T''_\prec$ some collisions (those in the forward flow of~$Z'_{N_k}$ leading to cycles)  have been left out, so one needs to check that the corresponding trajectories before time~$\tau_i $ can be reconstructed knowing only~$\hat x_{1}, \dots, \hat x_{i-1}, V_{2N_k-\ell }$. 
 
By construction,  for~$i \geq N_k$, the two particles $(q_i, \bar q_i)$ colliding at  time~$\tau_i$ belong  to two different connected components $C_{i-1}(q_i)$ and~$C_{i-1}(\bar q_i)$  of the dynamical graph~$ G_{i-1}  $. The trajectory of $q_i$ in the pseudo-trajectory of~$Z'_{N_k}$  up to time  $\tau_i$ depends only  
 \begin{itemize}
  \item on the relative positions $(\hat x_j) _{(q_j, \bar q_j) \in C_{i-1}(q_i)}$ at $t_{{\rm stop}}$ 
  \item and on any root of $C_{i-1}(q_i)$, for instance the  position $ x_{q_i} $ of $q_i$ at $t_{{\rm stop}}$.
  \end{itemize}
 The same holds for    the trajectory of $\bar q_i$.
 We can therefore write the colliding condition by moving rigidly the two connected components $C_{i-1}(q_i)$ and~$C_{i-1}(\bar q_i)$, which provides as previously a condition on $\hat x_i$.

\medskip

%
%

From this point, we can   proceed exactly as in the previous lemma and the sets $\cB_{T''_\prec, i}$   satisfy the same estimates as before. 

\begin{align}
\label{eq: lemme 4.2 integrale}
\sum_{T''_\prec } & \int d\hat X_{2N_k-\ell - 1}  \!  \!  \! 
\prod_{i=1}^{2N_k-\ell -1}  \!  \!  \! \indc_{\cB_{T''_\prec, i }}\\
& \leq \left( \frac{C}{\mu_\eps}\right)^{2N_k-\ell - 1}  \left( V_{N_k}^2 + N_k\right)^{N_k-1} N_k^{N_k-1}  
 \left( (V'_{N_k})^2+ N_k\right)^{N_k - \ell }   N_k^{N_k-\ell} \nonumber \\
& \qquad \qquad \times  \int_{t_{{\rm stop}}}^{\theta} d\tau_{ 1} \dots\!     \int_{\tau_{N_k-2} }^{\theta} d \tau_{N_k-1} 
 \indc_{\n_{k}} \times \int_{t_{{\rm stop}}}^{\theta} d\tau_{N_k }  \dots \!  \int_{\tau_{2N_k-\ell-2} }^{\theta} d \tau_{2N_k-\ell - 1}\,  . \nonumber
 \end{align}
Notice that the first $N_k-1$ ordered time integrals correspond to the constraints in the tree~$T_\prec$ and are estimated from above by 
$\frac{2^{N_k-1}}{(N_k-1)!} \tau^{n_k} \, \theta^{N_{k-1}-1}$
as  in \eqref{eq: time constraints 4.1}.
The sampling in~\eqref{eq: def I nk} is omitted for the remaining times which are simply constrained to satisfy $\tau_{N_k } < \dots < \tau_{2N_k-\ell - 1} \leq \theta$, so that 
$$
\begin{aligned}
&  \int_{t_{{\rm stop}}}^{\theta} d\tau_{ 1}\, \dots \int_{\tau_{N_k-2} }^{\theta} d \tau_{N_k-1} 
 \indc_{\n_{k}} \times \int_{t_{{\rm stop}}}^{\theta} d\tau_{N_k }\, \dots \int_{\tau_{2N_k-\ell-2} }^{\theta} d \tau_{2N_k-\ell - 1} \\
& \qquad 
\leq 
\frac{2^{N_k-1}}{(N_k-1) !} \tau^{n_k} \, \theta^{N_{k-1}-1} \times 
\frac{\theta^{N_k-\ell}}{(N_k-\ell)!}
\leq 
\frac{C^{N_k }}{( N_k-\ell)!(N_k-1)! }  \tau^{n_k} \, \theta^{2 N_k - \ell -1 - n_k} \, .
\end{aligned}
$$Plugging this estimate in \eqref{eq: lemme 4.2 integrale}, we deduce that 
$$
\begin{aligned}
\sum_{T''_\prec } & \int d\hat X_{2N_k-\ell - 1} 
\prod_{i=1}^{2N_k-\ell -1} \indc_{\cB_{T''_\prec, i }} \\
&  \leq 
\left( \frac{C}{\mu_\eps}\right)^{2N_k-\ell - 1}    \tau^{n_k} \, \theta^{2 N_k - \ell -1 - n_k}
  \left( V_{N_k}^2 + N_k\right)^{N_k-1}  \left( (V'_{N_k})^2+ N_k\right)^{N_k  } N_k^{-\ell}. 
\end{aligned}
$$
 We conclude as in the proof of Lemma \ref{lem: hyp quasi orth 1} by integrating  with respect to velocities $V_{2N_k - \ell}$, and by using the prefactor $(N_k!)^{-2}$ from~(\ref{eq: Phi au carre})  to compensate, up to a factor~$C^{N_k} $, the  divergence $N_k^{2N_k}$ coming from (\ref{eq: inegalite exponentielle}). 
\end{proof}


\section{The cost of non-clustering constraints}
\label{section - recollisions}

\setcounter{equation}{0}

In this section we prove Proposition~\ref{proposition - recollisions}.  
The proof consists in applying Proposition~\ref{prop: Quasi-orthogonality estimates}, and for this we revisit
the proof of Lemmas~\ref{lem: hyp quasi orth 1} and~\ref{lem: hyp quasi orth 2}, to gain some extra smallness thanks to the recollision.
Recall that 
\begin{equation}\label{eq:G1rec}
 \begin{aligned}
 &  G^{\eps,{\rm rec }} _1 (\theta ):=  \sum_{k = 1} ^{K } \sum _{( n_j \leq 2^j)_{j\leq k-1 } }  \sum_{ r = 1}^R  \sum_{n_k\geq 0    }   \sum_{   n_k^0 + n_k^{\rm{rec}}=n_k}Q^{\eps0} _{1, n_1} (\tau) \dots Q^{\eps0}_{N_{k-2}, N_{k-1} } (\tau)   \\
 &   \circ Q^{\eps0} _{N_{k-1},N_{k-1} +n_k^{0}  }( (r-1) \delta)   Q^{\rm rec } _{N_{k-1} +n_k^0 , N_{k-1} +n_k^0 + n_k^{\rm{rec}}   }  ( \delta)  \widetilde G^{\eps}_{N_k}  (\theta- (k-1)  \tau - r\delta  ) \indc_{|V_{N_k} | \leq\mathbb V} \,.
 \end{aligned}
\end{equation}
Let us start by fixing an integer~$1 \leq k \leq K$, integers~$ (n_j)_{1 \leq j \leq k-1}$ with~$n_j\leq 2^j$, as well as an integer $1 \leq r \leq R$, and two integers $n_k^0$, $n_k^{\rm{rec}}$ summing to~$n_k$. We set~$\n_k:=( (n_j)_{1 \leq j \leq k-1},n_k^0,n_k^{\rm{rec}},r)$
and
\begin{align*}
I^{{\rm rec}}_{r, \n_k}  := \int  & h(z_1)
Q^{\eps0} _{1, n_1} (\tau) \dots  \\
& \dots Q^{\eps0} _{N_{k-1},N_{k-1} +n_k^{0}  }( (r-1) \delta)Q^{\rm rec } _{N_{k-1} +n_k^0 , N_{k-1} +n_k^0 + n_k^{\rm{rec}}   }  ( \delta)   \widetilde G^\eps_{N_k} (t_{{\rm stop}}) \indc_{|V_{N_k} | \leq\mathbb V}
\end{align*}
with
$t_{{\rm stop}}: = \theta- (k-1)  \tau - r\delta$.
As previously we want to use a change of variables in order to write an expression of the type
$$
I^{{\rm rec}}_{r, \n_k}  = \int    \Phi^{{\rm rec}}_{  N_k}  (Z_{N_k} ) \widetilde G^{\eps}_{N_k} (t_{{\rm stop}} ,Z_{N_k}  )  dZ_{N_k}  \, .$$
However contrary to the previous case, the presence of recollisions requires the introduction of additional parameters to retrieve the injectivity of the change of variables~(\ref{change of variables}).  On the time interval~$(t_{{\rm stop}}+\delta, \theta)$, 
the 
situation is the same as in the previous section since there are no recollisions by definition.
On $(t_{{\rm stop}}, t_{{\rm stop}} +\delta)$ however,  the
construction of the forward dynamics starting from a configuration~$Z_{N_k}$ is more intricate since there is at least one recollision.   The important fact is that the number of   recollisions is under control.
Indeed
  the configuration at time $t_{{\rm stop}}$ has no cluster of more than $\gamma$ particles by construction, and $|V_{N_k} | $ has been set to be smaller than~$\mathbb V$. It follows that each particle is at a distance less than $2 \bbV \delta$ of at most $\gamma- 1$ other particles at time $t_{{\rm stop}}$. But thanks to the energy cut-off, two particles which are at a distance larger  than~$2 \bbV \delta$ at time~$t_{{\rm stop}}$ cannot collide during the time interval $(t_{{\rm stop}}, t_{{\rm stop}} +\delta)$. 
 Therefore, each particle may interact at most with~$\gamma- 1$ particles on this small interval. Furthermore, there cannot be any  recollision due to periodicity as~$\bbV \delta \ll1$.
Since the total number of collisions for  a system of $\gamma $ hard spheres in the whole space  is finite (see Section 5 in \cite{Va79}), 
 say at most~${\mathcal K}_\gamma$, each particle cannot have more than ${\mathcal K}_\gamma$ recollisions. We then associate with each particle $i$ an index $\kappa_i$ (less than ${\mathcal K}_\gamma$) which is decreased by one each time the particle undergoes a recollision. We denote by~$\bk_{N_k} $ the set of recollision indices~$(\kappa_i)_{1 \leq i \leq N_k}$. 
Given a collision tree $a \in \cA_{1,N_k-1} ^\pm$,
this new set of parameters enables us to recover the lost injectivity, by  applying the following rule to reconstruct the forward dynamics.
At each collision,
\begin{itemize}
\item if the two colliding particles have a positive index, then it is a recollision;
\item if one particle has  zero index, then it is a collision~: the label of the particle which disappears, and the possible scattering of the other colliding particle are prescribed by the collision tree $a$.
\end{itemize}
Note that the disappearing particle  should have zero index, else the trajectory is not admissible.

\medskip

  Finally let us define, for each~$a$ and each~$ \bk_{N_k}$ in  $ \{0,\dots, {\mathcal K}_\gamma\} ^{N_k}$,  the set~$\cR_{\bk_{N_k},a,\n_k}^{{\rm rec}} $ of configurations compatible with pseudo-trajectories having the following constraints:
\begin{itemize}
 \item[(i)]  the number of new particles added respectively on the time intervals $(\theta- j \tau, \theta- (j-1) \tau)$,  $(\theta-(k-1) \tau - (r-1) \delta, \theta-(k-1) \tau)$ and $(\theta-(k-1) \tau - r\delta,\theta-(k-1) \tau - (r-1) \delta)$ are respectively $n_j$,  $n_k^0$ and $n_k^{\rm{rec}}$; 
  \item[(ii)]the pseudo-trajectory involves no recollision on the interval~$(t_{{\rm stop}}+\delta, \theta)$ and at least one on~$(t_{{\rm stop}}, t_{{\rm stop}} +\delta)$;  
  \item[(iii)] 
   the addition of new particles is prescribed by the collision tree $a$ and recollisions between particles are compatible with~$\bk_{N_k}$; 
  \item[(iv)] the total energy at $t_{{\rm stop}} $ is less than $\bbV^2/2$:
   \item[(v)] the configuration at time $t_{{\rm stop}}$ has no cluster of more than $\gamma$ particles.
 \end{itemize}
Then  the change of variables, as in \eqref{change of variables},
$$
\big(z_1, ( t_i,\omega_i,v_{1+i})_{1 \leq i \leq N_k - 1}\big) \longmapsto \left(  Z_{N_k}^\eps (t_{{\rm stop}}), \bk_{N_k} \right) 
$$ 
of range
$$
\Big\{
(  Z_{N_k}, \bk_{N_k}) \in \cD^\eps_{N_k}\ \times   \{0,\dots, {\mathcal K}_\gamma\} ^{N_k}\, ,  \quad  Z_{N_k} \in \cR_{\bk_{N_k},a,\n_k}^{\rm rec} 
\Big \}
$$
is injective (of course not surjective). 
\medskip

So we can now write
$$
I^{{\rm rec}}_{r, \n_k}  = \int    \Phi^{{\rm rec}}_{  N_k}  (Z_{N_k} ) \widetilde G^{\eps}_{N_k} (t_{{\rm stop}} ,Z_{N_k}  )  dZ_{N_k}  $$
where  
 \begin{equation}\label{Phirec-def}
  \Phi^{{\rm rec}}_{  N_k}
  (Z_{N_k})
   := {\mu_\eps^{N_k-1}\over N_k !}  \sum_{\sigma\in \mathfrak S_{N_k} } \sum_{a \in \cA_{1,N_k-1} ^\pm} 
   \sum_{\bk_{N_k}} 
   h( z_{\sigma(1)}^\eps (\theta) ) \indc_{ \{ Z_{\sigma  }  \in 
   \cR_{\bk_{N_k},a,\n_k}^{\rm rec} \} }\prod_{i=1}^{N_k - 1}  s_i \, .
  \end{equation}
Proceeding as in \eqref{eq:hatPhidef 2}, we define $\hat \Phi^{{\rm rec}}_{N_k}$ by substracting the mean and rewrite $I^{{\rm rec}}_{r,\n_k}$ as an expectation
\begin{align*}
I^{{\rm rec}}_{r,\n_k}   = \bbE_\eps\Big( \mu_\eps^{1/2} \; 
\hat \Phi^{{\rm rec}}_{N_k} \big( {\bf Z}^\eps_\cN (t_{{\rm stop}}) \big) \;
  \zeta^\eps_0(  g_0) \; \indc_{  \Upsilon_\cN^\eps}  \Big)
 +   \mu_\eps^{1/2} \;  \bbE_\eps\left( \Phi^{{\rm rec}}_{N_k} \right) 
 \bbE_\eps\left( \zeta^\eps_0(  g_0)  \indc_{  \Upsilon_\cN^\eps}  \right) .
\end{align*}
 Following  \eqref{eq: inegalite I exp}, a Cauchy-Schwarz inequality implies
$$
|I^{{\rm rec}}_{r,\n_k}  |
  \leq 
\bbE_\eps \Big( (\zeta^\eps_0 (g_0)) ^2 \Big)^{1/2}
   \;    \bbE_\eps \Big(   \, 
  \mu_\eps \Big(  \hat \Phi^{{\rm rec}}_{N_k} \big( {\bf Z}^\eps_\cN (t_{{\rm stop}} ) \big) \Big)^2  \; 
  \Big)^{1/2} 
+ \mu_\eps^{1/2} \;  \bbE_\eps\left(\Phi^{{\rm rec}}_{N_k}\right) 
 \bbE_\eps\left( \zeta^\eps_0(  g_0)  \indc_{  \Upsilon_\cN^\eps}  \right) .
$$
As in~(\ref{eq: borne finale I}), this can be estimated by Proposition \ref{prop: Quasi-orthogonality estimates} and using~\eqref{eq: borne reste avec puissances en cN}, once
we check that~$\Phi^{{\rm rec}}_{N_k} $ satisfies the assumptions \eqref{eq: quasi orth 1} and \eqref{eq: quasi orth 2} of Proposition \ref{prop: Quasi-orthogonality estimates}. This is the purpose of the following two lemmas.
\begin{Lem}
\label{lem: hyp quasi orth 1bis}
There exists $C>0$ such that for~$d \geq 3$, 
\begin{equation}\label{eq:lem1}
\begin{aligned}
&\sup_{x_{N_k}\in \T^d} \int \big | \Phi^{{\rm rec}}_{N_k}  (Z_{N_k})  \big|  M^{\otimes N_k}(V_{N_k})
\, dX_{N_k-1} dV_{N_k}  \\
&\qquad  \leq C^{N_k }  \|h\|_{L^\infty(\D)}
\delta^{\max (1, n_k^{\rm{rec}})} \tau^{(n_k^0-1)_+} \, (\bbV \theta)^{d+1}\,\theta^{N_{k-1}-1} 
\eps |\log \eps| \;.
\end{aligned}
\end{equation}
\end{Lem}
\begin{Lem}
\label{lem: hyp quasi orth 2bis}
There exists $C>0$ such that, for any $\ell = 1,\dots, N_k $ and for~$d \geq 3$, 
\begin{equation}\label{eq:lem2}
\begin{aligned}
& \sup_{x_{2N_k - \ell}\in \T^d}\int  \big| \Phi^{{\rm rec}}_{N_k}  (Z_{N_k}) \Phi^{{\rm rec}}_{N_k}  (Z_{\ell} , Z_{N_k +1, 2N_k - \ell})   \big|\\
& \qquad\qquad\qquad \times  M^{\otimes (2N_k - \ell)}(V_{{2N_k - \ell}})
\, dX_{{2N_k - \ell-1}}dV_{{2N_k - \ell}} \\
 &\qquad\leq C^{{N_k } } \mu_\eps^{\ell -1}N_k^{-\ell}  \|h\|_{L^\infty(\D)}^2
\; \delta^{\max (1, n_k^{\rm{rec}})} \; \tau^{(n_k^0-1)_+} \,
(\bbV \theta)^{d+1}\,\theta^{2N_{k}-\ell-1-n_k}  
\eps|\log \eps|  \;.
\end{aligned}
\end{equation}
\end{Lem}
 Assuming these lemmas are true, let us conclude the proof of Proposition~\ref{proposition - recollisions}.
Thanks to Proposition \ref{prop: Quasi-orthogonality estimates} and using~\eqref{eq: borne reste avec puissances en cN}, there holds
$$
 \begin{aligned}
| I^{{\rm rec}}_{r,\n_k}   |
& \leq 
\Big( C^{{N_k } } \|h\|_{L^\infty(\D)} \|g_0\|_{L^2_M}\Big) \\
& \times \Big[  \eps^\frac12|\log \eps| 
\Big(\eps   \, \theta^{2 (N_{k-1}-1)} +
 \sum_{\ell =1}^{N_k} \theta^{2N_{k}-\ell-1- n_k} \Big)^{1/2}  \delta^{ \frac{1}{2} \max (1, n_k^{\rm{rec}})} \tau^{\frac{1}{2}  (n_k^0 - 1 )_+} \\
&\qquad\qquad\qquad 
  +  \eps |\log \eps|  \theta^{N_{k-1} 
+ \frac12}  \tau^{(n_k^0-1)_+} \;  \bbV ^{\frac{d \gamma}{2}} \;\mu_\eps^{\frac{\gamma}{2} + 1} \, \delta^{\frac{d \gamma-1}{2}+\max (1, n_k^{\rm{rec}})} 
\Big] \, .
\end{aligned}
$$
Using  (\ref{H1}), we get
\begin{equation}
\label{eq: borne finale I rec}
| I^{{\rm rec}}_{r,\n_k}  | 
\leq \eps^\frac12|\log \eps|  \, 
\|h\|_{L^\infty(\D)}  \|g_0\|_{L^2_M}\; ( C \theta) ^{N_{k-1}+n_k/2}  
\delta^{\frac12\max (1, n_k^{\rm{rec}} )} \; \tau^{\frac12 (n_k^0-1)_+}\,  .
\end{equation}
Finally  we are in position 
to sum over all parameters. We find after summation  over~$n_k^0$ and~$n_k^{\rm{rec}}$, then    $r $ (which leads to a factor $\tau/\delta$) and finally ~$(n_j)_{j<k}$ and  $k$,
$$
\left| \int dz_1   G^{\eps,{\rm rec}}_{1} (\theta) h(z_1) \right|
\leq \frac{\tau}{\delta}
\Big( \sum_{k = 1} ^K 2^{k^2} \Big)
( C \theta) ^{2^K}
\delta^\frac12
\eps^\frac12|\log \eps|\,  \|h\|_{L^\infty(\D)}  \|g_0\|_{L^2_M}
\,. $$ 
Now the logarithm of  the right-hand side behaves as
$$
\frac{1}{2} \log \frac{\eps}{\delta}  + 2^K \log (C\theta) \longrightarrow -\infty \,, \quad  \mu_\eps \to \infty,
$$
by the scalings  (\ref{existsa}) to control $\frac{\eps}{\delta}$ (recalling that~$0<a<1$) and \eqref{H2} to bound from above $K= {\theta \over \tau} \leq \frac 12\log|\log\eps|$ for $\mu_\eps$  large enough.
It follows that
$$
\lim_{\mu_\eps \to \infty}\int dz_1   G^{\eps,{\rm rec}}_{1} (\theta)
h(z_1)= 0 
$$ 
which ends the proof of Proposition~\ref{proposition - recollisions}.  \qed

\begin{proof}[Proof of Lemma {\rm\ref{lem: hyp quasi orth 1bis}}]
We shall follow the method of the previous section,  by introducing the set of signs~$\bS _{N_k- 1} =( s_i, \bar s_i) _{1\leq i \leq N_k- 1}$, with $(s_i, \bar s_i)$ characterizing the $i$-th creation  (whether there is scattering or not and which particle remains). Then if~$\bS _{N_k- 1},\bk_{N_k}$ are prescribed, the mapping
$$
\big(a,\sigma,z_1, ( t_i,\omega_i,v_{1+i})_{1 \leq i \leq N_k - 1}\big) 
\longmapsto \left( Z^\e_{  \sigma} (t_{{\rm stop}})\right)
$$
is injective and we infer that
$$\big| \Phi^{{\rm rec}}_{N_k}  (Z_{N_k} ) \big| \leq  \| h \|_\infty 
{\mu_\eps^{N_k-1}\over N_k!}  \sum_{\bk_{N_k}, \bS_{N_k- 1}} 
 \indc_{\{ Z_{N_k} \in \cR_{ \bk_{N_k}, \bS_{N_k- 1}}^{{\rm rec}} \}    }\,. $$
We have defined~$ \cR_{ \bk_{N_k}, \bS_{N_k- 1}}^{{\rm rec}}$ as the set of configurations such that the forward flow compatible with $ \bk_{N_k},\bS _{N_k- 1}$ exists, and with the  constraints respecting the sampling in formula \eqref{eq:G1rec}.

\medskip

Now let us fix~$\bk_{N_k}, \bS_{N_k - 1}$, and evaluate the cost of the constraint that~$Z_{N_k} \in \cR_{\bk_{N_k}, \bS_{N_k-1}}^{{\rm rec}}$.
For this  as previously
 we split the  above sum according to  ordered trees~$ T_\prec=  (q_i, \bar q_i)_{1\leq i \leq N_k-1}$ encoding the ''clustering collisions'': the first  collision in the forward flow is necessarily clustering, say between particles~$q_1$ and~$\bar q_1$ at time~$\tau_1 \in (t_{{\rm stop}}, t_{{\rm stop}} +\delta)$. Clustering collisions are then defined recursively~: the $i$-th clustering collision is the first collision after time~$\tau_{i-1}$ involving two particles which are not in the same connected component of the collision graph $G_{i-1} = (q_j, \bar q_j)_{j \leq i - 1}$. We then denote by $(q_i, \bar q_i)$ the colliding particles and by $\tau_i$ the corresponding collision time. The last clustering collision is between~$q_{N_k-1}$ and~$\bar q_{N_k-1}$ at time~$\tau_{N_k-1} \in (\tau_{N_k-2},\theta)$.  Note that by construction we know that there are at least $\max(1, n_k^{rec})$  clustering collisions in the interval~$(t_{{\rm stop}}, t_{{\rm stop}} +\delta)$, and at least $n_k ^{rec}+ n_k^0$ clustering collisions in the interval~$(t_{{\rm stop}}, t_{{\rm stop}} +\tau)$. This leads to
\begin{equation}\label{estimatePhirec}
\big| \Phi^{{\rm rec}}_{N_k}  (Z_{N_k} ) \big| \leq  \| h \|_\infty 
{\mu_\eps^{N_k-1}\over N_k!}  \sum_{\bk_{N_k}, \bS_{N_k- 1}} \sum_{T_\prec \in \cT^\prec_{N_k}} 
 \indc_{\{ Z_{N_k} \in \cR_{T_\prec, \bk_{N_k}, \bS_{N_k- 1}}^{{\rm rec}} \}    } \, ,
\end{equation}
where $\cR_{T_\prec, \bk_{N_k}, \bS_{N_k- 1}}^{{\rm rec}}$ is the set of configurations such that the forward flow compatible with~$T_\prec,\bk_{N_k}, \bS_{N_k- 1}$ exists,
 and again with the  constraints respecting the sampling in formula~\eqref{eq:G1rec}.

 Notice that, since the  pseudo-trajectories involve recollisions,  the clustering collisions of the forward dynamics do not coincide in general with the creations in the backward dynamics. Furthermore, since the graph encoding all collisions has more than $(N_k - 1)$ edges, there will be at least one non clustering collision in the forward dynamics (see Figure \ref{figure: pseudo-trajectories recoll 1}).


 \begin{figure}[h] 
\centering
\includegraphics[width=3.2in]{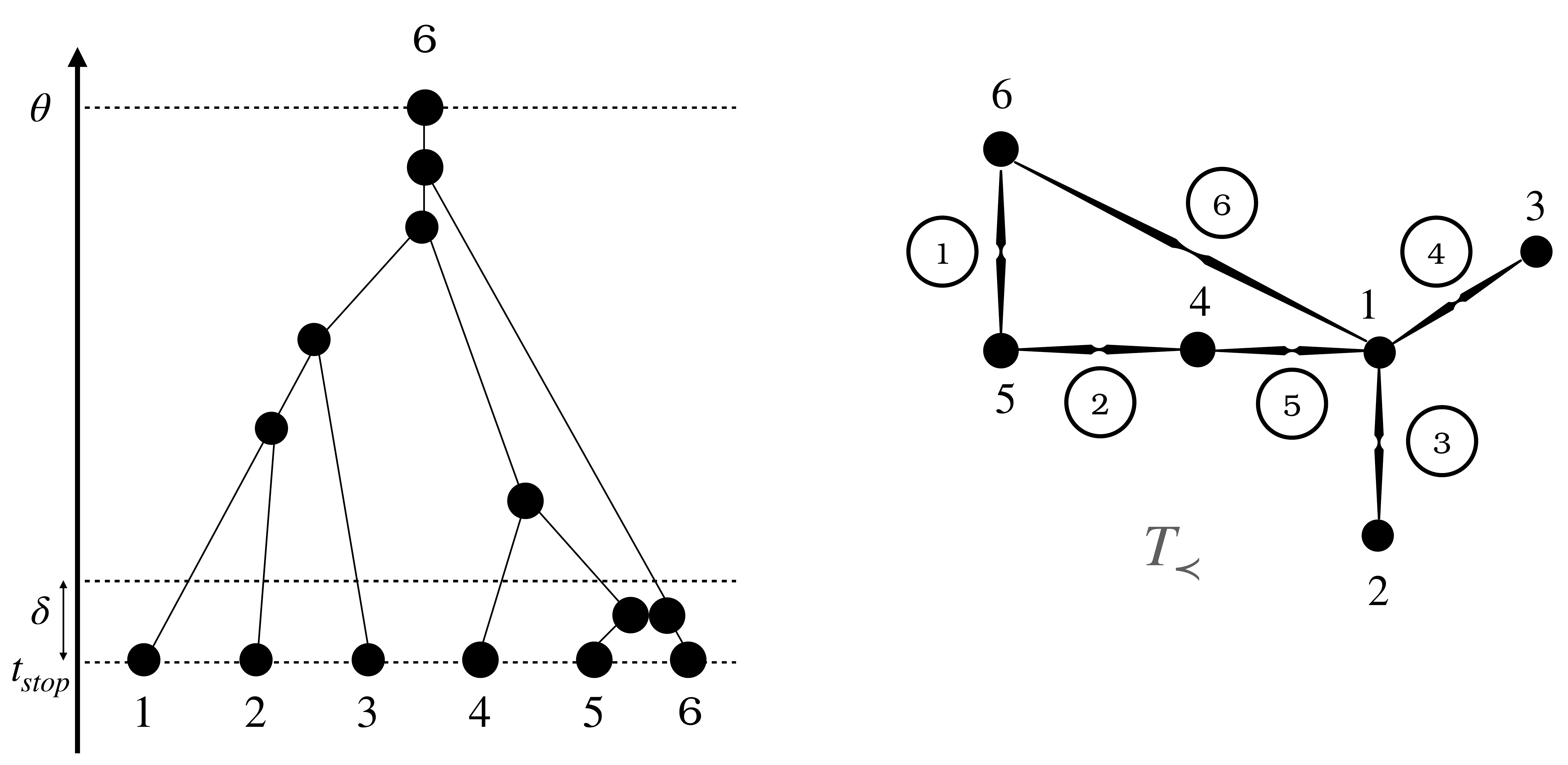} 
\caption{\small 
In the pseudo-trajectory (with $N_k = 6$) represented on the left figure,
a recollision occurs between $5,6$ in the time interval $[t_{{\rm stop}}, t_{{\rm stop}}+ \delta]$.
This recollision induces a cycle in the collision graph $T_\prec$ as shown in the second figure.
The time ordering of the clustering and non clustering collisions is represented by the circled numbers and the edges are added dynamically in $T_\prec$ following the forward dynamics, i.e. starting from $t_{{\rm stop}}$. As a consequence, the recollision between $5,6$ in the backward pseudo-dynamics becomes the first clustering collision in the forward dynamics and the non clustering collision is identified with the edge $(1,6)$ occurring close to time $\theta$.
}
\label{figure: pseudo-trajectories recoll 1}
\end{figure}

\medskip

We proceed exactly as in the proof of Lemma~\ref{lem: hyp quasi orth 1}.
Given  an admissible tree $T_\prec$,  the relative positions $\left(\hat x_s \right)_{s < i}$ and the velocities $V_{N_k}$,   we can vary $\hat x_i$ so that a forward collision at time~$\tau_i \in ( \tau_{i-1},\theta)$ occurs between $q_i$ and $\bar q_i$ and thus define the set~$\cB_{T_\prec, i}   (\hat x_{1}, \dots, \hat x_{i-1}, V_{N_k})$ of measure
$$
|\cB_{T_\prec, 1} | \leq \frac{C}{\mu_\eps} |v_{q_1} - v_{\bar q_1}| 	\delta
$$
and for $i>1$
$$
|\cB_{T_\prec, i} | \leq \frac{C}{\mu_\eps} |v^\e_{q_i}(\tau^+_{i-1}) - v^\e_{\bar q_i}(\tau^+_{i-1})| 	\left(\theta - \tau_{i-1}\right)\,.
$$
The point now is to see that the existence of a non clustering collision reinforces one of these conditions. As can be seen  in the two next propositions proved in Appendix~\ref{geometric estimates}, 
the first non clustering collision, say  at $\tau_{{\rm rec}}$ between $q $ and $q'$,  imposes  strong geometric constraints  on the history of these particles, especially on the  first deflection of the couple~$q,q'$ (moving backward in time).
This first  deflection corresponds to a clustering collision, say the~$j$-th, and we call ``parent" the corresponding  index~$j$.
\begin{Prop}\label{recollisionprop} 
Let~$q$ and~$q'$ be the labels of the two particles involved in the first non clustering collision at time $\tau_{{\rm rec}}$, and denote by~$\tau_j$ the first time of  deflection of~$q$ or~$q'$, moving down from $\tau_{{\rm rec}}$ to $t_{{\rm stop}}$. Assume that~$d \geq 3$. If the non clustering collision is due to space periodicity
$$
 \int \indc _{\mbox{\tiny {\rm Periodic non clustering collision  with parent~$j$}}}\, \indc_{\cB_{T_\prec, j }} \,   d \hat x_j  \leq  \frac C {\mu_\eps^2}( \bbV \theta) ^{d+1} \, .
$$
If the first deflection  involves $q$ and a particle~$c \neq  q'$,   then denoting by~$\bar v_q ,  \bar v_{ q'}$ the velocities at $\tau_{j-1}^+$,
$$
 \int \indc _{\mbox{\tiny {\rm Non clustering collision with parent $j$}}} \, \indc_{\cB_{T_\prec, j }} \, d \hat x_j  
\leq {C\over \mu_\eps}  (\bbV  \theta)^d \bbV \, {\eps |\log \eps|  \over  | \bar v_q - \bar v_{ q'}| } \, .
$$
\end{Prop} 

It finally remains to  eliminate the singularity~$1/ | \bar v_{ q} - \bar v_{q'}|  $, using the next deflection moving backward. Note that this singularity arises only  if  the first non clustering recollision is not a self-recollision, which ensures that the recolliding particles have at least two deflections before the non clustering collision in the forward flow.
\begin{Prop}\label{vsingularityprop}
 Let~$q$ and~$q'$ be the labels of two   particles  with velocities~$v_{q}$ and~$v_{q'}$, and denote by~$\tau_j$ the time of the first deflection of~$q$ or~$q'$ moving down to $t_{{\rm stop}}$.    Assume that~$d \geq 3$. Then,
$$
\int   \frac{\indc_{\cB_{T_\prec, j }} }{ |v_{ q}-v_{ q'}|}   \,  d\hat x_j \leq    \frac C {\mu_\eps} \,  
\big( \delta \indc_{j = 1} +  \theta \indc_{j \not =  1} \big)   \,.
$$
\end{Prop}
Note that if the first deflection ($j = 1$) occurs in the time interval $ (t_{{\rm stop}}, t_{{\rm stop}}+\delta)$, then the estimate is strengthened by a factor $\delta$.

\medskip
Combining both propositions and summing over all possible parents $j,j'$ of the first non clustering collision, we get
$$\begin{aligned} 
& \sum_{T_\prec \in \cT^{\prec}_{N_k}}\int d\hat X_{N_k-1} 
\prod_{i=1}^{N_k-1} \indc_{\cB_{T_\prec, i }} \leq \sum_{j,j'  {\tiny\hbox{parents}}} 
\sum_{T_\prec \in \cT^{\prec}_{N_k}} \int d\hat x_1
 \indc_{\cB_{T_\prec, 1 }} \int d\hat x_{ 2}\,  \dots\int d\hat x_{N_k-1} \indc_{\cB_{T_\prec, N_k-1}}\\   
& \quad  \leq \left( \frac{C}{\mu_\eps}\right)^{N_k-1}  (\bbV  \theta)^{d+1} N_k^2  \left( V_{N_k}^2 + N_k\right)^{N_k-1} N_k^{N_k-1}
 \int_{t_{{\rm stop}}}^{t_{{\rm stop}} +\delta } d\tau_{ 1}\, \dots \int_{\tau_{N_k-2} }^{\theta} d \tau_{N_k-1} 
\eps |\log \eps| \, \indc_{\n_k}
\end{aligned}$$
recalling that $\indc_{\n_{k}}$ is the constraint on times respecting the sampling in formula \eqref{eq:G1rec}.
Integrating over that simplex in time, and with respect to the Gaussian measure in velocity leads to the expected estimate.
Lemma~\ref{lem: hyp quasi orth 1bis} is proved.
\end{proof}

\begin{proof}[Proof of Lemma~{\rm\ref{lem: hyp quasi orth 2bis}}]
The proof combines arguments from the proofs of Lemmas \ref{lem: hyp quasi orth 2} and \ref{lem: hyp quasi orth 1bis}.
Our starting point is the estimate
\begin{equation}
\label{estimatePhirec}
\big| \Phi^{{\rm rec}}_{N_k}  (Z_{N_k} ) \big| \leq  \| h \|_\infty 
{\mu_\eps^{N_k-1}\over N_k!}  \sum_{\bk_{N_k}, \bS_{N_k- 1}} 
 \indc_{\{ Z_{N_k} \in \cR_{ \bk_{N_k}, \bS_{N_k- 1}}^{{\rm rec}} \}    }\,. 
\end{equation}
Let us fix  two families $(\bk_{N_k}, \bS_{N_k-1})$ and $(\bk_{N_k}', \bS_{N_k-1}')$ and consider a configuration $Z_{2N_k - \ell}$ such  that   $Z_{N_k}\in \cR_{\bk_{N_k}, \bS_{N_k-1}}^{{\rm rec}}$ and   $Z'_{N_k}= (Z_{\ell} , Z_{N_k+1, 2N_k - \ell}) \in \cR_{\bk_{N_k}', \bS_{N_k-1}'}^{{\rm rec}}$ .

We consider the forward flows of each set of particles~$Z_{N_k}$ and~$Z'_{N_k} $ starting
 at time~$t_{{\rm stop}}$. Both dynamics evolve independently and each one of them should have at least one non clustering collision.  As in the proof of Lemma~\ref{lem: hyp quasi orth 1bis} we denote by~$T_\prec$ the ordered collision tree corresponding to the clustering collisions of~$Z_{N_k}$, and by~$(\tau_i)_{1 \leq i \leq N_k-1}$ and~$(\hat x_i)_{1 \leq i \leq N_k-1}$ the  collision times and relative positions. 
Note that the non clustering collision on the dynamics of $Z_{N_k}$ reinforces one of the clustering constraint.

Starting from  this ordered  minimally connected tree~$T_\prec$ with $N_k$ vertices, we  construct an ordered  minimally connected graph with $2N_k - \ell $ vertices with the same procedure as in the proof of Lemma  \ref{lem: hyp quasi orth 2}. The  edges~$(q_i, \bar q_i  )_{N_k \leq i \leq 2N_k-\ell}$     are added by keeping only the ``clustering collisions" in the forward dynamics of $Z'_{N_k}$~:
\begin{itemize}
 \item The first clustering collision is the first collision in the forward flow of~$Z'_{N_k}$ involving at least one particle with label in~$[N_k+1,   2N_k-\ell]$.
  We denote  by~$(q_{N_k }, \bar q_{N_k })$   the  labels of the colliding particles and by~$\tau_{N_k }$  the corresponding   colliding time. We also define the ordered graph~$ G_{N_k} = (q_j,  \bar q_j)_{1\leq j \leq N_k}$;
  \item
  for~$N_k +1 \leq i \leq 2N_k-\ell-1$, the $i$-th clustering collision is the first collision (after $\tau_{i-1}$) in the forward flow of~$Z'_{N_k}$ involving two particles which are not  in the same connected component of the graph $G_{i-1}$. We denote  by~$(q_{i }, \bar q_{i })$   the  labels of the colliding particles and by~$\tau_{i }$  the   corresponding colliding time.   We also define the ordered graph~$ G_{i} = (q_j,  \bar q_j)_{1\leq j \leq i}$. 
   \end{itemize}
 By this procedure we end up with a tree~$T''_\prec := (q_i, \bar q_i  )_{1\leq i \leq 2N_k-\ell - 1}$ with no cycles (nor multiple edges).
 We define as above the relative positions~$ \hat x_i:=x_{q_i}-x_{\bar q_i}$.
 
 Necessary conditions to have  $Z_{N_k}\in \cR_{\bk_{N_k}, \bS_{N_k-1}}^{{\rm rec}}$ and   $Z'_{N_k}\in \cR_{\bk_{N_k}', \bS_{N_k-1}'}^{{\rm rec}}$ can be expressed recursively in terms of the collision sets~$(\cB_{T''_\prec, i})_{1 \leq i \leq 2N_k-\ell-1}$~:
 \begin{itemize}
  \item the sets   $\cB_{T''_\prec, i}$ only depend  on~$\hat x_{1}, \dots, \hat x_{i-1}, V_{2N_k-\ell}$ for any $i \leq 2N_k- \ell - 1$ (see Lemma \ref{lem: hyp quasi orth 2});
 \item One set of $(\cB_{T''_\prec, i})_{1 \leq i \leq N_k-1}$ has some extra smallness due to the existence of a non clustering collision in the dynamics of $Z_{N_k}$ 
 (see Lemma \ref{lem: hyp quasi orth 1bis}).
\end{itemize}
 We therefore end up with the estimate
$$
\begin{aligned}
\sum_{T''_\prec }\int & d\hat X_{2N_k-\ell - 1}  dV_{2N_k - \ell} M^{\otimes (2N_k - \ell)}
\prod_{i=1}^{2N_k-\ell -1} \indc_{\cB_{T''_\prec, i }}\\
&
  \leq 
\left( \frac{C}{\mu_\eps}\right)^{2N_k-\ell - 1} (\bbV  \theta)^{d+1}  \delta^{\max (1, n_k^{\rm{rec}})} \; \tau^{(n_k^0-1)_+}  \;
\theta^{2N_{k}-\ell - 1-n_k}  \;
\eps|\log \eps|  (N_k)^{2N_k-\ell}. 
\end{aligned}
$$
Summing over all possible $(\bk_{N_k}, \bS_{N_k-1})$ and $(\bk_{N_k}', \bS_{N_k-1}')$, we obtain the expected estimate.  Lemma~\ref{lem: hyp quasi orth 2bis} is proved.
\end{proof}


\section{Conclusion of the proof: convergence results}\label{section - principal}

\setcounter{equation}{0}

\subsection{Restricting the initial measure}
\label{section - conditioning}
 
This section is devoted to the proof of Proposition~\ref{proposition - conditioning}. 

To prove \eqref{eq: complementaire Upsilon},
 we evaluate the occurence of a cluster of size larger than $\gamma$ under the equilibrium measure. This can be estimated by considering the event that  $\gamma +1$ particles are located in a ball of 
radius~$ \gamma  \; \bbV \delta$ 
\begin{align*}
\bbP_\eps \big( \text{there is a cluster larger than $\gamma$ at time 0} \big) 
& \leq 
\bbE_\eps \left(  \sum _{(i_1,\dots, i_{\gamma + 1} )} 
\indc_{ \{ \text{$i_1,\dots, i_{\gamma + 1}$ are in a cluster} \}  }
\right)\\
& \leq  \mu_\eps^{\gamma +1} \big( \gamma \; \bbV \delta \big)^{d \gamma}\, .
\end{align*}
In the set $^c \Upsilon^\eps_\cN$ a cluster should appear (at least) at one of the $\theta/\delta$ time steps. Using a union bound, this completes \eqref{eq: complementaire Upsilon}.

\medskip

Let us now note that the measure restricted to   
$\Upsilon^\eps_\cN$ can be decomposed as 
\begin{equation*}
 \bbE_\eps \left( \zeta^\eps_0 (g_0) \indc_{ \Upsilon^\eps_\cN} \right) = 
\bbE_\eps \left( \zeta^\eps_0 (g_0) \right)  - \bbE_\eps \left( \zeta^\eps_0 (g_0)  \indc_{^c \Upsilon^\eps_\cN} \right) = - \bbE_\eps \left( \zeta^\eps_0 (g_0)  \indc_{^c \Upsilon^\eps_\cN} \right) ,
\end{equation*}
where we used that $\bbE_\eps \left( \zeta^\eps_0 (g_0) \right) = 0$.
Applying the Cauchy-Schwarz inequality, we get by \eqref{eq: moment ordre 2} and 
\eqref{eq: complementaire Upsilon} that
\begin{equation}
\label{eq: mesure tilde}
\begin{aligned}
\Big|  \bbE_\eps \left( \zeta^\eps_0 (g_0) \indc_{ \Upsilon^\eps_\cN} \right) \Big|
\leq  \bbE_\eps \big(   \zeta^\eps_0 (g_0)^2 \big)^{1/2}  
\bbP_\eps \big(  ^c \Upsilon^\eps_\cN \big) ^{1/2}
\leq    C_\gamma \|  g_0\|_{L^2_M}  \; 
\theta^{\frac12} \;  \bbV ^{\frac{d \gamma}{2}} \;\mu_\eps^{\frac{\gamma +1}{2}} \, \delta^{\frac{d \gamma-1}{2}}\, .
\end{aligned}
\end{equation}
This completes  \eqref{eq: borne reste avec puissances en cN}.

\medskip

To conclude the proof of Proposition~\ref{proposition - conditioning}, let us recall the
definitions
\begin{align}
& {}^c \widetilde G^{\eps }_{1}(\theta) =\sum_{m\geq0}   Q^\eps_{1,m+1}(\theta)   \, ^c  \tilde  G^{\eps0}_{m+1} \nonumber    \\
& G^{\eps,{\rm clust} }_1(\theta) =  \sum _{( n_k \leq 2^k)_{k\leq K} } Q^{\eps0} _{1, n_1} (\tau) \dots Q^{\eps0}_{N_{K-1}, N_K} (\tau) \,     ^c \tilde G^{\eps0}_{N_K} \;.\label{eq:G1ecldef}
\end{align}
We start by proving that  if 
 \begin{equation}
\label{eq: borne G clust1}
 \lim_{\mu_\eps \to \infty} \theta  \, \mu_\eps^{\gamma+1} \;  \delta^{d \gamma -1} \;\bbV^{d \gamma }   = 0
 \quad \text{then} \quad 
\lim_{\mu_\eps \to \infty}  \int dz_1  {}^c \widetilde G^{\eps }_{1}(\theta,z_1) h(z_1)  
=0\,.
\end{equation}We rewrite the integral in terms of an expectation
and then use the H\"older inequality
\begin{align*}
 \int dz_1 {}^c \widetilde G^{\eps }_{1}(\theta,z_1) h(z_1) & = 
 \bbE_\eps \big(  \indc_{^c \Upsilon^\eps_\cN} \zeta^\eps_0 (g_0) \zeta_\theta^\eps (h)  \big)\\
 & \leq  \bbP_\eps \big(  ^c \Upsilon^\eps_\cN \big)^{1/4}
\; \bbE_\eps \big(   \zeta^\eps_0 (g_0)^2 \big)^{1/2} 
\;  \bbE_\eps \big(   \zeta_0^\eps (h)^4  \big)^{1/4}\, .
\end{align*}
Recall that $h$ is in $L^\infty$.
Combining \eqref{eq: complementaire Upsilon} with the bounds in Proposition~\ref{Proposition - estimates on g0} on the moments of the fluctuation field, we get
\begin{align*}
 \int dz_1 {}^c \widetilde G^{\eps }_{1}(\theta,z_1) h(z_1) 
 \leq  C  \; \|g_0\|_{L^2_M}\;   
\theta^{1/4} \;
\mu_\eps^{ \frac{\gamma+1}4 } \;  \delta^{ \frac{d \gamma -1}4} \; \bbV^{ \frac{d \gamma }4} \, .
\end{align*}
This completes \eqref{eq: borne G clust1}.  
 
 \medskip
 
We turn now to proving that under~(\ref{H1}),
 \begin{equation}
\label{eq: borne G clust2}
\lim_{\mu_\eps \to \infty}  \int dz_1 G^{\eps,{\rm clust} }_1(\theta,z_1) h(z_1)  
=0\,.
\end{equation}
Proceeding as in \eqref{eq:hatPhidef 2}-\eqref{eq:hatPhidef 2 bis}, we get
\begin{align*}
& \int dz_1 G^{\eps,{\rm clust} }_1(\theta,z_1) h(z_1)\\
& \qquad = \sum_{\n_k} \bbE_\eps\Big( \mu_\eps^{1/2} \; \hat \Phi_{N_K} \big( {\mathbf Z}^\eps_\cN (0) \big) \;
  \zeta^\eps_0(  g_0) \; \indc_{^c  \Upsilon_\cN^\eps}  \Big)
 + \sum_{\n_k} \mu_\eps^{1/2} \;  \bbE_\eps\left(\Phi_{N_K}\right) 
 \bbE_\eps\left( \zeta^\eps_0(  g_0)  \indc_{^c  \Upsilon_\cN^\eps}  \right)\, ,
\end{align*}
where, according to \eqref{eq:G1ecldef}, $\Phi_{N_K} $ is conditioned on the sampling $\n_k$ with sub-exponential trees and no recollisions in $(0,\theta)$.
Applying the H\"older inequality to bound the first term and Cauchy-Schwarz for the second term leads to
$$
\begin{aligned} 
\left|  \int dz_1 G^{\eps,{\rm clust} }_1(\theta,z_1) h(z_1)  \right| 
  \leq&  \bbP_\eps \big(  ^c \Upsilon^\eps_\cN \big)^{1/4}
\; \bbE_\eps \big(   \zeta^\eps_0 (g_0)^4 \big)^{1/4} \;
\sum_{\n_k} \bbE_\eps\Big( \mu_\eps \; \Big( \hat \Phi_{N_K} \big( {\mathbf Z}^\eps_\cN (0) \big) \Big)^2 
\Big)^{1/2}\\
& +\bbP_\eps \big(  ^c \Upsilon^\eps_\cN \big)^{1/2}\; \bbE_\eps \big(   \zeta^\eps_0 (g_0)^2 \big)^{1/2}\;
 \sum_{\n_k} \mu_\eps^{1/2} \;  \bbE_\eps\left(\Phi_{N_K}\right)  \;
\, . \nonumber 
\end{aligned}
$$

Since $g_0$ belongs to $L^\infty$, the moments of the fluctuation field can be bounded by Proposition \ref{Proposition - estimates on g0}.
Thus the previous  term is   estimated as in~(\ref{eq: borne finale I}) and we find thanks to~(\ref{eq: complementaire Upsilon}) and \eqref{H1}
\begin{eqnarray}
 \left|  \int dz_1 G^{\eps,{\rm clust}}_1(\theta,z_1) h(z_1)  \right|  &\leq 
\theta^{\frac{1}{4}} \;
\mu_\eps^{ \frac{\gamma+1}4 } \;  \delta^{ \frac{d \gamma -1}4} \; \bbV^{ \frac{d \gamma }4} 
 \sum_{( n_k \leq 2^k)_{k\leq K} } ( C \theta )^{N_K}\nonumber \\
 & \leq \theta^{\frac{1}{4}} \;
\mu_\eps^{ \frac{\gamma+1}4 } \;  \delta^{ \frac{d \gamma -1}4} \; \bbV^{ \frac{d \gamma }4} \,2^{K^2} 
( C \theta )^{2^{K+1}}\;.\nonumber
\end{eqnarray}
Using the scaling \eqref{H2} of $K = \theta / \tau$,
this concludes the proof of Proposition~\ref{proposition - conditioning}.   
\qed


\subsection{Control of large velocities}\label{section - high energies}

Let us prove Proposition~\ref{proposition - high energies}. Recall that
$R = \tau / \delta$ and 
 $$
 \begin{aligned}
&  G^{\eps,{\rm vel}} _1 (\theta ) 
:=  \sum_{k = 1} ^{K } \sum _{( n_j \leq 2^j)_{j\leq k-1 } }  \sum_{ r = 1}^R  \sum_{n_k\geq 0    }   \sum_{   n_k^0 + n_k^{\rm{rec}}=n_k}Q^{\eps0} _{1, n_1} (\tau) \dots Q^{\eps0}_{N_{k-2}, N_{k-1} } (\tau)   \\
 &  \qquad \circ Q^{\eps0} _{N_{k-1},N_{k-1} +n_k^{0}  }( (r-1) \delta)Q^{\rm rec } _{N_{k-1} +n_k^0 , N_{k-1} +n_k^0 + n_k^{\rm{rec}}   } ( \delta)  \widetilde G^{\eps}_{N_k} (\theta- (k-1)  \tau - r\delta   )  \indc_{|V_{N_k} | >\mathbb V} \, .
  \end{aligned}$$
  Let us write (recalling that $Z^\e_{N_k}(t)$ denotes the coordinates of  the  pseudo-particles at time~$t$)
  $$
   \begin{aligned}
& \int dz_1 h(z_1) Q^{\eps0} _{1, n_1} (\tau) \dots Q^{\eps0}_{N_{k-2}, N_{k-1} } (\tau)  Q^{\eps0} _{N_{k-1},N_{k-1} +n_k^{0}  }( (r-1) \delta)  \\
 &  \qquad \circ Q^{\rm rec } _{N_{k-1} +n_k^0 , N_{k-1} +n_k^0 + n_k^{\rm{rec}}   } ( \delta)  \widetilde G^{\eps}_{N_k} (\theta- (k-1)  \tau - r\delta   )  \indc_{|V_{N_k} | >\mathbb V }\\
&   =\sum_{a \in \cA^\pm_{1,N_k-1}} \int_{\mathcal P_a^{\rm vel}}  h(z_1)\widetilde G^\eps_{N_k} \left(\theta- (k-1) \tau-r\delta ,Z_{N_k}^\e(\theta- (k-1) \tau-r\delta)\right)\\
& \qquad \qquad \times \prod_{i=1}^{N_k-1} \big((
v_{1+i}-v^\e_{a_i}(t_i^+))\cdot \omega_i
\big)_+ dt_id\omega_i
dv_{1+i} \, dz_1 \, , \end{aligned}$$
where~$\mathcal P_a^{\rm vel} $ is the subset of~$\D \times ((0,\theta) \times {\mathbb S}^{d-1} \times \mathbb \R^d)^{N_k-1}$ such that for any~$z_1, ( t_i,\omega_i,v_{1+i})_{1 \leq i \leq N_k-1} $ in~$\cP_a^{\rm vel} $, the associate backward trajectory is well defined and
 \begin{itemize}
 \item[(i)] there are $n_j$ particles added on the time intervals $(\theta- j \tau, \theta- (j-1) \tau)$
 for $j < k$,~$n_k^0$ particles added on $\big(\theta - (k-1) \tau- (r-1) \delta, \theta - (k-1)\tau \big)$ and $n_k^{\rm{rec}}$ particles added on~$\big(\theta - (k-1) \tau- r\delta, \theta - (k-1)\tau- (r-1) \delta\big)$;
 \item[(ii)] the pseudo-trajectory involves no recollision on $\big(\theta - (k-1) \tau- (r-1) \delta, \theta \big)$;
\item[(iii)] at time $\theta - (k-1) \tau- r\delta$, the total velocity satisfies $|V_{N_k} | >\mathbb V$.
 \end{itemize}

We notice that for all times
$$
|\widetilde G^\eps_{N_k} (t,Z_{N_k})| \leq C \mu_\eps M^{\otimes {N_k}}(V_{N_k})\|g_0\|_{L^\infty}
$$
so we find after integration, taking~$\mathbb V =  |\log \eps|$
$$
\begin{aligned}
\sum_{a \in \cA^\pm_{1,N_k-1}}&  \int_{\mathcal P_a^{\rm vel}}  h(z_1)
\widetilde G^\eps_{N_k} \left(\theta- (k-1) \tau-r\delta ,Z_{N_k}^\e(\theta- (k-1) \tau-r\delta)\right)\\
& \qquad \qquad\qquad \times \prod_{i=1}^{N_k-1} \big((
v_{1+i}-v^\e_{a_i}(t_i^+))\cdot \omega_i
\big)_+ dt_id\omega_i
dv_{1+i}  \,dz_1 \\
&\qquad\leq (C\theta)^{N_{k-1} } \|g_0\|_{L^\infty} \, \|h\|_{L^\infty} \, \mu_\eps (C\tau)^{ n_k^0} (C\delta)^{ n_k^{\rm{rec}}} \exp \left( - \frac{1} 4 |\log \eps|^2 \right) \,.
\end{aligned}$$
Since $\tau\ll1$ and $\delta \ll 1$, summing with respect to $n_k^0$ and $n_k^{\rm{rec}}$, then  to $r $ and finally  with respect to $(n_j)_{j<k}$ and  $k$, we get
\begin{equation}
\label{v-estimate-1}
 \int dz_1\   G^{\eps,{\rm vel}}_1(\theta,z_1) h(z_1) \leq 
 \left( \sum_{k=1} ^K 2^{k^2}  \right)R (C\theta)^{2^K} \exp \left( - \frac{1} 4 |\log \eps|^2 \right) 
 \|g_0\|_{L^\infty} \, \|h\|_{L^\infty} \, \mu_\eps \, .
\end{equation}
By assumptions (\ref{existsa}) and~(\ref{H2}),  there is~$a>0 $ such that
$$R= {\tau \over \delta} \ll   \eps^{-a} $$
and for~$\eps$ small enough there holds
$$K= {\theta \over \tau} \leq \frac 12\log|\log\eps|\,.$$
Then taking the logarithm of the right-hand side in (\ref{v-estimate-1}), we find that it is smaller than
$$  (a+d-1) |\log \eps | + |\log\e|^{\log2/2} \log (  C \theta) - \frac{1} 4 |\log \eps|^2 \rightarrow - \infty \hbox{ as } \eps \to 0\,.$$   The proposition is proved. \qed
%
%
%


\subsection{Convergence of the principal part}
\label{section - main part}

In this section, we prove Proposition \ref{proposition - PP}. This is based on classical arguments relying on $L^\infty$ estimates. We shall refer to the literature for details.

Following \cite{BLLS80}, the solution of the linearized Boltzmann equation \eqref{eq:LBE} can be rewritten as a Duhamel iterated formula (to conveniently compare it with \eqref{eq:G1epsmain}). 
To do so we introduce Boltzmann pseudo-trajectories~$\Psi_{1,m}$ on $(0,\theta)$, constructed as follows. 
For all $z_1$, all parameters~$(t_{i}, \omega_{i} , v_{n+i} )_{i= 1,\dots, m}$ with~$t_i>t_{i+1}$ and all collision trees $a \in \cA^\pm_{1,   m}$
(denoting by $Z_{m+1}(\tau )$ the coordinates of  the particles at time~$\tau \leq t_{m}$)
\begin{itemize}
\item start from $z_1$ at time $t $ and, by iteration on $i= 1,2,\dots, m$:
\item transport all existing particles backward on $(t_{i}, t_{i-1})$  (on $\D^{i}$),
\item add a new particle labeled $i+1$  at time $t_{i}$, at position $x _{a_{i}} (t_{i})$ and with velocity~$v_{1 +i}$,
\item   apply the scattering rule (\ref{scattlaw}) if $s_{i}>0$.
\end{itemize}
%
%
We then define the formal limit of \eqref{eq:seriesexp}
\begin{equation}
\label{eq: marginal 1 boltzmann hierarchy}
G _1 (\theta ):= \sum _{m \geq 0} Q_{1, m+1} (\theta)  G^{0}_{m+1}\,,\qquad n \geq 1
\end{equation}
where $Q_{1, m+1}$ is the Boltzmann's hierarchy operator
\begin{eqnarray*}
&&Q_{1, m+1}(\theta)G^{0}_{m+1}:=\sum_{a \in \cA^\pm_{1,   m} }\int    dT_{m}  d\Omega_{m}  dV_{2, 1+ m}
\prod_{i=  1}^{  m}  s_{i}\Big(\big( v_{1+i} -v_{  a_{i}} (t_{i}^+)\big) \cdot \omega_{i} \Big)_+
G_{m+1}^{0} \big (\Psi^{0}_{m+1}\big)
\end{eqnarray*}
and the initial data are given by 
\begin{equation} 
\label{eq:Gn0}
G^0_n(Z_n) := M^{\otimes n}(V_{n})\sum_{i = 1}^{n} g_0(z_i)\;,\qquad n \geq 1\;.
\end{equation}
From Property 3 in \cite{BLLS80}  (see also Section 1.1.3 in \cite{BGSR2}), 
$G_1(\theta)$ is equal to  the solution $Mg(\theta)$ of the  linearized Boltzmann equation  \eqref{eq:LBE}. 
Furthermore,  the $n$-particle correlation function~$G_n(t,Z_n)$  can  also be represented by 
Duhamel series as in \eqref{eq: marginal 1 boltzmann hierarchy} and it is  given by the following explicit expression for any $n \geq 1$ (see  \cite{BLLS80})
\begin{equation} 
\label{eq:Gnt}
\forall t \geq 0, \qquad 
G_n(t, Z_n) := M^{\otimes n}(V_{n})\sum_{i = 1}^{n} g(t,z_i)  \;.
\end{equation}
Following  the decomposition \eqref{main-decomposition} of $G^\eps _1 (\theta )$, we  write $G_1(\theta)$ as 
$$
G_1(\theta) = G_1^{{\rm main}}(\theta) + G^{{\rm exp}} _1(\theta) \, , 
$$
where the main part is given by 
$$
G_1^{{\rm main}}(\theta) := \sum _{( n_k \leq 2^k)_{k\leq K} }
Q _{1, n_1} (\tau) \dots Q_{N_{K-1}, N_K} (\tau) G_{N_K}^0\, ,
$$
and the superexponential part by 
$$
 G^{{\rm exp}} _1(\theta) :=
\sum_{k=1}^K \sum _{( n_j \leq 2^j)_{j\leq k-1} }
 \sum _{  n_k > 2^k } Q _{1, n_1} (\tau) \dots Q_{N_{K-1}, N_K} (\tau) G_{N_K}(\theta - k \tau)\, .
$$
This remainder term is much easier to control than the corresponding one of the particle system as 
there is no recollision in the limiting system and the correlation functions $G_{N_K}$
are explicit~\eqref{eq:Gnt}. Since the solution $g(t)$ of the  linearized Boltzmann equation  \eqref{eq:LBE} remains controlled in $L^\infty$-norm,
the correlation functions $G_{N_K}$ are also controlled in  $L^\infty$-norm at any time. 
Using $L^\infty$ estimates as in \cite{BGSR2}, the remainder $G^{{\rm exp}} _1$ can be neglected.

Recalling the principal part 
$$
  G^{\eps,{\rm main}} _1 (\theta )=  \sum _{( n_k \leq 2^k)_{k\leq K} } Q^{\eps0} _{1, n_1} (\tau) \dots Q^{\eps0}_{N_{K-1}, N_K} (\tau) \,  G^{\eps0}_{N_K}  \, ,
$$
we notice that the differences in this formula with respect to $G_1^{{\rm main}}(\theta)$ are due to:
\begin{itemize}
\item[1)] the initial data $G^{\eps0}_{N_K}$ vs.\,$G^{0}_{N_K}$\;;
\item[2)] the fact that pseudo-trajectories~$\Psi^\e_{1,m}$ are constrained to the set of parameters avoiding recollisions, and also to the set $\cG_{  m}^{\e}(a, Z_1 )$\,;
\item[3)] the fact that (at creations) particles in~$\Psi^\e_{1,m}$ collide at distance $\e$ while in~$\Psi^\e_{1,m}$ they collide at distance $0$.
\end{itemize}
These errors are controlled as in \cite{La75}.

For the initial data we apply the following classical lemma: we refer to~\cite{BLLS80, GSRT, PSS17} for a proof.
\begin{Lem}
\label{lem:ID}
There exists a positive constant $C$ such that, for any $n \in \N$,
$$
\left|\left(G^{\e 0}_{n} - G^0_n\right)\left(Z_n\right) \indc_{{\mathcal D}^\eps_{n}}\left(X_n\right)\right| \leq C^n M^{\otimes n}\left(V_n\right) \e \|g_0\|_{\infty}\,,
$$
when $\e$ is small enough.
\end{Lem}

Now, let $\Psi^E_{1,m}$ be an auxiliary pseudo-trajectory defined exactly as~$\Psi_{1,m}$, with the only difference that particle $i+1$  is created at position $x _{a_{i}} (t_{i}) +\eps s_{i}\omega_{i}$ (sometimes called Boltzmann-Enskog pseudo-trajectory). 
Correspondingly, we can define $Q^{E}_{1, m+1}$ exactly as $Q_{1,m+1}$, with $\Psi_{1,m}$ replaced by $\Psi^E_{1,m}$. By definition, 
$\Psi^{E}_{1,m}$ and $\Psi_{1,m}$ have identical velocities and the positions cannot differ more than $m \e$. In particular at time zero we have that the euclidean norm of the difference~$|\Psi^{E 0}_{N_k} - \Psi^{0}_{N_k}|$ is bounded by
\begin{equation} 
\label{eq:flowsconv}
|\Psi^{E 0}_{N_k} - \Psi^{0}_{N_k}| \leq N_k^\frac32\, \e\;.
\end{equation}

Next, we can simplify the integral in $G^{\eps,{\rm main}} _1$ by removing the constraint in point 2) above.
Let ${\mathcal O}^{\e}$ be the complement of the set of parameters in 2). 
Clearly the pseudo-particles in $\Psi^E_{1,m}$ can overlap (they can reach distance strictly smaller than $\e$).
However in absence of recollisions and overlaps, the auxiliary pseudo-trajectory coincides with the BBGKY pseudo-trajectory~$\Psi^E_{1,m} = \Psi^\e_{1,m}$.  We can therefore replace $\Psi^\e_{1,m}$ by $\Psi^E_{1,m}$ in the geometric representation for $G^{\eps,{\rm main}} _1$.
The contribution of ${\mathcal O}^{\e}$ to $Q^{E}_{1, 1+m}$  is bounded by a quantitative version of Lanford's argument. 
For instance by applying Eq. (D.3) in \cite{PS17}, combined with Eq. (C.7)  in \cite{PS17} to control the cross sections, 
one can show that there is a constant $\alpha \in (0,1)$ such that
\begin{eqnarray*}
&&\Big|\sum_{a \in \cA^\pm_{1,   m} }\int_{{\mathcal O}^{\e}}  dz\,  dT_{m}  d\Omega_{m}  dV_{ 2,1+  m}\,
h(z)\,\\
&& \qquad\times\prod_{i=  1}^{  m}  s_{i}\Big(\big( v_{1+i} -v^E_{  a_{i}} (t_{i}^+)\big) \cdot \omega_{i} \Big)_+
G_{1+m}^{0} \big (\Psi^{E0}_{1+m}\big)\Big|
\leq \|h\|_{L^\infty(\D)}\|g_0\|_{L^\infty(\D)} \eps^\alpha \left( C \theta \right)^{m}\;.
\end{eqnarray*}

Using this after Lemma \ref{lem:ID}, and controlling the error \eqref{eq:flowsconv} thanks to the Lipschitz norm of~$g_0$, we conclude that  
\begin{equation}
\left| \int\left( G^{\eps,{\rm main}} _1 (\theta )-G^{{\rm main}}_1 (\theta )\right)\, h(z)\, dz\right| 
\leq \|h\|_{L^\infty(\D)}\left(\e^\alpha \|g_0\|_{L^\infty(\D)}+ \e \|\nabla_x g_0\|_{L^\infty_M} \right) \sum _{\n_k} \left(C\theta\right)^{N_K+1}
 \end{equation}
which leads to Proposition \ref{proposition - PP}. \qed




\appendix

\section{$L^p$ a priori estimates}
\label{appendix-Lp}

For the sake of completeness, we state below some estimates on the  fluctuation field under the equilibrium measure. These bounds follow from a standard cluster expansion approach (see e.g. \cite{Ueltschi}).
\begin{Prop}
\label{Proposition - estimates on g0}
Let~$g_0$ be a function in $L^\infty$. 
Then for all~$1 \leq p < \infty$ and for~$\eps$ small enough, the moments of the fluctuation field are bounded: 
\begin{equation}
\label{eq: moment ordre 2,4}
 \bbE_\eps \Big( \big(\zeta^\eps_0 (g_0)\big) ^p \Big) \leq C_p \,,
\end{equation}
where the constant $C_p$ depends on $\|g_0\|_{L^\infty}$. 
\end{Prop}
\begin{proof}

The upper bounds on the moments will be recovered by taking the derivatives at $\lambda= 0$ of the following modified partition function
\begin{align*}
\Psi_\eps (\lambda) 
& := \frac{1}{\mu_\eps} \log  \bbE_\eps \Big( \exp  \big( \lambda \mu_\eps \pi^\eps_0 (g_0)  \big) \Big) +
\frac{1}{\mu_\eps} \log \cZ^ \eps\\
& = 
\frac{1}{\mu_\eps} \log 
\left( 1 + \sum_{N\geq 1}\frac{\mu_\eps^N}{N!}  
\int_{\T^{dN}\times \R^{dN}} 
 \Big( \prod_{i\neq j}\indc_{ |x_i - x_j| > \eps} \Big)
 \left( \prod_{i=1}^N M(v_i)
 \exp ( \lambda g_0(z_i) ) \right)dZ_N \right).
\end{align*}

As $\| g_0 \|_\infty < \infty$, Equation (26) from  \cite{Ueltschi} applies as soon as $\eps$ is small enough. Thus the modified partition function
can be expanded as a uniformly converging series for any $\lambda$ in a neighborhood of the origin
\begin{align}
\label{eq: fonction Psi lambda}
\Psi_\eps(\lambda)  = \sum_{n \geq 1}\frac{ \mu_\eps^{n-1} }{n!}
 \int_{\T^{dn}\times \R^{dn}} dZ_n \left( \prod_{i=1}^n M(v_i)
 \exp ( \lambda g_0(z_i) ) \right) \varphi (X_n)\, ,
\end{align}
denoting by~$\varphi$ the cumulants defined by 
\begin{align}
\label{eq: cumulant phi}
\varphi(X_n) =   \sum_{G \in \cC_n} \prod_{i,j \in G} (- 1_{|x_i -x_j| \leq \eps} )\, .
\end{align}

The Laplace transform of the fluctuation field can be related to the  modified partition function as follows
\begin{align}
\label{eq: fonction Psi lambda mu}
\log &\,   \bbE_\eps  \Big( \exp  \big( \lambda  \zeta^\eps_0 (g_0)  \big) \Big)
 =
\log  \bbE_\eps \Big( \exp  \big( \lambda \sqrt{\mu_\eps} \pi^\eps_0 (g_0)  \big) \Big) 
- \lambda \sqrt{\mu_\eps}  \;  \bbE_\eps \Big(    \pi^\eps_0 (g_0)   \Big)  \\
& = \mu_\eps \Psi_\eps \left(\frac{\lambda}{\sqrt{\mu_\eps} } \right)  
- \lambda \sqrt{\mu_\eps}  \; \bbE_\eps \Big(  \pi^\eps_0 (g_0)    \Big) 
- \log \cZ^ \eps   \nonumber \\
& = \sum_{n \geq 1}\frac{ \mu_\eps^n}{n!}
 \int_{\T^{dn}\times \R^{dn}} dZ_n \prod_{i=1}^n M(v_i)
\left( \prod_{i=1}^n  \exp \left( \frac{\lambda}{\sqrt{\mu_\eps} } g_0(z_i) \right) -  \sum_{i=1}^n \frac{\lambda}{\sqrt{\mu_\eps} }   g_0(z_i) -1 \right) 
\varphi (X_n),  \nonumber 
\end{align}
where we used that the decomposition \eqref{eq: fonction Psi lambda} applies as well to 
$$\log \cZ^ \eps = \mu_\eps \Psi_\eps (0)
\quad \text{and} \quad \bbE_\eps \Big(    \pi^\eps_0 (g_0)  \Big) = \partial_\lambda \Psi_\eps (0).$$
We are left to check that, uniformly   in $\eps$ small enough, $\lambda \mapsto \bbE_\eps  \Big( \exp  \big( \lambda  \zeta^\eps_0 (g_0)  \big) \Big)$ is an analytic function  in a neighborhood of 0. The estimate \eqref{eq: moment ordre 2,4} on the moments will then follow by taking derivatives with respect to $\lambda$. 
To derive the analyticity, each term of the series \eqref{eq: fonction Psi lambda mu} can be bounded as follows by a second order expansion of the exponential product
\begin{align*}
&\frac{ \mu_\eps^n}{n!} \int_{\T^{dn}\times \R^{dn}} dZ_n \prod_{i=1}^n M(v_i)
\left( \prod_{i=1}^n  \exp \left( \frac{\lambda}{\sqrt{\mu_\eps} } g_0(z_i) \right) -  \sum_{i=1}^n \frac{\lambda}{\sqrt{\mu_\eps} }   g_0(z_i) -1 \right) 
\varphi (X_n)\\ 
& \qquad \qquad \qquad  
\leq C^n \frac{\mu_\eps^{n-1}}{n!} \lambda^2 \int_{\T^{dn} } dX_n
\varphi (X_n)
\leq C^n  ( \mu_\eps \, \eps^d)^{n-1} \lambda^2,
\end{align*}
where the constant $C$ depends on $\| g_0 \|_\infty$ and is uniform in $\eps, \lambda$ small enough. To derive the last inequality, we used  that $\varphi$ in \eqref{eq: cumulant phi} satisfies the tree inequality \eqref{eq:treeineq} and that the total number of trees is  $n^{n-2}$ by Cayley's formula.
This shows that the series is absolutely convergent and completes the claim on its analyticity.
\end{proof}

\section{Geometric estimates}\label{geometric estimates}

In this section we prove Propositions~\ref{recollisionprop} and~\ref{vsingularityprop}.

\medskip

\noindent
{\it Proof of Proposition~{\rm\ref{recollisionprop}}.} $ $ 

\smallskip

\noindent
{\bf The case of a periodic non clustering collision.}

Denote by $q,q'$ the particles involved in the first non clustering collision, assumed here to be periodic. By definition, their first deflection (going backward in time from $\tau_{rec} $ to $t_{stop}$) involves both particles $q$ and $q'$. 

If the first deflection corresponds to the $j$-th clustering collision,  $\{q,q'\} = \{q_j, \bar q_j \}$, and  in addition to the condition 
$ \hat x_j\in B_{T_\prec, j}$ which encodes the clustering collision, we obtain the condition
\begin{equation}
\label{self-recollc}
\begin{aligned}
\eps \omega_{j}  + (v_q  -  v_{ q'})  (\tau_{\rm rec} - \tau_{j} ) = \eps \omega_{\rm rec} +\zeta \,  \hbox{ with } \,    \zeta\in \Z^d\,, \, \omega_{\rm rec} \in \bS^{d-1}\,,  \\
\hbox{ and } v_q - v_{ q'} = \bar v_q - \bar v_{q'} - 2(\bar v_q - \bar v_{q'}) \cdot \omega_{j}   \, \omega _{j}
\end{aligned}
\end{equation}
denoting by $\bar v_q, \bar v_{q'}$ the velocities before the clustering collision in the forward dynamics, and by~$\omega_{j}$ the impact parameter at the clustering collision.
We deduce from the first relation that~$v_q- v_{ q'}$ has to be in a small cone $K_\zeta$ of opening $\eps$, 
which implies by the second relation  that  $\omega_{j} $
has to be in a small cone $S_\zeta$ of opening $\eps$. 

 Using the  local  change of variables $\hat x_j \mapsto (\eps \omega_j,  \tau_j) $, it follows that
$$
\begin{aligned}
\int \indc _{\mbox{\tiny Periodic non clustering collision with parent $j$}} \indc_{B_{T_\prec, j}}  \, d \hat x_j & \leq C
 \eps^{d-1} \theta \sum_{\zeta}
\int \indc _{\omega_{j} \in  S_\zeta} \big( (
 \bar v_q - \bar v_{q'}) \cdot \omega_{j}
 \big)_+
d\omega_{j}
   \\&\leq C \eps^{2( d-1) } \left(\theta \bbV\right)^{d+1}  
\end{aligned}
$$
since there are at most~$(\theta \bbV)^{d}  $
 possibilities for the~$\zeta$'s.

\medskip
\noindent
 {\bf Non clustering collision.}

 Denote by $q,q'$ the particles involved in the first non clustering collision, assumed here to be non periodic.
Assume that the first deflection in the backward dynamics of $(q,q')$ is  the~$j$-th clustering collision   between $ q_j=q$ and $\bar q_j = c$ (with~$c \neq q'$) at time $\tau_j$ (which implies necessarily that  $j\geq 2$). Then  in addition to the condition
$ \hat x_j    \in B_{T_\prec, j}$ which encodes the clustering collision, we obtain the condition
\begin{equation}
\label{recolld}
\begin{aligned}
\big(x_q ( \tau_j  )  -  x_{q'} ( \tau_j ) \big)  + (v_q - \bar v_{ q'})  (\tau_{\rm rec} -  \tau_j) = \eps \omega_{\rm rec} +\zeta \, ,\\
\hbox{ and } v_q  = \bar v_q - (\bar v_q - \bar v_c) \cdot \omega_j  \,  \omega _j
\end{aligned}
\end{equation}
denoting by $\bar v_q, \bar v_c$ and $\bar v_{q'}$ the  velocities of $q$, $c$ and $q'$ at time $\tau_{j-1}^+$ (and therefore at time $\tau_j^-$).
Define 
$$\delta x :=  \frac1\eps ( x_{q '}( \tau_j  )  -  x_q ( \tau_j) +\zeta ) =: \delta x_\perp + (\bar v_{ q'} - \bar v_q ) \delta\tau_j\, ,
$$
where $\delta x_\perp$ is the component of $\delta x$ orthogonal to $(\bar v_{q'} - \bar v_q)$.
We also define the rescaled time  $$
 \delta \tau_{\rm rec}:= (\tau_{\rm rec} - \tau_j)/\eps\,.
$$
Note that, by definition 
$$ | (\bar v_{ q'} - \bar v_q ) \delta\tau_j| \leq |\delta x | \leq {C \over \eps}\,\cdotp$$

The first equation in (\ref{recolld}) restates
\begin{equation}
\label{cylinder}
v_q - \bar v_{ q'}  = {1\over \delta \tau_{\rm rec}} \Big( \omega_{\rm rec}  + \delta x_\perp + \delta \tau _j( \bar v_{ q'} - \bar v_q) \Big)\,.
\end{equation}

\medskip
\noindent
\underline {Case 1}~:  if $ | (\bar v_{ q'} - \bar v_q ) \delta\tau_j| \geq 2$, the triangular inequality implies
$$ {1\over 2 \delta \tau_{\rm rec}}  | (\bar v_{ q'} - \bar v_q ) \delta\tau_j| \leq | v_q - \bar v_{q'} |\,  ,$$
from which we deduce
$$  {1\over \delta \tau_{\rm rec}} \leq {2\bbV \over  | (\bar v_{ q'} - \bar v_q ) \delta\tau_j|}\,\cdotp$$ 
By (\ref{cylinder}), $v_q - \bar v_{ q'}$  belongs to a cylinder~$\cR $ of main axis~$\delta x_\perp +\bbR  (\bar v_q  -  \bar v_{ q'}) $ and 
 of width~${2\bbV \over  | (\bar v_{ q'} - \bar v_q ) \delta\tau_j|}.$

Then,  $v_q$ has to be both in the sphere of diameter~$[\bar v_q, \bar v_c]$ (by the second equation in (\ref{recolld})) and in the cylinder~$\bar v_{q'} + \cR$ (by (\ref{cylinder})). 
This imposes a strong constraint on the deflection angle $\omega_j$ in (\ref{recolld}), which has to belong to a union of at most two spherical caps.
 The maximal solid angle is obtained in the case when the cylinder is tangent to the sphere (see Figure \ref{fig:intersection}). It  is always less than $C_d\min (1, (\eta/R)^{(d-1)/2})$ denoting by $\eta$ the width of the cylinder, and by~$R$ the radius of the sphere.
 
 \begin{figure} [h] 
   \centering
  \includegraphics[width=11cm]{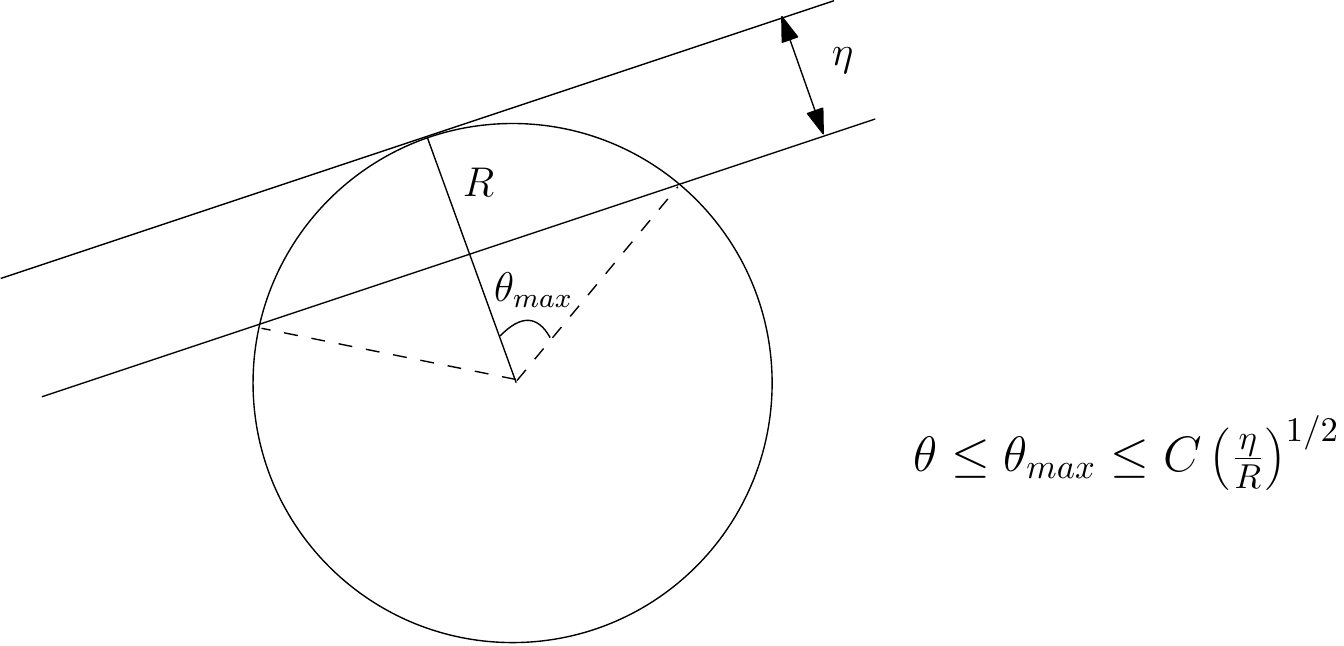} 
  \caption{Intersection of a cylinder and a sphere. The maximal solid angle is obtained in the case when the cylinder is tangent to the sphere. It  is always less than $C_d\min (1, (\eta/R)^{(d-1)/2})$.}
\label{fig:intersection}
\end{figure}

Thus $\omega_j   $ has to belong to a union of spherical caps $S_\zeta$, of solid angle  less than 
$$\int \indc _{\omega_j \in  S_\zeta} d\omega_j \leq  C \left( \frac{ \bbV }{|\delta \tau_j(\bar v_q -  \bar v_{ q'})| |\bar v_q - \bar v_c |}\right)^{(d-1)/2} \,.$$
Note that we can always replace the power $(d-1)/2$ by $1$ since we know that the left hand side is bounded by $|\bS^{d-1}|$.

\medskip
\noindent
\underline {Case 2}~: if $ | (\bar v_{ q'} - \bar v_q ) \delta\tau_j| < 2$, we have a strong constraint on $\tau_j$ and we do not need any additional constraint on $\omega_j$.

  Using the (local) change of variables $\hat x_j \mapsto (  \eps \delta\tau_j, \eps \omega_j  )$, it follows that
$$
\begin{aligned}\int \indc _{\mbox{\tiny Non clustering collision with parent $j$}} & \indc_{B_{T_\prec, j}} d \hat x_j  \\
& \leq {C\over \mu_\eps}  \sum_{\zeta}\int \indc_{ | (\bar v_{ q'} - \bar v_q ) \delta\tau_j| \geq 2}  \indc _{\omega_j \in  S_\zeta}
 |( \bar v_q - \bar v_c)\cdot \omega_j | d\omega_j  \, \eps d\delta\tau _j \\
 & \quad + {C\over \mu_\eps}  \sum_{\zeta}\int \indc_{ | (\bar v_{ q'} - \bar v_q ) \delta\tau_j| < 2} 
 |( \bar v_q - \bar v_c)\cdot \omega_j | d\omega_j  \, \eps d\delta\tau _j \\
 &    \leq {C\over \mu_\eps}  (\bbV  \theta)^d {\eps |\log \eps| \bbV \over  | \bar v_q - \bar v_{ q'}| } \; \cdotp
\end{aligned}
$$
This concludes the proof of Proposition~\ref{recollisionprop}.
 \qed

\begin{Rmk}
In dimension 2, the same strategy would lead to an estimate $O(\eps^{1/2})$, which is not strong enough. We would therefore have to be more careful when estimating the size of the intersection between the cylinder and the sphere, and track their relative positions by introducing an additional parent. For the sake of technical simplicity, we will not present a detailed proof here. \end{Rmk}

\bigskip

\noindent
{\it Proof of Proposition~{\rm\ref{vsingularityprop}}.} $ $ 

If the first deflection of $q$ corresponds to the $j$-th  clustering collision, in addition to the condition 
$ \hat x_j \in B_{T_\prec, j}$ which encodes the clustering collision, we obtain a condition on the velocity.

There are actually  two subcases~:
\begin{itemize}
\item $(q_j, \bar q_j) = (q,q')$ and $|v_q - v_{ q'} | = |\bar v_q - \bar v_{q'} |$ ; 
\item $(q_j, \bar q_j) = (q,c)$, $q' $ is not deflected at $\tau_j$,  and $ v_q = \bar v_q - (\bar v_q - \bar v_c) \cdot \omega_j \, \omega_j $ .
 \end{itemize}

\medskip
\noindent
\underline {Case 1}~: there holds
 \begin{equation*}
\int  {\indc_{B_{T_\prec, j}} \over | v_q - v_{q'} | }  d\hat x_j \leq {C\over \mu_\eps}  \int { \big |  (\bar v_q - \bar v_{q'}  ) \cdot  \omega_j  \big|\over |\bar v_q - \bar v_{q'} |}   d\omega_j  d\tau_j \leq  {C  \over \mu_\eps}( \delta \indc_{j=1} +\theta \indc_{j\neq 1})  \, \cdotp
\end{equation*}

\medskip
\noindent
\underline {Case 2}~: we have
 \begin{equation*}
\int {\indc_{B_{T_\prec, j}} \over | v_q - v_{q'} | } d\hat x_j \leq {C\over \mu_\eps}  \int {1 \over |\bar v_q - \bar v_{q'}- (\bar v_q - \bar v_c) \cdot \omega_j \, \omega_j  |} \big|(\bar v_q -  \bar v_c   ) \cdot  \omega_j  \big|  d\omega_j  d\tau_j \, .
\end{equation*}
Denoting $w := \bar v_q - \bar v_c$ and $u := \bar v_q - \bar v_{q'}$, we therefore have to study the integral
$$ \int {1\over |u- (w \cdot \omega) \, \omega |}   \big | w\cdot  \omega  \big| d\omega\,.$$
The denominator in the integrand vanishes when 
 $$
\omega _0= {u \over |u|} \, , \quad (u\cdot w) = |u|^2\,.$$
Consider an infinitesimal variation $\eta$ around $\omega_0 $. Since $\omega \in \bS^{d-1}$, $\eta$ is orthogonal to $\omega_0$. The first increment of  the denominator at $\omega _0$ is
$$ \left|  (w \cdot \eta) \omega_0+ (w\cdot \omega_0) \eta \right| \geq \left| (w\cdot \omega_0) \eta \right| \geq |u| |\eta|\,.$$
We therefore find that 
$$ { \big | w\cdot  \omega \big|\over |u- (w \cdot \omega )\, \omega |}  \leq  C{|u| \over |\eta|  |u|} \,\cdotp$$
Locally the measure $d\omega$ looks like $|\eta|^{d-2} d\eta$, from which we deduce that 
$$ \int {1 \over |u- (w \cdot \omega )\, \omega |}  \big | w\cdot  \omega \big| d\omega\leq  C \bbV $$
 since $d \geq 3$. 
Integrating with respect to $\tau_j$ (and for $j=1$ taking into account the constraint that $\tau_1 \in [t_{stop},  t_{stop} +\delta]$)  concludes the proof of Proposition~\ref{vsingularityprop}.
\qed

	\begin{Rmk}
 Proposition~{\rm\ref{vsingularityprop}}
can be easily extended to dimension 2, taking into account the logarithmic singularity~:
$$\int  \indc_{B_{T_\prec, j}}\min \left(1, {\eps \over | v_q - v_{q'} | } \right)  d\hat x_j \leq {C\bbV \eps |\log \eps| \over \mu_\eps}( \delta \indc_{j=1} +\theta \indc_{j\neq 1})\,.$$
\end{Rmk}

\end{document}